\documentclass[reqno]{amsart}

\usepackage{amsfonts}
\usepackage{amssymb}
\usepackage{amsmath}
\usepackage{enumitem}
\usepackage{xcolor}
\usepackage{todonotes}

\setenumerate{label={\rm (\alph{*})}}
\usepackage{bbm,bm,euscript,mathrsfs}
\usepackage{paper_diening}
\usepackage{graphicx}
\usepackage[top=1in, bottom=1.25in, left=1.10in, right=1.10in]{geometry}
\usepackage{hyperref}
\usepackage{cases, color,amsthm}
\usepackage{upref}
\hypersetup{linkcolor=blue, colorlinks=true,citecolor = red}

\allowdisplaybreaks

\numberwithin{equation}{section}

\usepackage{tikz-cd}
\usetikzlibrary{lindenmayersystems}
\usetikzlibrary{decorations.pathreplacing}

\DeclareMathOperator{\Bog}{Bog}

\DeclareMathOperator{\Div}{div}

\DeclareMathOperator{\diver}{div}
\newcommand{\R}{\mathbb R}
\newcommand{\N}{\mathbb N}
\newcommand{\dd}{\mathrm d}
\newcommand{\dt}{\,\mathrm{d} t}

\renewcommand{\bfB}{B}

\newcommand{\Testzeta}{{\mathscr{F}}^{\Div}_{\eta}}

\newcommand{\bx}{\mathbf{x}}
\newcommand{\by}{\mathbf{y}}
\newcommand{\bz}{\mathbf{z}}

\newcommand{\bu}{\mathbf{v}}

\newcommand{\bv}{\mathbf{v}}
\newcommand{\bn}{\mathbf{n}}
\newcommand{\be}{\mathbf{e}}

\newcommand{\dy}{\, \mathrm{d}\mathbf{y}}

\newcommand{\dx}{\, \mathrm{d} \mathbf{x}}

\newcommand{\divx}{\mathrm{div} }

\newcommand{\nabx}{\nabla }
\newcommand{\naby}{\nabla_{\mathbf{y}}}

\newcommand{\Delx}{\Delta }
\newcommand{\Dely}{\Delta_{\mathbf{y}}}

\newcommand{\ds}{\,\mathrm{d}\sigma}

\newcommand{\Oeta}{\Omega_\eta}

\newtheorem{theorem}{Theorem}[section]
\newtheorem{lemma}[theorem]{Lemma}
\newtheorem{proposition}[theorem]{Proposition}
\newtheorem{corollary}[theorem]{Corollary}
\newtheorem{remark}[theorem]{Remark}

\theoremstyle{definition}
\newtheorem{definition}[theorem]{Definition}

\begin{document}

\title[Serrin for FSI]{Ladyzhenskaya-Prodi-Serrin condition for fluid-structure interaction systems}

\author{Dominic Breit}
\address{Institute of Mathematics, TU Clausthal, Erzstra\ss e 1, 38678 Clausthal-Zellerfeld, Germany}
\email{dominic.breit@tu-clausthal.de}

\author{Prince Romeo Mensah}
\address{Institute of Mathematics, TU Clausthal, Erzstra\ss e 1, 38678 Clausthal-Zellerfeld, Germany}
\email{prince.romeo.mensah@tu-clausthal.de}

\author{Sebastian Schwarzacher}
\address{Department of Mathematical Analysis,
	Faculty of Mathematics and Physics,
	Charles University,
	Sokolovská 83,
	186 75 Praha 8, Czech Republic}
\address{and}
\address{Department of Mathematics, 
	Analysis and Partial Differential Equations, 
	Uppsala University, 
	L\"agerhyddsv\"agen 1,
	752 37 Uppsala, Sweden}
\email{schwarz@karlin.mff.cuni.cz}

\author{Pei Su}
\address{Department of Mathematical Analysis,
	Faculty of Mathematics and Physics,
	Charles University,
	Sokolovská 83,
	186 75 Praha 8, Czech Republic}
\email{peisu@karlin.mff.cuni.cz}



\date{\today}



\begin{abstract}

We consider the interaction of a viscous incompressible fluid with a flexible shell in three space dimensions. The fluid is described by the three-dimensional incompressible Navier--Stokes equations in a domain that is changing in  accordance with the motion of the structure. The displacement of the latter evolves along a visco-elastic shell equation. Both are coupled through kinematic boundary conditions and the balance of forces.

We prove a counterpart of the classical Ladyzhenskaya-Prodi-Serrin condition yielding conditional regularity and uniqueness of a solution. 

Our result is a consequence of the following three ingredients which might be of independent interest: {\bf (i)} the existence of local strong solutions, {\bf (ii)} an acceleration estimate (under the Serrin assumption) ultimately controlling the second-order energy norm, and {\bf (iii)} a weak-strong uniqueness theorem.
The first point, and to some extent, the last point were previously known for the case of elastic plates, which means that the relaxed state is flat. We extend these results to the case of visco-elastic shells, which means that more general reference geometries are considered such as cylinders or spheres. The second point, i.e. the acceleration estimate for three-dimensional fluids is new even in the case of plates.
\end{abstract}

\maketitle

\section{Introduction}
When three-dimensional Navier-Stokes equations are considered, only conditional smoothness and uniqueness of weak solutions are known for large data and time. The condition is that the fluid velocity satisfies some integrability \textit{beyond the natural energy estimate} that overcomes a certain scaling, namely\footnote{Here and later, we use $I=(0,T)$ as the time interval. Further notations can be found in the next section.}
\begin{align}
\label{eq:serrin}
\bu\in L^r(I;L^s(\Omega)), \qquad\tfrac{2}{r}+\tfrac{3}{s}= 1, \qquad 2\leq r<\infty. 
\end{align}
The above criterion is known as the
\textit{Ladyzhenskaya-Prodi-Serrin condition}, referring to the works by Prodi \cite{prodi1959} and Serrin \cite{serrin1962,serrin1963} on proving conditional uniqueness as well as that of Ladyzhenskaya \cite{lady} showing conditional regularity of solutions. Summarizing this means, if \eqref{eq:serrin} is satisfied the solution is {\bf a)} smooth, and {\bf b)} unique within all {\em weak-solutions satisfying an energy inequality}. The latter property is often referred to as {\em weak-strong-uniqueness}.

Many authors have since contributed to the generalization of this criterion~\cite{beale1984remark, beirao1995new, chen2006space, escauriaza, kozono2002critical, kozono2004bilinear, kozono2000limiting}. In particular, in recent years, seminal studies related to the borderline case $s=3$ indicate that the condition could potentially be sharp~\cite{ABC22,escauriaza,Sve14,JiaSve15}. As the {\em physical indications of non-uniqueness} are usually (necessarily) present in fluid-structure interaction problems, it seems worthwhile investigating how far the seminal work by Ladyzhenskaya, Prodi, Serrin and many others still holds true in this framework. {\em The aim of this paper is to advance this theory to the framework of elastic deformable shells interacting with the incompressible Navier-Stokes equation.} A second more practical motivation of our study is its potential application for numerical analysis. Indeed, the analysis here in particular shows that strong solutions are attractors, which is a first step towards convergence results. See Remark~\ref{rem:stab} for more details.

In the context of a fluid-structure interaction problem the domain of the fluid varies with time with respect to the evolution of the structure. Hence an estimate for the difference of two solutions cannot be directly obtained even when both solutions are smooth. This is already the case when a single rigid body is moving inside the fluid. For that regime, rather recently, the Ladyzhenskaya-Prodi-Serrin condition has been extended for the motion of a rigid ball immersed in a viscous incompressible fluid ~\cite{chemetov2019weak,maity2023uniqueness, muha2022regularity}. In that context falls also the uniqueness result for weak solutions in two dimensions~\cite{GlaSue15}.

The situation becomes even more dramatic, when flexible materials are considered that change the domain in an asymmetric fashion.

In this work we study {\em curved reference configurations} (see Figure~\ref{fig:2}). The most prominent reference geometries for shells are cylinders, that relate for example to the very relevant application of blood-flow or balls, relating for example to the motion of a balloon. But certainly many more complicated reference geometries may appear in applications. 
In short, we derive the following three novel results to be found in sections three, four and five that might each be of independent interest:
\begin{enumerate}
\item[Section 3] {\bf Local strong solutions.} We show the existence of a smooth solution for short times.
\item[Section 4] {\bf The acceleration estimate.} Here we show that as long as the fluid velocity satisfies \eqref{eq:serrin} and the displacement of the shell stays $C^1$ in space, the solutions satisfies some extra smoothness. This section relates to  the conditional smoothness of solutions {\bf a)}.
\item[Section 5] {\bf Weak-strong uniqueness.} Finally, in this section, it is shown that the constructed smooth solution is unique in the regime of weak solutions satisfying an energy estimate and possessing a bi-Lipschitz-in-space shell displacement; hence, conditional uniqueness is shown {\bf b)}.
\end{enumerate} 

The only additional assumption for a weak solution to be smooth and unique that we require on the shell displacement is that it is $C^1$ in space. As we will explain below, this is just an instant of regularity more than a weak solution enjoys.

The first point and, to some extent, the last point above were previously known for the case of elastic plates, which means that the relaxed state is flat. 
The latter one is also the first weak-strong uniqueness result in the context of fluid-structure interaction with flexible structure~\cite{schwarzacher2022weak}. The proof there relies heavily on the fact that the reference configuration is flat.
The second point above, the acceleration estimate for three-dimensional fluids, is new even in the case of plates.

\subsection{Analysis of fluid-structure interactions}
The results presented here strongly connect to previous works on fluid-structure interactions involving elastic structures interacting with an unsteady three-dimensional viscous incompressible fluid. Most results are on the existence theory. We refer to~\cite{sunny} for an overview of the setting considered in this paper and to~\cite{kaltenbacher} for various subjects in fluid-solid interactions. We may broadly classify these body of work into the construction of strong solutions and weak solutions for a viscous fluid interacting with an elastic structure.

For weak solutions,  a semi-group approach is used in \cite{barbu2007existence} in the construction of global-in-time weak solutions  for the interaction between a stationary elastic solid  immersed in a  viscous incompressible fluid, where the interaction happens at the boundary of the solid. The same authors then show in \cite{barbu2008smoothness} that for smooth enough data,  the weak solutions constructed in \cite{barbu2007existence} become smooth.
The existence of a weak solution is also shown in  \cite{boulakia2007existence} for a regularized three-dimensional elastic structure immersed in
an incompressible viscous fluid contained in a fixed bounded connected domain. These solutions exist
as long as deformations of the elastic solid are sufficiently small  and no collisions occurs between the structure and the boundary. However, large translations and rotations of the structure are accounted for. 
The authors in \cite{chambolle2005existence} used a Galerkin method to show the existence of a weak solution to the  three-dimensional  Navier--Stokes equations
coupled with a two-dimensional elastic plate model that is modified to include viscous effects. This weak solution exists as long as  the structure does not touch the fixed part of the fluid boundary. The viscous effects incorporated in the plate model is then removed in \cite{grandmont2008existence} by passing to the limit as the coefficient modelling the viscoelasticity tends to zero.
  The seminal work \cite{coutand2005motion} explores the motion of the linear Kirchhoff elastic solid material
 inside an incompressible viscous fluid. A topological fixed-point argument is used to construct a local-in-time weak solution which is then shown to be regular and unique. 
The authors in \cite{LeRu} study the interaction of an incompressible Newtonian fluid with a linearly elastic Koiter shell. Here, the fluid's boundary is described by the mid-section of the shell and the authors show the existence of weak solutions, without self-intersections of the shell, using an Aubin--Lions type argument. Eventually, an existence result for the fully nonlinear Koiter shell model has been proved in \cite{MuSc}.

When it comes to strong solutions, short time existence and uniqueness of solutions in Sobolev spaces are studied in \cite{cheng2007navier, CS} for  a viscous incompressible fluid interacting with a nonlinear thin elastic  shell. The shell equation for the former \cite{cheng2007navier} is modelled by the nonlinear Saint-Venant-Kirchhoff constitutive law, whereas that of the latter \cite{CS}  is modelled by the nonlinear Koiter shell model.  In \cite{coutand2006interaction}, however, the authors prove the existence of a unique local strong solution, without restriction on the size of the data, when the elastic structure is now governed by quasilinear elastodynamics. 
In \cite{ignatova2014well}, the elastic structure is modelled by a damped wave equation  with  additional boundary stabilization terms. For sufficiently small initial data, subject to said boundary stabilization terms,
global-in-time existence of strong solutions and exponential decay of the solutions are shown. 
The free boundary fluid-structure interaction problem consisting of a Navier--Stokes equation and a wave equation defined in two different but adjacent domains is studied in \cite{kukavica2012solutions}. A local strong solution is constructed under suitable compatibility conditions for the data.
Another local-in-time strong existence result is \cite{DRR} where the viscous Newtonian fluid is now interacting with an elastic structure modelled by a nonlinear damped shell equation.
Finally, a local strong solution is constructed for the motion of a linearly elastic Lam\'e solid moving in a viscous fluid in \cite{raymond2014fluid}. For the problem \eqref{1}--\eqref{interfaceCond} below, the only available local existence results deal with the case of a flat geometry, see \cite{CS} and \cite{DRR}, while the existence of global strong solutions in 2D is proved in \cite{GraHil} for flat geometry and in \cite{Br} for linear elastic shells in general geometries. A corresponding result for the 3D case is not yet known.

\subsection{The fluid-structure interaction problem}
We are interested in the interaction of an incompressible fluid with a flexible shell where the shell  reacts to the surface forces induced by the fluid and deforms the spatial reference domain $\Omega \subset \mathbb{R}^3$ to $\Omega_{\eta(t)}$ with respect to a coordinate transform $\bfvarphi_{\eta(t)}$ (see Figure \ref{fig:2} for the typical situation). The deformed domain $\Omega_{\eta}$ is defined in Subsection~\ref{ssec:geom} in a precise way. 
  We assume that the shell is visco-elastic. This means that besides the fluid forces, it is driven by its {\em elastic} properties and its {\em viscosity}. The reference model here is the {\em linearized Koiter shell model}, but also linearized versions of von Karman shells or pure bending shells are models that can be treated by the methods here. Following~\cite{LeRu,CanMuh13}, we find that the elastic part of the equation for the solid becomes $\alpha\Dely^2\eta+B \eta $, where $B$ is a linear second-order differential operator. Similarly, the part related to the viscosity of the shell becomes $\gamma\Dely^2\partial_t\eta + B'\partial_t\eta $, where $B'$ is another linear second-order differential operator. To simplify reading, we  take a form of the equation that contains only parts of the contributions of elasticity and viscosity, which are essential for the analysis to be performed. In particular, we reduce the elastic part to $\alpha\Dely^2\eta$ and the viscous part to $-\gamma\Dely\partial_t\eta$. We observe that the reduction is with no loss of generality, which certainly would not be the case if {\em non-linear} Koiter shell models were considered as in~\cite{breit2021incompressible,CanMuh13,MuSc}. 

 Accordingly, the shell function $\eta:(t, \by)\in I \times \omega \mapsto   \eta(t,\by)\in \mathbb{R}$ with $I=(0,T)$ for some $T>0$ solves
 \begin{equation}\label{1}
\left\{\begin{aligned}
& \varrho_s\partial_t^2\eta -\gamma\partial_t\Dely \eta + \alpha\Dely^2\eta=g-\bn^\intercal\bm{\tau}\circ\bm{\varphi}_\eta\bn_\eta
 \vert \mathrm{det}(\naby \bm{\varphi}_{\eta})\vert
&\text{ for all }  (t,\by)\in I\times\omega
 ,\\
&\eta(0,\by)=\eta_0(\by), \quad (\partial_t\eta)(0, \by)=\eta_*(\by)
&\text{ for all } \by\in\omega.
 	\end{aligned}\right.
 \end{equation}
Here, $\omega \subset \mathbb{R}^2$ is such that there is $\bfvarphi_\eta :\omega\to \partial \Omega_\eta$ that parametrizes  the boundary of the reference domain $\Omega$. The parameters $\varrho_s,\gamma$ and $\alpha$ are positive constants and the function $g:(t, \by)\in I \times \omega \mapsto  g(t,\by)\in   \mathbb{R}$ is a given forcing term. The vectors $\bn$ and $\bfn_\eta$ are the normal vectors  of the reference boundary and of the deformed boundary, respectively, whereas $\bftau$ denotes the Cauchy stress of the fluid given by Newton's rheological law, that is
$\bftau=\mu\big(\nabx\bu+(\nabx\bu)^\intercal\big)-\pi\mathbb I_{3\times 3}$. The positive constant $\mu$ represents the viscosity coefficient. Also, $\mathbf{v}:(t, \mathbf{x})\in I \times \Oeta \mapsto  \mathbf{v}(t, \mathbf{x}) \in \mathbb{R}^3$, the velocity field and $\pi:(t, \mathbf{x})\in I \times \Oeta \mapsto  \pi(t, \mathbf{x}) \in \mathbb{R}$, the pressure function are the unknown functions for the fluid whose motion is governed by the Navier--Stokes equations
 \begin{equation}\label{2}
\left\{\begin{aligned}
 &\varrho_f\big(\partial_t \bu  + (\mathbf{v}\cdot \nabx)\mathbf{v} \big)
 = 
 \mu\Delx \bu -\nabx\pi+ \bff &\text{ for all }(t,\bx)\in I\times\Omega_\eta,\\
 &\Div \bu=0&\text{ for all }(t,\bx)\in I \times\Omega_\eta,\\
 &\bu(0,\bx)=\bu_0(\bx) &\text{ for all } \bx\in \Omega_{\eta_0},
 \end{aligned}\right.
 \end{equation}
where $\varrho_f$ is a positive constant representing the density of the fluid and the function $\bff:(t, \mathbf{x})\in I \times \Oeta \mapsto  \bff(t, \mathbf{x}) \in \mathbb{R}^3$ is a given volume force. The equations \eqref{1} and \eqref{2} are coupled through the kinematic boundary condition
\begin{align}
\label{interfaceCond}
\bu\circ \bfvarphi_\eta=\partial_t\eta\bfn \quad\text{ for all } (t,\by)\in I\times \omega.
\end{align}

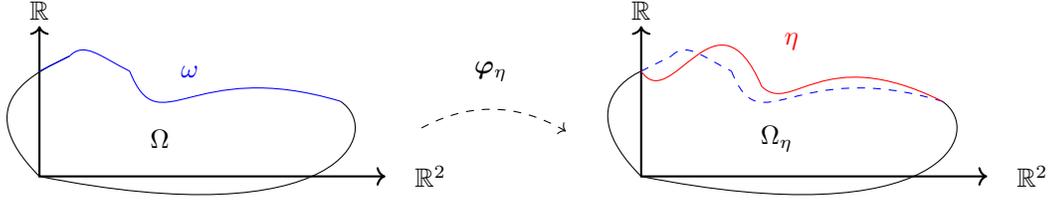
\begin{figure}
\begin{center}
\begin{tikzpicture}[scale=2]
  \begin{scope}     \draw [thick, <->] (0.5,0.5) -- (0.5,-0.5) -- (2.8,-0.5);
        \node at (0.5,0.6) {$\R$};
        \node at (1.3,-0.25) {$\Omega$};
          \node [blue] at (1.5,0.2) {$\omega$};
       \node at (3.1,-0.5) {$\R^2$};
       \draw (0.5,-0.5) .. controls (0,0) and (0.5,0.2) .. (0.7,0.3);
     \draw [blue] (0.5,0.2) .. controls (0.7,0.3) and (0.7,0.3) .. (0.7,0.3);
        \draw [blue] (0.7,0.3) .. controls (0.8,0.4) and (0.9,0.3) .. (1.1,0.2);
       \draw [blue] (1.1,0.2) .. controls (1.3,-0.3) and (1.5,0.3) .. (2.5,0);
       \draw (2.5,0) .. controls (2.8,-0.2) and (2.5,-0.9) .. (0.5,-0.5);
                   \node at (3.5,0.2) {$\bfvarphi_\eta$};
        \draw [thick, <->] (4.5,0.5) -- (4.5,-0.5) -- (6.8,-0.5);
        \node at (4.5,0.6) {$\R$};
       \node at (7.1,-0.5) {$\R^2$};
                     \draw (4.5,-0.5) .. controls (4,0) and (4.5,0.2) .. (4.5,0.2);
       \draw (6.5,0) .. controls (6.8,-0.2) and (6.5,-0.9) .. (4.5,-0.5);
        \draw [blue,dashed] (4.5,0.2) .. controls (4.7,0.3) and (4.7,0.3) .. (4.7,0.3);
        \draw [blue,dashed] (4.7,0.3) .. controls (4.8,0.4) and (4.9,0.3) .. (5.1,0.2);
       \draw [blue,dashed] (5.1,0.2) .. controls (5.3,-0.3) and (5.5,0.3) .. (6.5,0);
                \draw [red] (4.5,0.2) .. controls (4.7,-0.1) and (5,0.8) .. (5.3,0.1);
                \draw [red] (5.3,0.1) .. controls (5.5,-0.1) and  (5.7,0.4) .. (6.5,0);
                        \node at (5.4,-0.25) {$\Omega_\eta$};
                         \node [red] at (5.5,0.4) {$\eta$};
          \draw [<-,dashed] (4,-0.2) to [out=150,in=30] (3,-0.2);
  \end{scope}
\end{tikzpicture}
\caption{Domain transformation in the general set-up in 3D.}\label{fig:2}
\end{center}
\end{figure}

\subsection{Main result}
The main motivation for the present work is to prove an analog of the results by Serrin, Prodi and Ladyzhenskaya for the fluid-structure interaction problem \eqref{1}--\eqref{interfaceCond}.
The here presented result is a summary of Theorem \ref{thm:main}, where the statement can be found in its full extent. 

A weak solutions $(\eta,\bu)$ to \eqref{1}--\eqref{interfaceCond} can be constructed to satisfy the energy inequality and thus
\begin{align}
\sup_{I}\|\partial_t\eta\|_{L^2_\by}^2+\sup_{I}\|\Dely\eta\|_{L^2_\by}^2+\int_{I}\|\partial_t\naby\eta\|_{L^2_\by}^2\dt<\infty.\label{eq:apriorieta0}
\end{align}
We speak about a strong solution if all quantities in \eqref{1} and \eqref{2} are $L^2$-functions in space-time. The precise definitions can be found in Definitions \ref{def:weakSolution} and \ref{def:strongSolution}. 
\begin{theorem}[Shells]\label{thm:mainsimple}
Let $(\bu,\eta)$ be a weak solution to \eqref{1}--\eqref{interfaceCond}. Suppose that we
have
\begin{align}\label{eq:regu'0}
\bu&\in L^r(I;L^s(\Omega_\eta)),\quad \tfrac{2}{r}+\tfrac{3}{s}\leq1,\\
\eta&\in L^\infty(I;C^{1}(\omega)).\label{eq:regeta''0}
\end{align}
Then $(\bu,\eta)$ is a strong solution to \eqref{1}--\eqref{interfaceCond}.
Moreover, $(\bu,\eta)$  is unique in the class of weak solutions satisfying the energy inequality and are in $L^\infty(I;C^{0,1}(\omega))$.
\end{theorem}
We emphasis that the only additional assumption for the structure is given by $L^\infty(I;C^{1}(\omega))$ for conditional smoothness or $L^\infty(I;C^{0,1}(\omega))$ for the uniqueness class of the strong solution. Note that this is only an instant of regularity more than a weak solution enjoys as it belongs to $L^\infty(I;W^{2,2}(\omega))$, see \eqref{eq:apriorieta0}. Further note that the spaces $W^{2,2}(\omega),C^{0,1}(\omega)$ and $C^1(\omega)$ even have the same index in 2D, but the embedding $W^{2,2}(\omega)\hookrightarrow C^{0,1}(\omega)$ just fails. This extra assumption is, however, essential for the approach presented here. Hence it remains {\em a challenging open problem of whether the uniqueness regime (of a strong solution) can be extended to all energy preserving weak-solutions in the case of a curved reference geometry}. In contrast, in case the reference geometry is flat (the plate case), a direct approach for weak-strong uniqueness is available for which the $C^{0,1}(\omega)$ assumption is not necessary~\cite{schwarzacher2022weak}. As the theory presented here is in particular valid for plates, we have the following corollary.
\begin{corollary}[Plates]\label{cor:mainsimple}
Let $(\bu,\eta)$ be a weak solution of \eqref{1}--\eqref{interfaceCond} with {\em flat reference geometry}. Suppose that we
have
\begin{align*}
\bu&\in L^r(I;L^s(\Omega_\eta)),\quad \tfrac{2}{r}+\tfrac{3}{s}\leq1,\\
\eta&\in L^\infty(I;C^{1}(\omega)).
\end{align*}
Then $(\bu,\eta)$ is a strong solution of \eqref{1}--\eqref{interfaceCond}.
Moreover, $(\bu,\eta)$ is unique in the class of weak solutions satisfying an energy inequality.\footnote{Strictly speaking the weak-strong uniqueness result in ~\cite{schwarzacher2022weak} considers elastic plates. However, the presence of dissipation in the structure equation does not change the argument at all.}
\end{corollary}

\begin{remark}[Stability and convergence of numerical schemes]
\label{rem:stab}
{\rm 
Related to the weak-strong uniqueness results is a stability estimate (see Section \ref{sec:weakStrong}). It comes naturally, as the difference of two solutions is estimated. Hence a further result of this work is that any solution satisfying \eqref{eq:regu'0} and \eqref{eq:regeta''0} are actually attractors in the respective uniqueness class.

Stability are of particular importance for numeric applications. Indeed they form the first step in order to show that the difference between a discrete solution and the continuous solution decreases (with a rate), provided that the continuous solution is unique, as was shown for plates in 2D~\cite{SchSheTum23,schwarzacher2022weak}.
}
\end{remark}


The formal proof for the regularity of Theorem \ref{thm:mainsimple} consists in proving an acceleration estimate which combines and extends the results in~\cite{Br, GraHil}. In
order to appreciate the moving boundary, the correct test function for the momentum equation is -- roughly speaking -- the material derivative $\partial_t\bv+\bfv\cdot\nabla\bv$ combined with the test-function $\partial_t^2\eta$ for the structure equation. A key tool is eventually to estimate $\nabla^2\bv$ (as well as $\nabla\pi$) by means of $\partial_t\bv+\bfv\cdot\nabla\bv$. This can be done by means of a steady Stokes theory for irregular domains proved in \cite{Br} (applied to the domain $\Omega_{\eta(t)}$ for fixed $t$). It strongly requires a boundary with a small local Lipschitz constant and thus \eqref{eq:regeta''0} is essentially needed here. Otherwise, such a regularity estimate is not known and probably not even expected.
Furthermore, in order to avoid the appearance of the pressure function, an extension operator from $\omega$ to $\Omega_\eta$ has been used in \cite{GraHil} which is at the same time solenoidal and satisfies a homogeneous Neumann-type boundary condition. The construction of the latter is only possible for a flat reference domain. Hence, we introduce the pressure function and work with a more common extension operator which does not preserve solenoidability.
The advantage of the latter is that it has the usual regularization property (it ``gains'' the differentiability which is lost by the trace theorem, see Section \ref{sec:ext}) different from the solenoidal extensions used in \cite{GraHil}, \cite{LeRu} and \cite{MuSc}.
A major difference between the 2D and 3D cases is that one has to use the full strength of this operator to compensate for the worsening embeddings.

In order to make this argument rigorous, we work with a strong solution to \eqref{1}--\eqref{interfaceCond}.
Thus, we prove the existence of a local strong solution in Theorem
\ref{thm:fluidStructureWithoutFK}. 
With  a strong solution at hand, we can justify the estimates mentioned above. To close the argument (see the proof of Theorem \ref{thm:main}), we have to compare weak and strong solutions by means of a weak-strong uniqueness result.
The difficulty of the latter is that one needs to compare two velocity fields which are a priori defined on different (time-changing) domains. Nevertheless,
such a result
has been established very recently in \cite{schwarzacher2022weak} for linear elastic plates. The key idea is to transform the strong solution into the domain of the weak solution and then estimate their difference. When doing so, the strong solution loses its solenoidal character which must be corrected to avoid the appearance of the pressure function in the weak formulation. In the case of a flat geometry as in \cite{schwarzacher2022weak} this can be done by an explicit construction, but our situation is more complicated. We thus work with a Bogovkij-operator for moving domains~\cite{KamSchSpe20}. 
Its properties crucially hinge on the spatial Lipschitz continuity of the moving boundary and thus require
that the weak solution is $L^\infty(C^{0,1})$. For details on the criticality of Lipschitz regularity see~\cite{SaaSch21}, where estimates for the Bogovkij-operator in rough time-dependent domains are studied.

\subsection{Organization of the paper}
We introduce in Section \ref{sec:prelim}, some notations, definitions and the functional analytic framework. In particular, we give the definition of the notion of a weak and a strong solution for the system \eqref{1}--\eqref{interfaceCond}. We then construct in Section \ref{sec:loc}, the local strong solution by linearizing the system and employing the Banach fixed-point argument. Section \ref{sec:reg} is devoted to proposing the conditions of Serrin type to obtain the acceleration estimate. In Section \ref{sec:weakStrong}, we focus on showing the weak-strong uniqueness result. Finally, we give in Section \ref{summary} a short summary to formulate the main result by collecting the key elements of the previous sections.

\section{Preliminaries}
\label{sec:prelim}
\subsection{Conventions}
For simplicity, we set all physical constants in \eqref{1}--\eqref{interfaceCond} to 1. The analysis is not affected as long as they are strictly positive.
For two non-negative quantities $f$ and $g$, we write $f\lesssim g$  if there is a $c>0$ such that $f\leq\,c g$. Here $c$ is a generic constant which does not depend on the crucial quantities. If necessary, we specify particular dependencies. We write $f\approx g$ if both $f\lesssim g$ and $g\lesssim f$ hold. In the notation for function spaces (see next subsection), we do not distinguish between scalar- and vector-valued functions. However, vector-valued functions will usually be denoted in bold case.
For simplicity, we supplement \eqref{1} with periodic boundary conditions and identify $\omega$ (which represents the complete boundary of $\Omega$) with  $(0,1)^2$. We consider periodic function spaces for zero-average functions.
It is only a technical matter to consider \eqref{1} on a nontrivial subset of $\partial\Omega$ together with zero boundary conditions for $\eta$ and $\naby\eta$ instead of considering \eqref{1} on $(0,1)^2$, see e.g. \cite{LeRu} or \cite{BrSc} for the corresponding geometrical set-up.
We shorten the time interval $(0,T)$ by $I$.

\subsection{Classical function spaces}
Let $\mathcal O\subset\R^3$ be open.
The function spaces of continuous or $\alpha$-H\"older-continuous functions, $\alpha\in(0,1)$,
 are denoted by $C(\overline{\mathcal O})$ or $C^{0,\alpha}(\overline{\mathcal O})$ respectively, where $\overline{\mathcal O}$ is the closure of $\mathcal O$. Similarly, we write $C^1(\overline{\mathcal O})$ and $C^{1,\alpha}(\overline{\mathcal O})$.
We denote  by $L^p(\mathcal O)$ and $W^{k,p}(\mathcal O)$ for $p\in[1,\infty]$ and $k\in\mathbb N$, the usual Lebesgue and Sobolev spaces over $\mathcal O$. 
For a bounded domain $\mathcal O$,  the notation $(f)_{\mathcal O}:=\dashint_{\mathcal O}f\dx:=\mathcal L^3(\mathcal O)^{-1}\int_{\mathcal O}f\dx$ represents the mean or average value of $f\in L^p(\mathcal O)$.
 We denote by $W^{k,p}_0(\mathcal O)$, the closure of the smooth and compactly supported functions in $W^{k,p}(\mathcal O)$. If $\partial\mathcal O$ is regular enough, this coincides with the functions vanishing $\mathcal H^{2}$ -a.e. on $\partial\mathcal O$. 
 We also denote by $W^{-k,p}(\mathcal O)$ the dual of $W^{k,p}_0(\mathcal O)$.
  Finally, we consider subspaces
$W^{1,p}_{\Div}(\mathcal O)$ and $W^{1,p}_{0,\Div}(\mathcal O)$ of divergence-free vector fields which are defined accordingly. The space $L^p_{\Div}(\mathcal O)$ is defined as the closure of the set of smooth and compactly supported solenoidal functions in $L^p(\mathcal O).$ We will use the shorthand notations $L^p_\bx$ and $W^{k,p}_\bx$ in the case of $3$-dimensional domains (typically spaces defined over $\Omega\subset\R^3$ or $\Omega_\eta\subset\R^3$) and   
$L^p_\by$ and $W^{k,p}_\by$ for $2$- dimensional sets (typcially spaces of periodic functions defined over $\omega\subset\R^{2}$). 
For any pair of separable Banach spaces $(X,\|\cdot\|_X)$ and $(Y,\|\cdot\|_Y)$ with $X\subset Y$, we write $X\hookrightarrow Y$ if $X$ is continuously embedded in $Y$, that is $\Vert\cdot\Vert_Y\lesssim \Vert \cdot\Vert_X$.
Since we only consider functions on $\omega$ with periodic boundary conditions and zero mean values, we have the following equivalences
\begin{align*}
\|\cdot\|_{W^{1,2}_\by}\approx \|\nabla_\by\cdot\|_{L^2_\by},\quad \|\cdot\|_{W^{2,2}_\by}\approx \|\Dely\cdot\|_{L^2_\by},\quad \|\cdot\|_{W^{4,2}_\by}\approx \|\Dely^2\cdot\|_{L^2_\by}.
\end{align*}
For a separable Banach space $(X,\|\cdot\|_X)$, we denote by $L^p(I;X)$, the set of (Bochner-) measurable functions $u:I\rightarrow X$ such that the mapping $t\mapsto \|u(t)\|_{X}$ belongs to $L^p(I)$. 
The set $C(\overline{I};X)$ denotes the space of functions $u:\overline{I}\rightarrow X$ which are continuous with respect to the norm topology on $(X,\|\cdot\|_X)$. For $\alpha\in(0,1]$ we write
$C^{0,\alpha}(\overline{I};X)$ for the space of H\"older-continuous functions with values in $X$. The space $W^{1,p}(I;X)$ consists of those functions from $L^p(I;X)$ for which the distributional time derivative belongs to $L^p(I;X)$ as well. The space $W^{k,p}(I;X)$ is defined accordingly.
We use the shorthand $L^p_tX$ for $L^p(I;X)$. For instance, we write $L^p_tW^{1,p}_\bx$ for $L^p(I;W^{1,p}(\mathcal O))$. Similarly, $W^{k,p}_tX$ stands for $W^{k,p}(I;X)$.

\subsection{Fractional differentiability and Sobolev mulitpliers}
For $p\in[1,\infty)$, the fractional Sobolev space (Sobolev-Slobodeckij space) with differentiability $s>0$ with $s\notin\mathbb N$ will be denoted by $W^{s,p}(\mathcal O)$. For $s>0$, we write $s=\lfloor s\rfloor+\lbrace s\rbrace$ with $\lfloor s\rfloor\in\N_0$ and $\lbrace s\rbrace\in(0,1)$.
 We denote by $W^{s,p}_0(\mathcal O)$, the closure of the smooth and compactly supported functions in $W^{s,p}(\mathcal O)$. For $s>\frac{1}{p}$ this coincides with the functions vanishing $\mathcal H^{n-1}$ -a.e. on $\partial\mathcal O$ provided that $\partial\mathcal O$ is regular enough. We also denote by $W^{-s,p'}(\mathcal O)$, for $s>0$ and $p,p'\in[1,\infty)$, with $\frac{1}{p}+\frac{1}{p'}=1$, the dual of $W^{s,p}_0(\mathcal O)$. Similar to the case of unbroken differentiabilities above, we use the shorthand notations $W^{s,p}_\bx$  and $W^{s,p}_\by$. 
We will denote by $\bfB^s_{p,q}(\R^n)$, the standard Besov spaces on $\R^n$ with differentiability $s>0$, integrability $p\in[1,\infty]$ and fine index $q\in[1,\infty]$. They can be defined (for instance) via Littlewood-Paley decomposition leading to the norm $\|\cdot\|_{\bfB^s_{p,q}(\R^n)}$. 
 We refer to \cite{RuSi} and \cite{Tr,Tr2} for an extensive description. 
For a bounded domain $\mathcal O\subset\R^n$, the Besov spaces $\bfB^s_{p,q}(\mathcal O)$ are defined as the restriction of functions from $\bfB^s_{p,q}(\R^n)$, that is
 \begin{align*}
 \bfB^s_{p,q}(\mathcal O)&:=\{f|_{\mathcal O}:\,f\in \bfB^s_{p,q}(\R^n)\},\\
 \|g\|_{\bfB^s_{p,q}(\mathcal O)}&:=\inf\{ \|f\|_{\bfB^s_{p,q}(\R^n)}:\,f|_{\mathcal O}=g\}.
 \end{align*}
 If $s\notin\mathbb N$ and $p\in(1,\infty)$ we have $\bfB^s_{p,p}(\mathcal O)=W^{s,p}(\mathcal O)$.
 
In accordance with \cite[Chapter 14]{MaSh}, the Sobolev multiplier norm  is given by
\begin{align}\label{eq:SoMo}
\|\varphi\|_{\mathcal M(W^{s,p}(\mathcal O))}:=\sup_{\mathbf{u} :\,\|\mathbf{u}\|_{W^{s-1,p}(\mathcal O)}=1}\|\nabla\varphi\cdot\mathbf{u}\|_{W^{s-1,p}(\mathcal O)},
\end{align}
where $p\in[1,\infty]$ and $s\geq1$.
The space $\mathcal M(W^{s,p}(\mathcal O))$ of Sobolev multipliers is defined as those objects for which the $\mathcal M(W^{s,p}(\mathcal O))$-norm is finite. By mathematical induction with respect to $s$, one can prove for Lipschitz-continuous functions $\varphi$ that membership to $\mathcal M(W^{s,p}(\mathcal O))$,  in the sense of \eqref{eq:SoMo}, implies that
\begin{align}\label{eq:SoMo'}
\sup_{w:\,\|w\|_{W^{s,p}(\mathcal O)}=1}\|\varphi \,w\|_{W^{s,p}(\mathcal O)}<\infty.
\end{align}
The quantity \eqref{eq:SoMo'} also serves as customary definition of the Sobolev multiplier norm in the literature but \eqref{eq:SoMo} is more suitable for our purposes.
Note that in our applications, we always assume that the functions in question are Lipschitz continuous so that the implication above holds true.

Let us finally collect some useful properties of Sobolev multipliers.
By \cite[Corollary 14.6.2]{MaSh} we have
\begin{align}\label{eq:MSa}
\|\phi\|_{\mathcal M(W^{s,p}(\R^{n}))}\lesssim\|\nabla\phi\|_{L^{\infty}(\R^n)},
\end{align}
provided that one of the following conditions holds:
\begin{itemize}
\item $p(s-1)<n$ and $\phi\in \bfB^{s}_{\varrho,p}(\R^n)$ with $\varrho\in\big[\frac{n}{s-1},\infty\big]$;
\item $p(s-1)=n$ and $\phi\in\bfB^{s}_{\varrho,p}(\R^n)$ with $\varrho\in(p,\infty]$.
\end{itemize}
Note that the hidden constant in \eqref{eq:MSa} depends on the $\bfB^{s}_{\varrho,p}(\R^n)$-norm of $\phi$.
By \cite[Corollary 4.3.8]{MaSh}, it holds
\begin{align}\label{eq:MSb}
\|\phi\|_{\mathcal M(W^{s,p}(\R^n))}\approx
\|\nabla\phi\|_{W^{s-1,p}(\R^n)}, 
\end{align}
for $p(s-1)>n$. 
 Finally, we note the following rule about the composition with Sobolev multipliers which is a consequence of \cite[Lemma 9.4.1]{MaSh}. For open sets $\mathcal O_1,\mathcal O_2\subset\R^n$, $u\in W^{s,p}(\mathcal O_2)$ and a Lipschitz continuous function $\bfphi:\mathcal O_1\rightarrow\mathcal O_2$ with Lipschitz continuous inverse and $\bfphi\in \mathcal M(W^{s,p}(\mathcal O_1))$ we have
\begin{align}\label{lem:9.4.1}
\|u\circ\bfphi\|_{W^{s,p}(\mathcal O_1)}\lesssim \|u\|_{W^{s,p}(\mathcal O_2)}
\end{align}
with constant depending on $\bfphi$. Using Lipschitz continuity
of $\bfphi$ and $\bfphi^{-1}$, estimate \eqref{lem:9.4.1} is obvious for $s\in(0,1]$. The general case can be proved by mathematical induction with respect to $s$.

\subsection{Function spaces on variable domains}
\label{ssec:geom}
 The spatial domain $\Omega$ is assumed to be an open bounded subset of $\mathbb{R}^3$ with smooth boundary $\partial\Omega$ and an outer unit normal ${\bfn}$. We assume that
 $\partial\Omega$ can be parametrised by an injective mapping ${\bfvarphi}\in C^k(\omega;\R^3)$ for some sufficiently large $k\in\N$. We suppose for all points $\by=(y_1,y_2)\in \omega$ that the pair of vectors  
$\partial_i {\bfvarphi}(\by)$, $i=1,2,$ are linearly independent.
 For a point $\bx$ in the neighborhood
of $\partial\Omega$, we define the functions $\by$ and $s$ by  
\begin{align*}
 \by(\bx)=\arg\min_{\by\in\omega}|\bx-\bfvarphi(\by)|,\quad s(\bx)=(\bx-\by(\bx))\cdot\bfn(\by(\bx)).
 \end{align*}
Moreover, we define the projection $\bfp(\bx)=\bfvarphi(\by(\bx))$. We define $L>0$ to be the largest number such that $s,\by$ and $\bfp$ are well-defined on $S_L$, where
\begin{align}\label{eq:boundary1}
S_L=\{\bx\in\R^n:\,\mathrm{dist}(\bx,\partial\Omega)<L\}.
\end{align}
Due to the smoothness of $\partial\Omega$ for $L$ small enough we have $\abs{s(\bx)}=\min_{\by\in\omega}|\bx-\bfvarphi(\by)|$ for all $\bx\in S_L$. This implies that $S_L=\{s\bfn(\by)+\by:(s,\by)\in (-L,L)\times \omega\}$.
For a given function $\eta : I \times \omega \rightarrow\R$ we parametrise the deformed boundary by
\begin{align*}
{\bfvarphi}_\eta(t,\by)={\bfvarphi}(\by) + \eta(t,\by){\bfn}(\by), \quad \,\by \in \omega,\,t\in I.
\end{align*}
By possibly decreasing $L$, one easily deduces from this formula that $\Omega_{\eta}$ does not degenerate, that is
\begin{equation}\label{eq:1705}
\begin{aligned}
\partial_1\bfvarphi_\eta\times\partial_2\bfvarphi_\eta(t,\by)\neq0,\quad
 \bfn(\by)\cdot\bfn_{\eta(t)}(\by)&>0,\quad \,\by \in \omega,\,t\in I,
 \end{aligned}
\end{equation}
provided that $\sup_t\|\eta\|_{W^{1,\infty}_{\by}}<L$. Here $\bfn_{\eta(t)}$
is the normal of the domain $\Omega_{\eta(t)}$
 defined through
\begin{equation}\label{eq:2612}
\partial\Omega_{\eta(t)}=\set{{\bfvarphi}(\by) + \eta(t,\by){\bfn}(\by):\by\in \omega}.
\end{equation}
With the abuse of notation we define deformed space-time cylinder $I\times\Omega_\eta=\bigcup_{t\in I}\set{t}\times\Omega_{\eta(t)}\subset\R^{4}$.
The corresponding function spaces for variable domains are defined as follows.
\begin{definition}{(Function spaces)}
For $I=(0,T)$, $T>0$, and $\eta\in C(\overline{I}\times\omega)$ with $\|\eta\|_{L^\infty(I\times\omega)}< L$ we define for $1\leq p,r\leq\infty$
\begin{align*}
L^p(I;L^r(\Omega_\eta))&:=\big\{v\in L^1(I\times\Omega_\eta):\,\,v(t,\cdot)\in L^r(\Omega_{\eta(t)})\,\,\text{for a.e. }t,\,\,\|v(t,\cdot)\|_{L^r(\Omega_{\eta(t)})}\in L^p(I)\big\},\\
L^p(I;W^{1,r}(\Omega_\eta))&:=\big\{v\in L^p(I;L^r(\Omega_\eta)):\,\,\nabla v\in L^p(I;L^r(\Omega_\eta))\big\}.
\end{align*}
\end{definition}
\noindent 
In order to establish a relationship between the 
fixed domain and the time-dependent domain, we introduce the Hanzawa transform $\bm{\Psi}_\eta : \Omega \rightarrow\Omega_\eta$ defined by
\begin{equation}
\label{map}
\bfPsi_\eta(\bx)
=
 \left\{
  \begin{array}{lr}
    \mathbf{p}(\bx)+\big(s(\bx)+\eta(\by(\bx))\phi(s(\bx))\big)\bn(\by(\bx)) &\text{if dist}(\bx,\partial\Omega)<L,\\
    \bx &\text{elsewhere}.
  \end{array}
\right.
\end{equation}
for any $\eta:\omega\rightarrow (-L,L)$. Here $\phi\in C^\infty(\mathbb R)$ is such that 
$\phi\equiv 0$ in a neighborhood of $-L$ and $\phi\equiv 1$ in a neighborhood of $0$. The other variables  $\mathbf{p}$, $s$ and $\bn$ are as defined earlier in this Section \ref{ssec:geom}.
Due to the size of $L$, we find that $\bfPsi_\eta$ is a homomorphism such that $\bfPsi_\eta|_{\Omega\setminus S_L}$ is the identity. Furthermore, $\bm{\Psi}_\eta$ together with its inverse\footnote{It exists provided that we choose $\phi$ such that $|\phi'|<L/\alpha$.} 
 $\bm{\Psi}_\eta^{-1} : \Omega_\eta \rightarrow\Omega$  possesses the following properties, see \cite{Br} for details. If for some $\alpha,R>0$, we assume that
\begin{align*}
\Vert\eta\Vert_{L^\infty_\by}
+
\Vert\zeta\Vert_{L^\infty_\by}
< \alpha <L \qquad\text{and}\qquad
\Vert\naby\eta\Vert_{L^\infty_\by}
+
\Vert\naby\zeta\Vert_{L^\infty_\by}
<R
\end{align*}
holds, then for any  $s>0$, $\varrho,p\in[1,\infty]$ and for any $\eta,\zeta \in B^{s}_{\varrho,p}(\omega)\cap W^{1,\infty}(\omega)$, we have that
\begin{align}
\label{210and212}
&\Vert \bm{\Psi}_\eta \Vert_{B^s_{\varrho,p}(\Omega\cup S_\alpha)}
+
\Vert \bm{\Psi}_\eta^{-1} \Vert_{B^s_{\varrho,p}(\Omega\cup S_\alpha)}
 \lesssim
1+ \Vert \eta \Vert_{B^s_{\varrho,p}(\omega)},
\\
\label{211and213}
&\Vert \bm{\Psi}_\eta - \bm{\Psi}_\zeta  \Vert_{B^s_{\varrho,p}(\Omega\cup S_\alpha)} 
+
\Vert \bm{\Psi}_\eta^{-1} - \bm{\Psi}_\zeta^{-1}  \Vert_{B^s_{\varrho,p}(\Omega\cup S_\alpha)} 
\lesssim
 \Vert \eta - \zeta \Vert_{B^s_{\varrho,p}(\omega)}
\end{align}
and
\begin{align}
\label{218}
&\Vert \partial_t\bm{\Psi}_\eta \Vert_{B^s_{\varrho,p}(\Omega\cup S_\alpha)}
\lesssim
 \Vert \partial_t\eta \Vert_{B^{s}_{ \varrho,p}(\omega)},
\qquad
\eta
\in W^{1,1}(I;B^{s}_{\varrho,p}(\omega))
\end{align}
holds uniformly in time with the hidden constants depending only on the reference geometry, on $L-\alpha$ and $R$. 

\subsection{Extension and smooth approximation on variable domains}\label{sec:ext}
In this subsection, we construct an extension operator
which extends functions from $\omega$ to the moving domain $\Omega_\eta$ for a given function $\eta$ defined on $\omega$. We follow \cite[Section 2.3]{BrScF}.
Since $\Omega$ is assumed to be sufficiently smooth, it is well-known that there is an extension operator $\mathscr F_\Omega$ which extends functions from $\partial\Omega$ to $\R^3$ and satisfies
\begin{equation*}\label{2.17a}
\mathscr F_\Omega:W^{\sigma,p}(\partial\Omega)\rightarrow W^{\sigma+1/p,p}(\R^3),
\end{equation*}
for all $p\in[1,\infty]$ and all $\sigma>0$, as well as $\mathscr F_\Omega v|_{\partial\Omega}=v$. Now we define $\mathscr F_\eta$ by 
\begin{align}\label{eq:2401b}
\mathscr F_\eta b=\mathscr F_\Omega ((b\bfn)\circ\bfvarphi^{-1})\circ{\bfPsi}_\eta^{-1},\quad b\in W^{\sigma,p}(\omega),
\end{align}
where $\bfvarphi$ is the function in the parametrization of  $\partial\Omega$.
If $\eta$ is regular enough, $\mathscr F_\eta$ behaves as a classical extension. 
We obtain the following Lemma which is a version of \cite[Lemma 2.2]{Br}, but also includes differentiabilities larger than 1.
\begin{lemma}\label{lem:3.8}
Let $\sigma>0$ and $p\in[1,\infty]$.
Let $\eta\in C^{0,1}(\omega)$ with $\|\eta\|_{L^\infty_\by}<\alpha<L$. Suppose further that $\eta\in B^{\sigma+1/p}_{\varrho,p}(\omega)$, where $p$ and $\varrho$ are related as in \eqref{eq:MSa} and \eqref{eq:MSb}.
 The operator
$\mathscr F_\eta$ defined in \eqref{eq:2401b} satisfies
\begin{align*}
\mathscr F_\eta: W^{\sigma,p}(\omega)\rightarrow W^{\sigma+1/p,p}(\Omega\cup S_\alpha)
\end{align*}
and $(\mathscr F_\eta b)\circ\bfvarphi_\eta=b\bfn$ on $\omega$ for all $b\in W^{\sigma,p}(\omega)$. In particular, we have
\begin{align*}
\|\mathscr F_\eta b\|_{W^{\sigma+1/p,p}(\Omega\cup S_\alpha)}\lesssim\|b\|_{W^{\sigma,p}(\omega)},
\end{align*}
where the hidden constant depends only on $\Omega,p,\sigma$, $\|\naby\eta\|_{L^\infty_\by}$, $\|\eta\|_{B^\sigma_{\varrho,p}}$ and $L-\alpha$.
\end{lemma}
\begin{proof}
On account of \eqref{210and212}
we have  $\bfPsi_\eta^{-1}\in B^{\sigma+1/p}_{\varrho,p}(\Omega\cup S_\alpha)$ as well.
By \eqref{eq:MSa} and \eqref{eq:MSb} this implies that $\bfPsi_\eta^{-1} \in \mathcal M(W^{\sigma+1/p,p}(\Omega\cup S_\alpha))$. Now 
\eqref{lem:9.4.1} becomes applicable and we obtain
\begin{align*}
\|\mathscr F_\eta b\|_{W^{\sigma+1/p,p}(\Omega\cup S_\alpha)}&\lesssim \|\mathscr F_\Omega ((b\bfn)\circ\bfvarphi^{-1})\|_{W^{\sigma+1/p,p}(\Omega)}\\
&\lesssim \|(b\bfn)\circ\bfvarphi^{-1})\|_{W^{\sigma,p}(\partial\Omega)}
\lesssim\|b\|_{W^{\sigma,p}(\omega)},
\end{align*}
which yields the claim.
\end{proof}
Finally, we prove the following smooth approximation result. For that we also require a solenoidal extension. The proof has been provided in \cite[Proposition 3.3]{MuSc}.
\begin{lemma}
\label{prop:musc}
For a given $\eta\in L^\infty(I;W^{1,2}( \omega ))$ with $\|\eta\|_{L^\infty_{t,y}}<\alpha<L$, there are linear operators
\begin{align*}
\mathscr K_\eta:L^1( \omega )\rightarrow\mathbb R,\quad
\Testzeta:\{\xi\in L^1(I;W^{1,1}( \omega )):\,\mathscr K_\eta(\xi)=0\}\rightarrow L^1(I;W^{1,1}_{\Div}(\Omega\cup S_{\alpha} )),
\end{align*}
such that the tuple $(\Testzeta(\xi-\mathscr K_\eta(\xi)),\xi-\mathscr K_\eta(\xi))$ satisfies
\begin{align*}
\Testzeta(\xi-\mathscr K_\eta(\xi))&\in L^\infty(I;L^2(\Omega_\eta))\cap L^2(I;W^{1,2}_{\Div}(\Omega_\eta)),\\
\xi-\mathscr K_\eta(\xi)&\in L^\infty(I;W^{2,2}( \omega ))\cap  W^{1,\infty}(I;L^{2}( \omega )),\\
\mathrm{tr}_\eta (\Testzeta&(\xi-\mathscr K_\eta(\xi))=\xi-\mathscr K_\eta(\xi),\\
\Testzeta(\xi-\mathscr K_\eta&(\xi))(t,x)=0 \text{ for } (t,x)\in I \times (\Omega \setminus S_{\alpha})
\end{align*}
provided that we have $\xi\in L^\infty(I;W^{2,2}( \omega ))\cap  W^{1,\infty}(I;L^{2}(\omega))$.
In particular, we have the estimates
\begin{align*}
\|\Testzeta(\xi-\mathscr K_\eta(\xi))\|_{L^q(I;W^{1,p}(\Omega \cup S_{\alpha}  ))}
&\lesssim \|\xi\|_{L^q(I;W^{1,p}( \omega ))}+\|\xi\nabla \eta\|_{L^q(I;L^{p}( \omega ))},\\
\|\partial_t\Testzeta(\xi-\mathscr K_\eta(\xi))\|_{L^q(I;L^{p}( \Omega\cup S_{\alpha}))}
&\lesssim \|\partial_t\xi\|_{L^q(I;L^{p}( \omega ))}+\|\xi\partial_t \eta\|_{L^q(I;L^{p}( \omega ))},
\\
\|\Testzeta(\xi-\mathscr K_\eta(\xi))\|_{L^q(I;W^{2,p}(\Omega \cup S_{\alpha}  ))}
&\lesssim \|\xi\|_{L^q(I;W^{2,p}( \omega ))}+\|\xi\nabla^2 \eta\|_{L^q(I;L^{p}( \omega ))}
\\
&\quad+\|\abs{\nabla \xi}\abs{\nabla \eta}\|_{L^q(I;L^{p}( \omega ))}+\|\abs{\xi}\abs{\nabla \eta}^2\|_{L^q(I;L^{p}( \omega ))}\\
&\quad+\|\xi\nabla \eta\|_{L^q(I;L^{p}( \omega ))},
\\
\|\partial_t\Testzeta(\xi-\mathscr K_\eta(\xi))\|_{L^q(I;W^{1,p}( \Omega\cup S_{\alpha}))}
&\lesssim \|\partial_t\xi\|_{L^q(I;W^{1,p}( \omega ))}+\|\xi\partial_t \nabla \eta\|_{L^q(I;L^{p}( \omega ))}
\\
&\quad+\|\abs{\partial_t \xi}\abs{\nabla \eta}\|_{L^q(I;L^{p}( \omega ))}+\|\abs{\nabla \xi}\abs{\partial_t \eta}\|_{L^q(I;L^{p}( \omega ))}
\\
&\quad+\|\xi\abs{\partial_t\eta}\abs{\nabla \eta}\|_{L^q(I;L^{p}( \omega ))},
\end{align*}
for any $p\in (1,\infty),q\in(1,\infty]$.
\end{lemma}
With the help of Lemma \ref{prop:musc} we obtain the following.
\begin{lemma}\label{lem:smooth}
For any tuple $(\eta,\bfv)$ belonging to the class
\begin{align}\label{eq:2206}\begin{aligned}
& L^2 \big(I; W^{2,\infty}(\omega) \big) \cap W^{1,2}(I;W^{2,2}(\omega)) 
\times L^2 \big(I; W^{1,2}_{\Div}(\Omega_{\eta}) \big)\cap C^0(\overline I;L^2(\Omega_{\eta}) )
\end{aligned}
\end{align}
and satisfying $\bfv\circ\bfvarphi_\eta=\partial_t\eta\bn$ on $\omega$, where $\|\eta\|_{L^\infty_{t,y}}<\alpha<L$ there is a sequence $(\eta_n,\bfv_n)$
which belongs to the class \eqref{eq:2206} and satisfies additionally
\begin{align*}
\eta_n\in C^\infty(I\times \omega),\quad \bfv_n\in W^{1,2} \big(I; W^{1,2}_{\Div}(\Omega_{\eta})\big),
\end{align*}
and $\bfv_n\circ\bfvarphi_{\eta}=(\partial_t\eta_n-\mathscr K_\eta(\partial_t\eta_n))\bfn$ on $\omega$, which converges to $(\eta,\bfv)$ strongly and has uniform bounds in the spaces given in \eqref{eq:2206}.
\end{lemma}
\begin{proof}
The proof is strongly related to \cite[Section 6]{MuSc}. We define
\begin{align*}
\tilde\bfv_0=\bfv-\Testzeta&(\partial_t\eta\bfn),
\end{align*}
which we use as decomposition. 
Note that $\partial_t\eta\bfn$ is the trace of the divergence free function $\bfv$, then it is not difficult to derive that (please refer to \cite[Lemma 6.3]{MuSc} for more details)
\begin{align*}
\mathscr K_\eta(\partial_t\eta\bfn)=0.
\end{align*}
This implies that $\Testzeta(\partial_t\eta\bfn)$ is well defined.  Further $\tilde\bfv_0$ has zero trace on $\partial\Omega_\eta$. 

Each part can be smoothly approximated. The first part uses $\eta_n$ as a smooth approximation of $\eta$ for instance by convolution in space and time. This also makes $\partial_t\eta_n$ and $\partial_t^2\eta_n$ smooth. Thus by Lemma \ref{prop:musc} we obtain
\begin{align*}
\Testzeta&(\partial_t\eta_n\bfn)\in W^{1,2} \big(I; W^{1,2}(\Omega_{\eta}) \big)\cap  L^2 \big(I; W^{2,2}(\Omega_{\eta}) \big)\cap C^0(\overline I;W^{1,2}(\Omega_{\eta}) ).
\end{align*}
Moreover, the tuple $(\eta_n,\Testzeta(\partial_t\eta_n-\mathscr K_\eta(\partial_t\eta_n))\bfn))$ converges to the expected limit with respect to the topology from \eqref{eq:2206}. As was already realized in \cite{MuSc,schwarzacher2022weak} this part of the approximation only uses the regularity of $\eta$ related to the energy estimate.

It is for the approximation  of $\tilde\bfv_0$ that more regularity has to be assumed for $\eta$. This is where  we use the Piola transformation $\mathscr{T}_{\eta}$ that changes the support of a function from $\Omega$ to $\Omega_\eta$ without changing the divergence.
It is defined as
\begin{align*}
\mathscr{T}_{\eta}\bfw
=
\big(\nabx  \bm{\Psi}_{\eta}(\mathrm{det}\nabx  \bm{\Psi}_{\eta})^{-1}
\bfw
\big)\circ \bm{\Psi}_{\eta}^{-1}\text{ with inverse }\mathscr{T}_{\eta}^{-1}\bfw
=
\big((\nabx  \bm{\Psi}_{\eta})^{-1}(\mathrm{det}\nabx  \bm{\Psi}_{\eta}) 
\bfw
\big)\circ \bm{\Psi}_{\eta} .
\end{align*}
Please note that the derivatives of $\mathscr{T}_{\eta}\bfw$ and $\mathscr{T}_{\eta}^{-1}\bfw$ can naturally be bounded by the respective derivatives of $\eta$. Hence as $\eta\in L^2(I,W^{2,\infty}(\omega))\cap W^{1,2}(I;W^{2,2}(\omega))$ we find for any $p,q\in [1,\infty]$ that 
\begin{align*}
\mathscr{T}_{\eta}:  W^{1,2}(I; W^{1,2}_{0,\Div}(\Omega)&\rightarrow  W^{1,2} \big(I; W^{1,2}_{0,\Div}(\Omega_\eta) \big),
\\
\mathscr{T}_{\eta}: L^p \big( I; L^q(\Omega)\big)&\to L^p\big( I; L^{q}(\Omega_\eta)\big),
\\
\mathscr{T}_{\eta}: L^2 \big( I; W^{1,2}_{0,\Div}(\Omega)\big)\cap L^\infty\big(I; L^{2}(\Omega)\big)&\rightarrow L^2 \big(I; W^{1,2}_{0,\Div}(\Omega_\eta) )\cap L^\infty\big( I; L^{2}(\Omega_\eta)\big),
\end{align*}
with uniform bounds. The same bounds hold for $\mathscr{T}_{\eta}^{-1}$. Estimates yielding these continuities can be shown by direct computations. Note, for instance, that the first estimate requires that $\nabla \eta\in L^\infty(I\times \omega)$.

 The approximation is then defined by first considering 
 $$\mathscr{T}_{\eta}^{-1}\tilde\bfv_0\in L^2 \big(I; W^{1,2}_{0,\Div}(\Omega) )\cap W^{1,2} \big(I; L^2(\Omega) \big).$$ 
 This function can now be smoothly approximated by a sequence $(\mathscr{T}_{\eta}^{-1}\tilde\bfv_0)_n\subset W^{1,2}(I;C^\infty_{c,\Div}(\Omega))$, where $(\mathscr{T}_{\eta}^{-1}\tilde\bfv_0)_n\to\mathscr{T}_{\eta}^{-1}\tilde\bfv_0$ almost everywhere as $n\rightarrow\infty$. This can be achieved by cutting-off the boundary, then applying a Bogovskij operator to make the function solenoidal and convoluting in time-space. We remark that all these operations are linear and continuous in the spaces regarded here. Now we fix
 \[
 \tilde\bfv_{0,n}:=\mathscr{T}_{\eta}(\mathscr{T}_{\eta}^{-1}\tilde\bfv_0)_n.
 \] 
 Now, by construction we have $\tilde\bfv_{0,n}\in W^{1,2} \big(I; W^{1,2}_{\Div}(\Omega_{\eta})\big)$, as required. Furthermore, we find that a.e.
 \begin{align*}
 \abs{\nabla \tilde\bfv_{0,n}}\lesssim \abs{(\mathscr{T}_{\eta}^{-1}\tilde\bfv_0)_n}+\abs{\nabla(\mathscr{T}_{\eta}^{-1}\tilde\bfv_0)_n}
 \end{align*}
with a hidden constant depending on $\abs{\nabla^2\eta}$.
 Hence we obtain that
 \begin{align*}
 \norm{\nabla \tilde\bfv_{0,n}}_{L^2(I,L^2(\Omega_\eta))}&\lesssim \norm{\nabla(\mathscr{T}_{\eta}^{-1}\tilde\bfv_0)_n}_{L^2(I,L^2(\Omega_\eta))}+\norm{(\mathscr{T}_{\eta}^{-1}\tilde\bfv_0)_n}_{L^\infty(I;L^2(\Omega_\eta))}
 \\
 &\lesssim \norm{\tilde\bfv_0}_{L^2(I,W^{1,2}(\Omega_\eta))}+\norm{\tilde\bfv_0}_{ L^\infty(I;L^2(\Omega_\eta))},
 \end{align*}
with a hidden constant depending on $\norm{\nabla \eta}_{L^\infty(I\times\omega)}$ and $\norm{\nabla^2 \eta}_{L^2(I;L^\infty(\omega))}$.
%
Combining the above, the sequence
$(\eta_n,\tilde\bfv_{0,n}+\Testzeta((\partial_t\eta_n-\mathscr K_\eta(\partial_t\eta_n))\bfn))$ has  the desired properties. 
\end{proof}

\subsection{The concept of solutions}
In this subsection, we introduce the notions of a solution to \eqref{1}--\eqref{interfaceCond} that are under consideration. We start with the definition of a weak solution.
\begin{definition}[Weak solution] \label{def:weakSolution}
Let $(\bff, g, \eta_0, \eta_*, \bu_0)$ be a dataset such that
\begin{equation}
\begin{aligned}
\label{dataset}
&\bff \in L^2\big(I; L^2_{\mathrm{loc}}(\mathbb{R}^3)\big),\quad
g \in L^2\big(I; L^2(\omega)\big), \quad
\eta_0 \in W^{2,2}(\omega) \text{ with } \Vert \eta_0 \Vert_{L^\infty( \omega)} < L, 
\\
& 
\eta_* \in L^2(\omega), \qquad \bu_0\in L^2_{\mathrm{\Div}}(\Omega_{\eta_0}) \text{ is such that }\bu_0\circ\bfvarphi_{\eta_0} =\eta_* \bfn\text{ on $\omega$}.
\end{aligned}
\end{equation} 
We call the tuple
$(\eta,\bu)$
a weak solution to the system \eqref{1}--\eqref{interfaceCond} with data $(\bff, g, \eta_0,  \eta_*,\bu_0)$ provided that the following holds:
\begin{itemize}
\item[(a)] The structure displacement $\eta$ satisfies
\begin{align*}
\eta \in W^{1,\infty} \big(I; L^2(\omega) \big)\cap W^{1,2} \big(I; W^{1,2}(\omega) \big)\cap  L^\infty \big(I; W^{2,2}(\omega) \big) \quad \text{with} \quad \Vert \eta \Vert_{L^\infty(I \times \omega)} <L,
\end{align*}
as well as $\eta(0)=\eta_0$ and $\partial_t\eta(0)=\eta_*$.
\item[(b)] The velocity field $\bu$ satisfies
\begin{align*}
 \bu \in L^\infty \big(I; L^2(\Omega_{\eta}) \big)\cap  L^2 \big(I; W^{1,2}_{\Div}(\Omega_{\eta}) \big) \quad \text{with} \quad 
\bu\circ\bfvarphi_\eta =\partial_t \eta {\bfn}\quad\text{on}\quad I\times\omega,
\end{align*}
as well as $\bu(0)=\bu_0$.
\item[(c)] For all  $(\phi, {\bfphi}) \in C^\infty(\overline{I}\times\omega) \times C^\infty(\overline{I} ;C^\infty_{\Div}(\R^3))$ with $\phi(T,\cdot)=0$, ${\bfphi}(T,\cdot)=0$ and $\bfphi\circ\bfvarphi_\eta =\phi {\bfn}$ on $I\times\omega$, we have
\begin{equation*}
\begin{aligned}
&\int_I  \frac{\mathrm{d}}{\dt}\bigg(\int_\omega \partial_t \eta \, \phi \dy
+
\int_{\Oeta}\bu  \cdot {\bfphi}\dx
\bigg)\dt 
\\
&=\int_I  \int_{\Oeta}\big(  \bu\cdot \partial_t  {\bfphi} + \bu \otimes \bu: \nabla {\bfphi} 
 -  
\nabla \bu:\nabla {\bfphi} +\bff\cdot{\bfphi} \big) \dx\dt
\\
&\quad+
\int_I \int_\omega \big(\partial_t \eta\, \partial_t\phi-\partial_t\naby\eta\cdot\naby \phi+
 g\, \phi-\Dely\eta\,\Dely \phi \big)\dy\dt.
 \end{aligned}
\end{equation*}
\item[(d)] For all $t\in I$, we have
\begin{equation*}\label{energyEst}
\begin{aligned}
&\mathcal{E}(t)
+
\int_0^t\int_\omega\vert\partial_t\naby\eta \vert^2\dy\ds
+
\int_0^t
 \int_{\Omega_{\eta(\sigma)}}\vert \nabla \bu \vert^2 \dx\ds
\\& \leq
 \mathcal{E}(0)
+\frac{1}{2}
\int_0^t \int_\omega  g\partial_t\eta\dy\ds
 + \frac{1}{2}\int_0^t \int_{\Omega_{\eta(\sigma)}}  \bu\cdot \mathbf{f} \dx\ds .
\end{aligned}
\end{equation*}
where
\begin{equation*}
\mathcal{E}(t)
:=
\frac{1}{2}
\int_\omega
\big(
\vert \partial_t\eta(t)\vert^2 + \vert \Dely\eta(t)\vert^2
\big)\dy
+
\frac{1}{2}
\int_{\Omega_{\eta(t)}}\vert   \bu(t) \vert^2 \dx.
\end{equation*}
\end{itemize}
\end{definition}

The existence of a weak solution can be shown as in \cite{LeRu}. The term $\partial_t\Dely\eta$ is not included there, but it does not alter the arguments. 
Note that here, we use a pressure-free formulation (that is, with test-function satisfying additionally $\Div\bfphi=0$). If the solution possesses more regularity,
the pressure can be recovered by solving
\begin{align}
\label{eq:press}
\begin{aligned}
-\Delta \pi&= \diver \big(\partial_t \bu +(\bu\cdot\nabla)\bu-\Delta\bu -\bff\big)\qquad&\text{ in }I\times \Omega_\eta,
\\
\pi  &= \big(\vert \mathrm{det}(\naby \bm{\varphi}_{\eta})\vert^{-1}(\partial_t^2\eta - \partial_t\Dely \eta +\Dely^2\eta-g)\bn^\intercal\bn_\eta\big)\circ\bm{\varphi}_{\eta}^{-1}
&
\\&\quad+\bn\circ\bm{\varphi}_{\eta}^{-1}\big(\nabx\bu+(\nabx\bu)^\intercal \big)\big)\bn_\eta\circ\bm{\varphi}_{\eta}^{-1} &\text{ on }I\times \partial\Omega_\eta.
\end{aligned}
\end{align}
 Setting
$\pi(t)=\pi_0(t)+c_\pi(t)$, where $\int_{\Omega_{\eta(t)}}\pi_0(t)\dx=0$ and $c_\pi=\pi-\pi_0$ is constant in space and 
testing the structure equation with 1 we obtain
\begin{equation}
\begin{aligned}\label{eq:pressure}
c_\pi(t)\int_{\omega}\bfn\cdot\bfn_\eta|\det(\naby\bfvarphi_\eta)|\dy
&=
\int_{\omega}\bfn\big(\nabla\bu+(\nabla\bu)^\intercal-\pi_0\mathbb I_{3\times 3}\big)\circ\bfvarphi_\eta\bfn_\eta|\det(\naby\bfvarphi_\eta)|\dy
\\&\quad+\int_\omega\partial_t^2\eta\dy-\int_\omega g\dy.
\end{aligned}
\end{equation}
Since $\Omega_\eta$ is $C^1$ uniformly in time, the operator $\Delta$ has the usual regularity and uniqueness properties for $C^1$ domains. In particular, it allows for a unique solution in $L^2$, if the right hand side is in $W^{-2,2}$ and the boundary value in $W^{-\frac{1}{2},2}$ or for a unique solution in $W^{1,2}$, provided that its boundary value is in $W^{\frac{1}{2},2}$ and the right hand side is in $W^{-1,2}$. Moreover, in this particular case, the solution of \eqref{eq:press} satisfies 
\begin{align*}
-\nabla \pi=\partial_t \bu +(\bu\cdot\nabla)\bu-\Delta\bu-\bff,
\end{align*}
distributionally which implies that
\begin{align*}
\int_I\int_{\Omega_\eta}|\nabla\pi|^2\dx\dt&\lesssim \int_I\int_{\Omega_\eta}\big(|\partial_t\bv\vert^2+|(\bu\cdot\nabla)\bu)|^2+|\Delta\bu|^2+\vert\bff\vert^2\big)\dx\dt
\\&\lesssim
\int_I\|\partial_t\bu\|^2_{L^{2}(\Omega_\eta)}\dt
+
\bigg( \int_I\|\bu\|^4_{L^4(\Omega_\eta)}\bigg)^{\frac{1}{2}}\bigg(\int_I\|\nabla\bu\|^4_{L^{4}(\Omega_\eta)}\dt\bigg)^{\frac{1}{2}}
\\&\quad
+\int_I\|\nabla^2\bu\|^2_{L^{2}(\Omega_\eta)}\dt
+\int_I\|\bff\|^2_{L^{2}(\Omega_\eta)}\dt,
\end{align*}
whenever the right hand side is finite, independent of the boundary value of $\pi$ in \eqref{eq:press}.
This is the case for a strong solution defined as follows.

\begin{definition}[Strong solution] \label{def:strongSolution}
We call the triple $(\eta,\bu,\pi)$ a strong solution to \eqref{1}--\eqref{interfaceCond} provided that $(\eta,\bu)$ is a weak solution to \eqref{1}--\eqref{interfaceCond}, which satisfies
$$
\eta \in W^{1,\infty} \big(I; W^{1,2}(\omega) \big)\cap W^{1,2} \big(I; W^{2,2}(\omega) \big)\cap  L^\infty \big(I; W^{3,2}(\omega) \big) \cap W^{2,2}(I;L^2(\omega))\cap L^2(I;W^{4,2}(\omega)),$$
$$ \bu \in W^{1,2} \big(I; L^{2}(\Omega_{\eta}) \big)\cap  L^2 \big(I; W^{2,2}(\Omega_{\eta}) \big),\quad\pi\in  L^2 \big(I; W^{1,2}(\Omega_{\eta}) \big).
$$
\end{definition}

For a strong solution $(\eta,\bu,\pi)$ the momentum equation holds in the strong sense, that is we have
 \begin{align} 
 \partial_t \bu+(\bu\cdot\nabla)\bu&=\Delta\bu-\nabla \pi+\bff&\label{1'}
 \end{align}
a.e. in $I\times\Omega_\eta$. The shell equation together with the regularity properties above yield $\eta\in L^2(I;W^{4,2}(\omega))$.  Hence the shell equation holds in the strong sense as well, that is, we have
 \begin{align}\label{2'}
\ \partial_t^2\eta-\partial_t\Dely\eta+\Dely^2\eta=g-\bn^\intercal\bm{\tau}\circ\bm{\varphi}_\eta\bn_\eta
\vert \mathrm{det}(\naby \bm{\varphi}_\eta)\vert
 \end{align}
a.e. in $I\times\omega$. 
Note that for a strong solution, the Cauchy stress $\bftau=\nabla\bu+(\nabla\bu)^\intercal-\pi\mathbb I_{3\times 3}$ possesses enough regularity to be evaluated at the moving boundary (this is due to the trace theorem and the uniform Lipschitz continuity of $\Omega_\eta$).

\subsection{The Stokes equations in non-smooth domains}
In this section, we present the necessary framework to parametrise
the boundary of the underlying domain $\Omega\subset\R^3$ by local maps of a certain regularity. This yields, in particular, a rigorous definition of a  $\mathcal M^{s,p}$-boundary. We follow the presentation from \cite{Br} (see also \cite{Br2}).

We assume that $\partial\Omega$ can be covered by a finite
number of open sets $\mathcal U^1,\dots,\mathcal U^\ell$ for some $\ell\in\mathbb N$, such that
the following holds. For each $j\in\{1,\dots,\ell\}$ there is a reference point
$\by^j\in\R^3$ and a local coordinate system $\{\be^j_1,\be^j_2,\be_3^j\}$ (which we assume
to be orthonormal and set $\mathcal Q_j=(\be_1^j|\be_2^j |\be_3^j)\in\mathbb R^{3\times 3}$), a function
$\varphi_j:\mathbb R^{2}\rightarrow\mathbb R$
and $r_j>0$
with the following properties:
\begin{enumerate}[label={\bf (A\arabic{*})}]
\item\label{A1} There is $h_j>0$ such that
$$\mathcal U^j=\{\bx=\mathcal Q_j\bz+\by^j\in\mathbb R^3:\,\bz=(\bz',z_3)\in\R^3,\,|\bz'|<r_j,\,
|z_3-\varphi_j(\bz')|<h_j\}.$$
\item\label{A2} For $\bx\in\mathcal U^j$ we have with $\bz=\mathcal Q_j^\intercal(\bx-\by^j)$
\begin{itemize}
\item $\bx\in\partial\Omega$ if and only if $z_3=\varphi_j(\bz')$;
\item $\bx\in\Omega$ if and only if $0<z_3-\varphi_j(\bz')<h_j$;
\item $\bx\notin\Omega$ if and only if $0>z_3-\varphi_j(\bz')>-h_j$.
\end{itemize}
\item\label{A3} We have that
$$\partial\Omega\subset \bigcup_{j=1}^\ell\mathcal U^j.$$
\end{enumerate}
In other words, for any $\bx_0\in\partial\Omega$ there is a neighborhood $U$ of $\bx_0$ and a function $\phi:\mathbb R^{2}\rightarrow\mathbb R$ such that after translation and rotation\footnote{By translation via $\by^j$ and rotation via $\mathcal Q_j$ we can assume that $\bx_0=\bm{0}$ and that the outer normal at~$\bx_0$ is pointing in the negative $x_3$-direction.}
 \begin{equation*}\label{eq:3009}
 U \cap \Omega = U \cap G,\quad G = \set{(\bx',x_3)\in \R^3 \,:\, \bx' \in \R^{2}, x_3 > \phi(\bx')}.
 \end{equation*}
 The regularity of $\partial\Omega$ will be described by means of local coordinates as just described.
 \begin{definition}\label{def:besovboundary}
 Let ${\mathcal{O}}\subset\R^3$ be a bounded domain, $s>0$ and $1\leq \rho,q\leq\infty$. We say that $\partial{\mathcal{O}}$ belongs to the class $\bfB^s_{\rho,q}$ if there is $\ell\in\mathbb N$ and functions $\varphi_1,\dots,\varphi_\ell\in\bfB^s_{\rho,q}(\mathbb R^{2})$ satisfying \ref{A1}--\ref{A3}.
 \end{definition}
Clearly, a similar definition applies for a Lipschitz boundary (or a $C^{1,\alpha}$-boundary with $\alpha\in(0,1)$) by requiring that $\varphi_1,\dots,\varphi_\ell\in W^{1,\infty}(\mathbb R^{2})$ (or $\varphi_1,\dots,\varphi_\ell\in C^{1,\alpha}(\mathbb R^{2})$). We say that the local Lipschitz constant of $\partial{\mathcal{O}}$, denoted by $\mathrm{Lip}(\partial{\mathcal{O}})$, is (smaller or) equal to some number $L>0$ provided that the Lipschitz constants of $\varphi_1,\dots,\varphi_\ell$ are not exceeding $L$. 

After these preparations let us consider the steady Stokes system
\begin{equation}\label{eq:Stokes}
\left\{\begin{aligned}
&\Delta \bu-\nabla\pi=-\bff,	\\
&\Div\bu=0,\\
&\bu|_{\partial{\mathcal{O}}}=\bu_{\partial},
\end{aligned}\right.
\end{equation}
in a domain ${\mathcal{O}}\subset\R^3$ with unit normal $\bfn$. The result given in the following theorem is a maximal regularity estimate for the solution of \eqref{eq:Stokes} in terms of the right-hand side. The boundary data under minimal assumption on the regularity of $\partial\mathcal O$ is obtained in \cite[Theorem 3.1]{Br}. 
\begin{theorem}\label{thm:stokessteady}
Let $p\in(1,\infty)$, $s\geq 1+\frac{1}{p}$ and 
\begin{align*}
\varrho\geq p\quad\text{if}\quad p(s-1)\geq 3,\quad \varrho\geq \tfrac{2p}{p(s-1)-1}\quad\text{if}\quad p(s-1)< 3,
\end{align*}
 such that 
$3\big(\frac{1}{p}-\frac{1}{2}\big)+1\leq  s$.
 Suppose that ${\mathcal{O}}$ is a $\bfB^{\theta}_{\varrho,p}$-domain for some $\theta>s-1/p$ with locally small Lipschitz constant, $\bff\in W^{s-2,p}({\mathcal{O}})$ and $\bu_{\partial}\in W^{s-1/p,p}(\partial{\mathcal{O}})$ with $\int_{\partial{\mathcal{O}}}\bu_\partial\cdot\bfn\,\dd\mathcal H^{2}=0$. Then there is a unique solution to \eqref{eq:Stokes} and we have
\begin{align*}
\|\bu\|_{W^{s,p}({\mathcal{O}})}+\|\pi\|_{W^{s-1,p}({\mathcal{O}})}\lesssim\|\bff\|_{W^{s-2,p}({\mathcal{O}})}+\|\bu_{\partial}\|_{W^{s-1/p,p}(\partial{\mathcal{O}})}.
\end{align*}
\end{theorem}

\begin{remark}\label{rem:stokes}
{\rm
In Section \ref{sec:reg} we have to apply Theorem \ref{thm:stokessteady} to the domain $\mathcal O={\Omega}_{\eta(t)}$ for a fixed $t$.
Here $\eta:I\times\omega\rightarrow\R$ and $\partial\Omega_{\eta(t)}$ is parametrised via the function $\bfvarphi_{\eta(t)}$ defined on $\omega$, see Section \ref{ssec:geom} for details.
 We exclude self-intersection and degeneracy by assumption (in particular, $\partial_1\bfvarphi_\eta\times\partial_2\bfvarphi_\eta\neq 0$ such that $\bfn_{\eta(t)}$ is well-defined). Given $\bx_0\in \partial{\Omega}_{\eta(t)}$, for some $t\in I$ fixed, we rotate the coordinate system such that $\bfn_{\eta(t)}(\by(\bx_0))=(0,0,1)^\intercal$. Accordingly, it holds
\begin{align*}\mathrm{det}\big(\naby\widetilde\bfvarphi_{\eta(t)}\big)=1,\quad\widetilde\bfvarphi_\eta=\begin{pmatrix}\varphi^1_\eta\\\varphi^2_\eta\end{pmatrix}.
\end{align*}
Hence the function $\widetilde\bfvarphi_{\eta(t)}$ is invertible in a neighborhood $\mathcal U$ of $\by(\bx_0)$. We define in
$\widetilde\bfvarphi_{\eta(t)}(\mathcal U)$ the function
\begin{align*}
\bfphi(z)
=\begin{pmatrix}z\\ \phi(z))\end{pmatrix}
=\begin{pmatrix}z\\ \varphi^3_{\eta(t)}((\widetilde\bfvarphi_{\eta(t)})^{-1}(z))\end{pmatrix}.
\end{align*}
It describes the boundary $\partial{\Omega}_{\eta(t)}$ close to $\bx_0$.
One easily checks that $\partial_z\phi(z_0)=0$ such that $\partial_z\phi$ is small close to $z_0$. Suppose now that $\eta\in L^\infty(\overline{I};B^{\theta}_{\varrho,p}\cap C^{1}(\omega))$, $\theta>2-1/p$, where $p$ and $\varrho$ are related to \eqref{eq:MSa} and  \eqref{eq:MSb} and that $\displaystyle\sup_I \mathrm{Lip}(\partial\Omega_{\eta(t)})$ is sufficiently small. Then we conclude that 
\begin{align*}
\|\phi\|_{\mathcal M(W^{2-1/p,p}_\by)}\leq \delta,\quad \|\phi\|_{W^{1,\infty}_\by}\leq \delta,
\end{align*}
holds uniformly in time for some sufficiently small $\delta$ in a neighbourhood of $z_0$ (using also $\phi(z_0)=0$ and $\partial_z\phi(z_0)=0$).
 }
\end{remark}

\subsection{Universal Bogovskij}

Bogovskij operators are natural to be considered in star-shaped domains. As Lipschitz domains are unions of star-shaped domains, for some time Bogovskij operators are available on Lipschitz domains. Recently the concept of universal Bogovskij operators was introduced in \cite{KamSchSpe20}. Observe that the same Bogovskij operator actually can be used for a family of domains, as long as the Lipschitz constant is controlled. This allows to use a (locally) steady operator to correct the divergence in the time-changing domains.

More precisely in \cite[Corollary 3.4]{KamSchSpe20} the following statement was shown:
\begin{theorem} \label{thm:ndBog1}
	Let $\Sigma \subset \R^{n-1}$ be a bounded Lipschitz-domain, $M > \gamma > 0$, $C_L>0$, $b\in C_0^\infty(\Sigma \times [0,\gamma])$ with unit integral. Then there exists a linear, universal Bogovskij operator $\Bog: C_0^\infty(\Sigma \times [0,M]) \to C_0^\infty(\Sigma \times [0,M];\R^{n-1})$ such that for any $C_L$-Lipschitz function (i.e. with Lipschitz constant $C_L$) $\eta:\Sigma \to [\gamma,M]$ and $\Omega_\eta := \{(\bx',x_n) \in \Sigma \times [0,M]: 0< x_n < \eta(\bx')\}$ the operator $\Bog$ maps $C_0^\infty(\Omega_\eta)$ to $C_0^\infty(\Omega_\eta;\R^n)$ with $\Div\Bog f = f - b \int f \dx$. 
	In addition,
	\[\norm{\Bog(f)}_{W^{s+1,p}( \Omega_\eta;\R^n )} \leq C_B^{s,p}\norm{f}_{W^{s,p}( \Omega_\eta)},\]
	for all $ 1<p<\infty $ and $ s \ge 0$
	with $C_B^{s,p}$ only depending on $s$, $p$, $\operatorname{diam}(\Sigma), C_L,\gamma $ and the Lipschitz properties of $\Sigma$. %
\end{theorem}
In order to make this operator admissible for our needs we introduce the following version. 
\begin{theorem} \label{thm:ndBog}
	There is a universal Bogovskij operator, such that for all $\Omega_\eta$ defined through \eqref{eq:2612} with $\norm{\nabla \eta}_\infty\leq C_L$, $\norm{\eta}_\infty\leq L$ and $b\in C_0^\infty(\Omega\setminus S_L)$ with unit integral
	\[
	\Bog: C_0^\infty(\Omega_\eta) \to C_0^\infty(\Omega_\eta;\R^n)\text{ with }\Div\Bog f = f - b \int f \dx.
	\]
	In addition,
	\[\norm{\Bog(f)}_{W^{s+1,p}( \Omega_\eta;\R^n )} \leq C_B^{s,p}\norm{f}_{W^{s,p}( \Omega_\eta)},\]
	for all $ 1<p<\infty $ and $ s \ge 0$
	with $C_B^{s,p}$ only depending on $L,\varphi,C_L$.
\end{theorem}

\begin{proof}
	The proof is by now standard. One covers the domain $S_L$ with balls of finite overlap, such that on each ball all possible functions $\eta$ can be written as a graph. On these sets one may apply Theorem~\ref{thm:ndBog1}. Hence the partition of unity argument introduced in~\cite[Section 3.1]{SaaSch21} allows to construct the desired operator.
\end{proof} 

\begin{remark}[Time-derivative of Bogovskij operator]\label{time-bog}
{\rm Please observe that the same operator can be applied to a time-changing function with time-changing support. The operator then automatically has zero trace on the variable support of the function. In particular for $\sup_t\norm{\nabla \eta(t)}_\infty\leq C_L$, $\sup_t\norm{\eta(t)}_\infty\leq L$, we find that
 $\partial_t \mathcal{B}(f\chi_{\Omega_\eta})=\mathcal{B} (\partial_tf\chi_{\Omega_\eta})$ and $\mathcal{B}(\partial_t f\chi_{\Omega_\eta})=0$ on $\partial\Omega_\eta$.
}
\end{remark}

\section{Local strong solutions}
\label{sec:loc}
Our goal in this section is to construct a local-in-time strong solution of \eqref{1}--\eqref{interfaceCond}. The main theorem is the following:
\begin{theorem}
\label{thm:fluidStructureWithoutFK}
Suppose that the dataset
$(\bff, g, \eta_0, \eta_*, \bu_0)$
satisfies \eqref{dataset} and in addition
\begin{align}
\label{datasetImproved}
\bff\in L^2(I; L^2_{\rm loc}(\mathbb{R}^{3})), \quad
g\in L^2(I; L^2(\omega)),\quad \eta_0 \in W^{3,2}(\omega), \quad \eta_* \in W^{1,2}(\omega), \quad \bu_0 \in W^{1,2}_{\divx}(\Omega_{\eta_0}).
\end{align}
There is a time $T^*>0$ such that there exists a unique strong solution to \eqref{1}--\eqref{interfaceCond} in the sense of Definition \ref{def:strongSolution}.
\end{theorem}

The main ideas to prove Theorem \ref{thm:fluidStructureWithoutFK} are as follows.
\begin{itemize}
\item We transform the fluid-structure system to its reference domain.
\item We then linearize the resulting system on the reference domain and obtain estimates for the linearized system.
\item We construct a contraction map for the linearized problem (by choosing the end time
small enough) which gives the local solution to the system on its original/actual domain.
\end{itemize}
This is reminiscent of the approach in \cite{Br,GraHil,GraHilLe,Le}.

\subsection{Transformation to reference domain}\label{Sectionlinear}
For a solution $( \eta, \bu,  \pi )$  of \eqref{1}--\eqref{interfaceCond}, we set $\overline{\pi}=\pi\circ \Psi_\eta$ and 
$\overline{\bu}=\bu\circ \Psi_\eta$
and define
\begin{equation*}\label{matrices}
\begin{aligned}
\mathbf{A}_\eta=J_\eta\big( \nabx \bfPsi_\eta^{-1}\circ \bfPsi_\eta \big)^\intercal\nabx \bfPsi_\eta^{-1}\circ \bfPsi_\eta,&\\
\mathbf{B}_\eta=J_\eta \left(\nabx \bfPsi_\eta^{-1}\circ \bfPsi_\eta\right)^\intercal,&\\
h_\eta(\overline{\bu})=\big( \mathbf{B}_{\eta_0}-\mathbf{B}_\eta\big):\nabx \overline{\bu},&
\\
\mathbf{H}_\eta(\overline{\bu}, \overline{\pi})
=\
\big( \mathbf{A}_{\eta_0}-\mathbf{A}_\eta\big)\nabx \overline{\bu}
-
\big( \mathbf{B}_{\eta_0}-\mathbf{B}_\eta\big) \overline{\pi},&
\\
\mathbf{h}_\eta(\overline{\bu})
=
(J_{\eta_0}-J_\eta)\partial_t \overline{\bu}
-
J_\eta \nabx\overline{\bu}\cdot \partial_t \bfPsi_\eta^{-1}\circ \bfPsi_\eta -\mathbf{B}_\eta\nabx\overline{\bfv}~\overline{\bu}
+
J_\eta  \bff \circ \bfPsi_\eta,
\end{aligned}
\end{equation*}
where $J_\eta=\mathrm{det}(\nabla\bfPsi_\eta)$.
Exactly as in the two-dimensional case considered in  \cite[Lemma 4.2]{Br} we obtain the following result.
\begin{theorem}
\label{thm:transformedSystem}
Suppose that the dataset
$(\bff, g, \eta_0, \eta_*, \bu_0)$
satisfies \eqref{dataset} and \eqref{datasetImproved}.
Then $( \eta, \bu,  \pi )$ is a strong solution to \eqref{1}--\eqref{interfaceCond} in the sense of Definition \ref{def:strongSolution}, if and only if $( \eta, \overline{\bu},  \overline{\pi} )$ is a strong solution of
\begin{align}
\label{contEqAloneBar}
\mathbf{B}_{\eta_0}:\nabx \overline{\bu}= h_\eta(\overline{\bu}),
\\
\partial_t^2\eta - \partial_t\Dely \eta + \Dely^2\eta
=
g+\bn^\intercal \big[\mathbf{H}_\eta(\overline{\bu}, \overline{\pi})-\mathbf{A}_{\eta_0}  \nabx\overline{\bu} +\mathbf{B}_{\eta_0}\overline{\pi}\big]\circ\bm{\varphi} \bn ,
\label{shellEqAloneBar}
\\
J_{\eta_0}\partial_t \overline{\bu}  -\divx(\mathbf{A}_{\eta_0}  \nabx\overline{\bu}) 
 +\divx(\mathbf{B}_{\eta_0}\overline{\pi}) 
 = 
\mathbf{h}_\eta(\overline{\bu})-
\divx  \mathbf{H}_\eta(\overline{\bu}, \overline{\pi})
\label{momEqAloneBar}
\end{align}
with  $\overline{\bu}  \circ \bm{\varphi}  =(\partial_t\eta)\bn$ on $I\times \omega$.
\end{theorem}

\subsection{The linearized problem}
In this section, we let $(g, \eta_0, \eta_*, \bu_0)$ be as before in Theorem \ref{thm:transformedSystem}. In addition, we consider $(h,\mathbf{h}, \mathbf{H})$ such that 
\begin{equation}
\begin{aligned}
\label{differentHdata}
&h\in L^2\big(I;W^{1,2}(\Omega)\big) \cap W^{1,2}\big(I;W^{-1,2}(\Omega)\big)\cap \{h(0,\bx)=0\},
\\
&\mathbf{h} \in L^2(I\times\Omega),\quad \mathbf{H}\in L^2(I;W^{1,2}(\Omega)),
\end{aligned}
\end{equation}
and study the following linear system
\begin{align}
\label{contEqAloneBarLinear}
\mathbf{B}_{\eta_0}:\nabx \overline{\bu}= h,
\\
 \partial_t^2\eta - \partial_t\Dely \eta +  \Dely^2\eta
=
g+\bn^\intercal \big[\mathbf{H} -\mathbf{A}_{\eta_0}  \nabx\overline{\bu} +\mathbf{B}_{\eta_0}\overline{\pi}\big]\circ\bm{\varphi} \bn ,
\label{shellEqAloneBarLinear}
\\
J_{\eta_0}\partial_t \overline{\bu}  -\divx(\mathbf{A}_{\eta_0}  \nabx\overline{\bu}) 
 +\divx(\mathbf{B}_{\eta_0}\overline{\pi}) 
 = 
\mathbf{h}-
\divx  \mathbf{H} 
\label{momEqAloneBarLinear}
\end{align}
with  $\overline{\bu}  \circ \bm{\varphi}  =(\partial_t\eta)\bn$ on $I\times \omega$. It is important to note that $\mathbf{B}_{\eta_0}$ and $\mathbf{A}_{\eta_0} $ are time-independent.

\begin{proposition}
\label{thm:transformedSystemLinear}
Suppose that the dataset
$(g, \eta_0, \eta_*, \overline{\bu}_0, h, \mathbf{h},\mathbf{H})$
satisfies \eqref{dataset}, \eqref{datasetImproved} and \eqref{differentHdata}.
Then there exists a strong solution $( \eta, \overline{\bu},  \overline{\pi} )$ of \eqref{contEqAloneBarLinear}--\eqref{momEqAloneBarLinear} such that
\begin{equation}
\begin{aligned}
\label{energyEstLinear}
&\sup_I\int_\omega
\big(\vert \partial_t\naby \eta\vert^2 
+
\vert \naby\Dely \eta\vert^2
\big)
\dy
+
\sup_I\int_\Omega\vert\nabx \overline{\bu}\vert^2\dx
\\&\quad+
\int_I\int_\omega
\big(\vert \partial_t\Dely \eta \vert^2 + \vert \partial_t^2 \eta\vert^2+| \Dely^2\eta|^2
 \big)\dy\dt
 +
\int_I\int_\Omega\big( \vert \nabx^2\overline{\bu}\vert^2 +\vert \partial_t\overline{\bu} \vert^2  +\vert \overline{\pi}\vert^2 + \vert \nabx \overline{\pi}\vert^2
 \big)\dx\dt
 \\&\lesssim
 \int_\omega\big( \vert \eta_*\vert^2
 +
 \vert \naby\eta_*\vert^2
 +
 \vert \Dely\eta_0\vert^2
 +
  \vert \naby\Dely\eta_0\vert^2
  \big)\dy
  +
  \int_\Omega \big(\vert
  \overline{\bu}_0\vert^2
  +
   \vert\nabx\overline{\bu}_0\vert^2 \big)\dx
   \\&\quad+
   \int_I\Vert \partial_t h \Vert_{W^{-1,2}(\Omega)}^2\dt
   +
 \int_I\int_\omega \vert g\vert^2 \dy\dt
  \\&\quad+
  \int_I\int_\Omega\big(
  \vert h\vert^2
  +
  \vert \nabx h\vert^2
  +
   \vert \mathbf{h}\vert^2 +
  \vert \mathbf{H}\vert^2
  +
  \vert \nabx\mathbf{H}\vert^2
  \big)
 \dx\dt.
\end{aligned}
\end{equation}
\end{proposition}
\begin{proof}
The solution can be constructed by a finite-dimensional Galerkin approximation.
 Let us explain how to construct the Galerkin basis. By solving  eigenvalue problems of the Stokes operator we construct a smooth orthogonal basis $(\tilde \bfX_\ell)_{\ell\in\N}$
of $W^{1,2}_{0,\Div}(\Omega_{\eta_0})$. Further by solving eigenvalue problems for the Laplace operator we construct a smooth orthogonal basis $(Y_\ell)_{\ell\in\N}$ of $$\bigg\{\zeta\in W^{2,2}(\omega):\,\,\int_{\Omega_{\eta_0}}\zeta\circ\bfvarphi_{\eta_0}^{-1}\,\dd\mathcal H^2\bigg\}.$$ Then we set $\bfX_\ell:=\tilde\bfX_\ell\circ\bfPsi_{\eta_0}$, which yields a basis of
\begin{align}\label{eq:2306}
\{\bfw\in W^{1,2}_0(\Omega):\,\,\mathbf B_{\eta_0}:\nabla\bfw=0\}.
\end{align}
We define vector fields $\tilde\bfY_\ell$ by setting $\tilde \bfY_\ell=\mathscr F^{\Div}_\Omega(( Y_\ell \bfn)\circ\bfvarphi^{-1})$, where $\mathscr F^{\Div}_\Omega$ is a solenoidal extension operator. It can be constructed by means of a standard extension and a Bogosvkii correction and thus maps $W^{k,2}(\omega)\rightarrow W^{k,2}(\R^n)$ for $k\in\N$ such that the $\tilde\bfY_\ell$'s are smooth. The functions
$\bfY_\ell:=\tilde\bfY_\ell\circ\bfPsi_{\eta_0}$ also belong to the function space in \eqref{eq:2306}. Now we choose an enumeration $(\bfomega_\ell)_{\ell\in\N}$ of $(\bfX_\ell)_{\ell\in\N}\cup (\bfY_\ell)_{\ell\in\N}$. 
Note that we use $(Y_\ell)_{\ell\in\N}$ as the Ansatz space for $\partial_t\eta$ and not $\eta$, as this is what is related to the fluid-velocity. The terms depending on $\eta$ are hence taken as primitive of the discretization of $\partial_t\eta$.
We now give a formal proof of estimate \eqref{energyEstLinear}, which can be made rigorous with the help of the Galerkin approximation.

Consider the pair of test functions $(\partial_t \eta, \overline{\bu})$ (which can be used as a test-function on the Galerkin level by construction) for
 \eqref{shellEqAloneBarLinear} and  \eqref{momEqAloneBarLinear}
 respectively. 
We use the ellipticity of $\mathbf{A}_{\eta_0} $,
\begin{equation}
\begin{aligned}
\label{3.8}
&\sup_I
\int_\omega\big(\vert \partial_t\eta\vert^2
+
\vert  \Dely \eta\vert^2 \big)\dy
 +
 \int_I\int_\omega
\vert \partial_t \naby \eta\vert^2 \dy\dt
 +
 \sup_I
 \int_\Omega\vert  \overline{\bu}\vert^2\dx
+
 \int_I\int_\Omega\vert  \nabx\overline{\bu}\vert^2\dx\dt
\\
&\lesssim
\int_\omega\big(\vert \eta_*\vert^2
+
\vert \Dely \eta_0\vert^2 \big)\dy
 +
 \int_I\int_\omega
\vert  g\vert^2 \dy\dt
  +
   \int_\Omega\vert \overline{\bu}_0\vert^2\dx
+
 \int_I\int_\Omega\big(\vert  \mathbf{h}\vert^2 
 +
 \vert  \mathbf{H}\vert^2 
 \big)
 \dx\dt
\\&
\quad+
 \int_I\int_\Omega\big(\vert  h\vert^2 
+
 \vert  \overline{\pi}\vert^2 
 \big)
 \dx\dt.
\end{aligned}
\end{equation}
If we now consider  $(\partial_t^2\eta,\partial_t \overline{\bu})$ as test functions (which can be used as a test-function on the Galerkin level since $B_{\eta_0}$ is time independent) for \eqref{shellEqAloneBarLinear} and \eqref{momEqAloneBarLinear} respectively, then we obtain
\begin{equation*}
\begin{aligned}
\label{3.9}
&
 \int_I\int_\omega
\vert \partial_t^2 \eta\vert^2 \dy\dt
 +
 \frac{1}{2}
 \int_\omega \vert \partial_t \naby\eta\vert^2
\dy
 +
 \int_I\int_\Omega\vert \partial_t \overline{\bu}\vert^2\dx\dt
 +
 \frac{1}{2}
 \int_\Omega\vert \nabx\overline{\bu}\vert^2\dx
\\
&=
\frac{1}{2}
 \int_\omega \vert \naby\eta_*\vert^2\dy
 + 
 \int_I\int_\omega
 \big( g -
  \Dely^2 \eta
 +\bn^\intercal \mathbf{H}\circ \bm{\varphi}\bn
 \big)  \partial_t^2\eta
  \dy\dt
\\&
\quad +
 \frac{1}{2}
 \int_\Omega\vert \nabx\overline{\bu}_0\vert^2\dx
 +
\int_I\int_\Omega\big( \mathbf{h}
-
\divx \mathbf{H} 
 \big)\cdot\partial_t \overline{\bu}
 \dx\dt
+
 \int_I\int_\Omega
 \overline{\pi}\partial_th
 \dx\dt
 .
\end{aligned}
\end{equation*}
We note that
\begin{equation}
\begin{aligned}
\label{3.10}
 -\int_I\int_\omega
   \Dely^2 \eta
 \cdot \partial_t^2\eta
  \dy\dt
  &=
  \int_I\int_\omega
 \partial_t( \naby \Dely \eta
 \cdot \partial_t  \naby\eta)
  \dy\dt
  +
  \int_I\int_\omega
  \vert
 \partial_t  \Dely \eta
 \vert^2
  \dy\dt
  \\&
  \leq
  \frac{1}{4}
  \sup_I\int_\omega
 \vert \partial_t  \naby\eta
 \vert^2
  \dy
  +
  \sup_I\int_\omega
 \vert \naby \Dely \eta
 \vert^2
  \dy
  +
  \int_I\int_\omega
  \vert
 \partial_t  \Dely \eta
 \vert^2
  \dy\dt,
\end{aligned}
\end{equation}
and we thus obtain
\begin{equation}
\begin{aligned}
\label{3.11}
&
 \int_I\int_\omega
\vert \partial_t^2 \eta\vert^2 \dy\dt
 + 
 \sup_I\int_\omega \vert \partial_t \naby\eta\vert^2
\dy
 +
 \int_I\int_\Omega\vert \partial_t \overline{\bu}\vert^2\dx\dt
 + 
 \sup_I\int_\Omega\vert  \nabx\overline{\bu}\vert^2\dx
\\
&\lesssim
 \int_\omega \vert \naby\eta_*\vert^2\dy
+
  \sup_I\int_\omega
 \vert \naby \Dely \eta
 \vert^2
  \dy
   + 
 \int_I\int_\omega
 \big(\vert  g \vert^2
 +
 \vert
 \partial_t  \Dely \eta
 \vert^2
 \big)
  \dy\dt
  \\&\quad + 
\int_I\Vert
 \mathbf{H}\Vert_{W^{1,2}_{\bx}}^2\dt
 +
 \int_\Omega\vert \nabx\overline{\bu}_0\vert^2\dx
 +
\int_I\int_\Omega\big( \vert \mathbf{h}\vert^2
+
\vert
 \nabx\mathbf{H}\vert^2 
 \big) \dx\dt
 \\&
\quad+
 \int_I\Vert \partial_t h\Vert_{W^{-1,2}_\bx}^2\dt
 +
 \int_I\Vert\overline{\pi}\Vert_{W^{1,2}_\bx}^2\dt,
\end{aligned}
\end{equation}
where we have used the trace theorem
\begin{align*}
 \int_I\int_{\partial\Omega}
 \vert
   \mathbf{H}\vert^2
  \dd\mathcal{H}^2\dt
  \lesssim
   \int_I\Vert
 \mathbf{H}\Vert_{W^{1,2}_{\bx}}^2\dt.
\end{align*}
We now test \eqref{shellEqAloneBarLinear}  with $-\partial_t \Dely\eta$ (which can be used as a test-function on the Galerkin level since we used the eigenvalues of the Laplace equation $-\Dely Y_\ell=\lambda_\ell Y_\ell$), to obtain
\begin{equation}
\begin{aligned}
\label{3.13}
&\int_\omega\big( \vert \partial_t \naby\eta\vert^2
+
\vert  \naby\Dely\eta\vert^2
\big)
\dy
   + 
 \int_I\int_\omega
 \vert
 \partial_t \Dely \eta
 \vert^2
  \dy\dt\\
 &\lesssim
 \int_\omega\big( \vert \naby\eta_*\vert^2
+
\vert   \naby\Dely\eta_0\vert^2
\big)
\dy
+ 
\int_I\int_\omega
\vert g \vert^2
\dy\dt
  \\& 
   \quad+
 \int_I\int_\omega
 \big\vert\bn^\intercal\big( \mathbf{H} -\mathbf{A}_{\eta_0}  \nabx\overline{\bu} +\mathbf{B}_{\eta_0} \overline{\pi}\big)\circ\bm{\varphi}  \bn\,\partial_t \Dely\eta\big\vert
  \dy\dt.
\end{aligned}
\end{equation}
For the last term, we have
\begin{equation}
\begin{aligned}
\label{findPressBar}
 &\int_I\int_\omega
 \big\vert\bn^\intercal\big( \mathbf{H} -\mathbf{A}_{\eta_0}  \nabx\overline{\bu} +\mathbf{B}_{\eta_0} \overline{\pi}\big)\circ\bm{\varphi}  \bn\,\partial_t \Dely\eta\big\vert
  \dy\dt
\\&\lesssim
 \int_I
 \big(\Vert \mathbf{H}\Vert_{W^{1/2,2}(\partial\Omega)} + \Vert  \nabx\overline{\bu}\Vert_{W^{1/2,2}(\partial\Omega)} 
 +
\Vert \overline{\pi}\Vert_{W^{1/2,2}(\partial\Omega)}
 \big)\Vert\partial_t \Dely\eta\Vert_{W^{-1/2,2}_{\by}}
 \dt
 \\&\lesssim
 \int_I
  \big(\Vert \mathbf{H}\Vert_{W^{1,2}_{\bx}} + \Vert  \nabx\overline{\bu}\Vert_{W^{1,2}_{\bx}} 
 +
\Vert \overline{\pi}\Vert_{W^{1,2}_{\bx}}
 \big)\Vert\partial_t  \eta\Vert_{W^{3/2,2}_{\by}}
 \dt
  \\&\lesssim
 \int_I
  \big(\Vert \mathbf{H}\Vert_{W^{1,2}_{\bx}} + \Vert   \nabx\overline{\bu}\Vert_{W^{1,2}_{\bx}} 
 +
\Vert \overline{\pi}\Vert_{W^{1,2}_{\bx}}
 \big)\Vert\partial_t  \eta\Vert_{W^{1,2}_{\by}}^{1/2}
 \Vert\partial_t  \eta\Vert_{W^{2,2}_{\by}}^{1/2}
 \dt
  \\&\leq \kappa
 \int_I
 \big(\Vert \mathbf{H}\Vert_{W^{1,2}_{\bx}}^2 + \Vert  \nabx\overline{\bu}\Vert_{W^{1,2}_{\bx}}^2 
 +
\Vert \overline{\pi}\Vert_{W^{1,2}_{\bx}}^2
+ 
\Vert\partial_t  \Dely\eta\Vert_{L^2_{\by}}^2
 \big)\dt
 +
 c(\kappa)
 \int_I
 \Vert\partial_t  \eta\Vert_{W^{1,2}_{\by}}^2
 \dt,
\end{aligned}
\end{equation}
where $\kappa>0$ is arbitrary. Here, we have used the equivalent relation
\begin{align}
\label{equivNorm}
\Vert f\Vert_{W^{2,2}_{\by}}
\lesssim
\Vert \Dely f\Vert_{L^{2}_{\by}}
\lesssim
\Vert f\Vert_{W^{2,2}_{\by}},
\end{align}
which holds for any $f\in W^{2,2}_{\by}$. The second inequality in \eqref{equivNorm} is straightforward whereas the first inequality follows from the estimate $\Vert\naby^2 f\Vert_{L^2_{\by}}
\lesssim
\Vert \Dely f\Vert_{L^{2}_{\by}}$ that is derived by solving the trivial elliptic equation $\Dely f=\Dely f$.
\\
Moreover, notice that by using \eqref{3.8} to estimate the last term in \eqref{findPressBar}, we obtain from \eqref{3.13} and \eqref{3.8} that
\begin{equation}
\begin{aligned}
\label{3.16}
&\sup_I\int_\omega\big( \vert \partial_t \naby\eta\vert^2
+
\vert  \naby\Dely\eta\vert^2
\big)
\dy
   + 
 \int_I\int_\omega
 \vert
 \partial_t \Dely \eta
 \vert^2
  \dy\dt
\\&
\leq
 \kappa
 \int_I
 \big(  \Vert  \nabx\overline{\bu}\Vert_{W^{1,2}_{\bx}}^2 
 +
\Vert \overline{\pi}\Vert_{W^{1,2}_{\bx}}^2
  \big)\dt
  +
c(\kappa)
 \int_\omega\big( \vert \eta_*\vert^2+\vert \naby\eta_*\vert^2
+
\vert \Dely \eta_0\vert^2+
\vert   \naby\Dely\eta_0\vert^2
\big)
\dy
  \\&  \quad+
  c(\kappa)
   \int_\Omega\vert \overline{\bu}_0\vert^2\dx 
   + 
   c(\kappa)
 \int_I\int_\omega
  \vert g \vert^2
  \dy\dt
 +
 c(\kappa)
 \int_I\int_\Omega\big(\vert  h\vert^2+\vert  \mathbf{h}\vert^2 
 +
 \vert  \mathbf{H}\vert^2 
  +
 \vert \nabx \mathbf{H}\vert^2 
 \big)
 \dx\dt.
\end{aligned}
\end{equation}
To find an estimate for the pressure term in \eqref{3.16}, we decompose it into   $\overline{\pi}=\overline{\pi}_0 +c_{\overline{\pi}}$ where $\int_\Omega \overline{\pi}_0\dx=0$ and $c_{\overline{\pi}}$ is only dependent of time. We therefore deduce from
\eqref{shellEqAloneBarLinear} that
\begin{align*}
 c_{\overline{\pi}}\int_\omega\bn^\intercal\mathbf{B}_{\eta_0}\circ\bm{\varphi} \bn 
\dy
=
\int_\omega
\big(
 \partial_t^2\eta 
 -
  g
-
\bn^\intercal \big[
\mathbf{B}_{\eta_0}  \overline{\pi}_0
+
 \mathbf{H} -\mathbf{A}_{\eta_0}  \nabx\overline{\bu} \big]\circ\bm{\varphi} \bn
\big)\dy,
\end{align*}
where we used the zero-mean property of $\eta$.
Since $\mathbf{B}_{\eta_0}$ is uniformly elliptic, it follows from the above and Poincar\'e's inequality that
\begin{equation}
\begin{aligned}
\label{317}
\int_I
\Vert  \overline{\pi} \Vert_{W^{1,2}_{\bx}}^2\dt 
&\lesssim
\int_I
\big(
\Vert \nabx\overline{\pi} \Vert_{L^2_{\bx}}^2 
+
\Vert  \overline{\pi}_0 \Vert_{L^2_{\bx}}^2\big)\dt 
+
\int_I
( c_{\overline{\pi}})^2 \dt
\\&\lesssim
\int_I
\big(
\Vert \nabx\overline{\pi} \Vert_{L^2_{\bx}}^2 
+
\Vert \overline{\pi}_0 \Vert_{L^2_{\bx}}^2\big)\dt  
+
\int_I\int_\omega
\big(
\vert \partial_t^2 \eta\vert^2 
+
\vert  g\vert^2 
\big)\dy\dt
 \\&\quad+ 
  \int_I
 \big(\Vert \pi_0\Vert_{W^{1,2}_{\bx}}^2 
 +
 \Vert \mathbf{H}\Vert_{W^{1,2}_{\bx}}^2 + \Vert  \nabx\overline{\bu}\Vert_{W^{1,2}_{\bx}}^2
 \big)\dt,
\end{aligned}
\end{equation}
where we have used the trace theorem,
\begin{equation*}
\begin{aligned} 
& \int_I
 \big(\Vert \pi_0\Vert_{L^2(\partial\Omega)}^2 
 +
 \Vert \mathbf{H}\Vert_{L^2(\partial\Omega)}^2 + \Vert \nabx\overline{\bu}\Vert_{L^2(\partial\Omega)}^2
 \big)\dt
 \\&
 \lesssim
 \int_I
 \big(\Vert \pi_0\Vert_{W^{1,2}_{\bx}}^2 
 +
 \Vert \mathbf{H}\Vert_{W^{1,2}_{\bx}}^2 + \Vert  \nabx\overline{\bu}\Vert_{W^{1,2}_{\bx}}^2
 \big)\dt.
\end{aligned}
\end{equation*}
Collecting the inequalities from \eqref{3.8}, \eqref{3.11}, \eqref{3.16} and \eqref{317}, we conclude that
\begin{equation}
\begin{aligned}
\label{319}
&\sup_I\int_\omega
\big(\vert \partial_t \naby \eta\vert^2 
+
\vert  \naby\Dely \eta\vert^2
\big)
\dy
+
\sup_I\int_\Omega\vert  \nabx \overline{\bu}\vert^2\dx
\\&\quad+
\int_I\int_\omega
\big(\vert \partial_t \Dely \eta \vert^2 
+
\vert \partial_t \naby \eta\vert^2
+ \vert \partial_t^2 \eta\vert^2
 \big)\dy\dt
 +
\int_I\int_\Omega \vert \partial_t \overline{\bu} \vert^2  
  \dx\dt
 \\&\lesssim
 \int_\omega\big( \vert \eta_*\vert^2
 +
 \vert \naby\eta_*\vert^2
 +
 \vert \Dely\eta_0\vert^2
 +
  \vert \naby\Dely\eta_0\vert^2
  \big)\dy
  +
  \int_\Omega \big(\vert
  \overline{\bu}_0\vert^2
  +
   \vert\nabx\overline{\bu}_0\vert^2 \big)\dx
   \\& 
   \quad+
 \int_I\int_\omega  \vert  g\vert^2  
 \dy\dt
 +
  \int_I\int_\Omega\big(
  \vert  h\vert^2
  +
   \vert  \mathbf{h}\vert^2 +
  \vert  \mathbf{H}\vert^2
  +
  \vert  \nabx\mathbf{H}\vert^2
  \big)
 \dx\dt
 \\&
 \quad+
 \int_I\Vert \partial_th\Vert_{W^{-1,2}_\bx}^2\dt
 +
 \kappa
 \int_I
 \big(\Vert \pi_0\Vert_{W^{1,2}_{\bx}}^2 
  + \Vert   \nabx\overline{\bu}\Vert_{W^{1,2}_{\bx}}^2
  \big)\dt.
\end{aligned}
\end{equation}
Our next goal is to estimate the $\kappa$-terms above to get \eqref{energyEstLinear}. For this, we transform
\eqref{contEqAloneBarLinear} and \eqref{momEqAloneBarLinear} by applying $\bm{\Psi}_{\eta_0}^{-1}$ to them. By setting $\underline{\bu}:=\overline{\bu} \circ \bm{\Psi}_{\eta_0}^{-1}$ and $\underline{\pi}:=\overline{\pi} \circ \bm{\Psi}_{\eta_0}^{-1}$, we obtain
\begin{equation*}
	\left\{\begin{aligned}
&\divx \underline{\bu}=h\circ \bm{\Psi}_{\eta_0}^{-1},
\\&
 \partial_t \underline{\bu}  - \Delx\underline{\bu}
 + \nabx \underline{\pi} 
 = 
 J_{\eta_0}^{-1}\big(
 \mathbf{h}-
\divx  \mathbf{H} \big)\circ \bm{\Psi}_{\eta_0}^{-1},
\end{aligned}\right.
\end{equation*}
in $I \times \Omega_{\eta_0}$ with  $\underline{\bu}  \circ \bm{\varphi}_{\eta_0}  =(\partial_t\eta)\bn$   on $I\times \omega$. Based on the maximal regularity theorem for the classical unsteady Stokes system (Please refer to, for instance, \cite{solonnikov1977}), we obtain
\begin{equation*}
\begin{aligned}
&	\int_I\int_{\Omega_{\eta_0}}\big(\vert
 \partial_t \underline{\bu}\vert^2+ \vert \nabx^2\underline{\bu}
 \vert^2
 +
 \vert
 \nabx \underline{\pi} \vert^2\big)\dx\dt\\
 &\lesssim
 \int_I\Vert \partial_t \eta\Vert_{W^{3/2,2}_\by}^2\dt
 +
 \int_I\int_{\Omega_{\eta_0}}
\vert
\nabx( h\circ \bm{\Psi}_{\eta_0}^{-1})
\vert^2\dx\dt
 \\
 &\quad+
 \int_I\int_{\Omega_{\eta_0}}
 \big(\vert
 \mathbf{h}\circ \bm{\Psi}_{\eta_0}^{-1}
\vert^2
+
\vert(\divx  \mathbf{H} )\circ \bm{\Psi}_{\eta_0}^{-1}\vert^2\big)\dx\dt +\int_{\Omega_{\eta_0}}|\nabla\underline \bu_0|^2\dx .
\end{aligned}
\end{equation*}
We now transform back to $\Omega$, 
 interpolate the regularity for the shell and use \eqref{equivNorm}, we obtain for any $\kappa>0$,
\begin{equation*}
\begin{aligned}
	&	\int_I\int_{\Omega}\big(\vert
 \partial_t \overline{\bu}\vert^2+ \vert \nabx^2\overline{\bu}
 \vert^2
 +
 \vert
  \nabx \overline{\pi} \vert^2\big)\dx\dt\\
  & \leq
 \kappa
 \int_I\Vert \partial_t\Dely \eta\Vert_{L^2_\by}^2\dt
 +
 c(\kappa)
 \int_I\Vert \partial_t \eta\Vert_{W^{1,2}_\by}^2\dt
 \\
 &\quad+c
 \int_I\int_{\Omega}
 \big(\vert\nabx h\vert^2 + 
 \vert
 \mathbf{h}
\vert^2
+
\vert    \nabx\mathbf{H}  \vert^2\big)\dx\dt 
+
c\int_{\Omega }|\nabla\overline \bu_0|^2\dx.
\end{aligned}
\end{equation*}
If we now combine this with \eqref{319}, then we obtain the desired estimate \eqref{energyEstLinear} (note that one can finally control $\Dely^2\eta$ by means of equation \eqref{shellEqAloneBarLinear}).
\end{proof}

\subsection{Fixed-point argument}
Based on Proposition \ref{thm:transformedSystemLinear}, 
in this section, we study the existence of the solution of the nonlinear system \eqref{contEqAloneBar}--\eqref{momEqAloneBar}, by employing the Banach fixed-point argument. 
We assume that the triplet $(\zeta, \overline{\mathbf{w}}, \overline{q})$ are given and we wish to solve
\begin{align}
\label{contEqAloneBarFixed}
\mathbf{B}_{\eta_0}:\nabx \overline{\bu}= h_\zeta(\overline{\mathbf{w}}),
\\
\partial_t^2\eta - \partial_t\Dely \eta +  \Dely^2\eta
=
g+\bn^\intercal \big[\mathbf{H}_\zeta(\overline{\mathbf{w}}, \overline{q})-\mathbf{A}_{\eta_0}  \nabx\overline{\bu} +\mathbf{B}_{\eta_0}\overline{\pi}\big]\circ\bm{\varphi} \bn ,
\label{shellEqAloneBarFixed}
\\
J_{\eta_0}\partial_t \overline{\bu}  -\divx(\mathbf{A}_{\eta_0}  \nabx\overline{\bu}) 
 +\divx(\mathbf{B}_{\eta_0}\overline{\pi}) 
 = 
\mathbf{h}_\zeta(\overline{\mathbf{w}})-
\divx  \mathbf{H}_\zeta(\overline{\mathbf{w}}, \overline{q})
\label{momEqAloneBarFixed}
\end{align}
with  $\overline{\bu}  \circ \bm{\varphi}  =(\partial_t\eta)\bn$ on $I_*\times \omega$. Here, $I_*:=(0,T_*)$ is to be determined later. We define the space
\begin{align*}
X_{I_*}:=&\left(W^{1,\infty}\big(I_*;W^{1,2}(\omega)  \big)  \cap L^{\infty}\big(I_*;W^{3,2}(\omega)  \big) \cap W^{1,2}\big(I_*;W^{2,2}(\omega)  \big)
\cap W^{2,2}\big(I_*;L^{2}(\omega)  \big)\right)
\\&
\times
\left(L^\infty \big(I_*; W^{1,2}(\Omega ) \big)\cap  W^{1,2}\big(I_*;L^{2}(\Omega)  \big)\cap L^{2}\big(I_*;W^{2,2}(\Omega)  \big)\right)
\times
L^2\big(I_*;W^{1,2}(\Omega)  \big),
\end{align*}
equipped with the norm
\begin{align*}
\Vert (\zeta,\overline{\mathbf{w}}, \overline{q}) \Vert_{X_{I_*}}^2
&:=
\sup_{I_*}\int_\omega\big( \vert \partial_t\zeta\vert^2 +  \vert \partial_t\naby\zeta\vert^2 +  \vert\Dely\zeta \vert^2 +  \vert\naby\Dely\zeta \vert^2 \big)\dy
\\&\quad+
\int_{I_*}\int_\omega\big( \vert \partial_t\naby\zeta\vert^2 +  \vert \partial_t\Dely\zeta\vert^2 +  \vert\partial_t^2\zeta \vert^2 \big)\dy\dt
\\&
\quad+\sup_{I_*}\int_\Omega\big( \vert \overline{\mathbf{w}}\vert^2 +  \vert \nabx \overline{\mathbf{w}} \vert^2  \big)\dx
\\&\quad+
\int_{I_*}\int_\Omega\big( \vert \nabx \overline{\mathbf{w}}\vert^2 +  \vert \nabx^2 \overline{\mathbf{w}} \vert^2 +  \vert\partial_t  \overline{\mathbf{w}} \vert^2
+\vert \overline{q}\vert^2 +\vert \nabx\overline{q}\vert^2 \big)\dx\dt.
\end{align*}

Let $B_R^{X_{I_*}}$ be a ball defined as
\begin{align*}
B_R^{X_{I_*}}:= \big\{ (\zeta,\overline{\mathbf{w}}, \overline{q})\in X_{I_*}
\text{ with } \quad \zeta(0)=\eta_0, \quad \partial_t\zeta(0) = \eta_*,\quad \overline{\mathbf{w}}(0)=\overline \bu_0  \, :\, \Vert(\zeta,\overline{\mathbf{w}}, \overline{q})\Vert_{X_{I_*}}^2\leq R \big\},
\end{align*}
for some $R>0$ large enough. 

\begin{theorem}\label{fixedpoint}
	There exists a time $T>0$ and $R>0$, such that the map $\mathcal{T}$, defined by 
	\begin{equation*}
	\begin{aligned}
	\mathcal{T}:B_R^{X_{I_*}}&\rightarrow B_R^{X_{I_*}}\\
	(\zeta, \overline{\mathbf{w}}, \overline{q})&\mapsto(\eta,\overline{\bu}, \overline{\pi}),
	\end{aligned}
	\end{equation*}
 is a contraction map, which thereby possesses a fixed point in $X_{I_*}$.
\end{theorem}
\begin{proof}
We would like to show that the map $\mathcal{T}$ defined above maps the ball $B_R^{X_{I_*}} $ into itself and that for any $(\zeta_i, \overline{\mathbf{w}}_i, \overline{q}_i)\in B_R^{X_{I_*}}$, for $i=1,2$, we can find $\rho<1$ such that
\begin{align*}
	\Vert \mathcal{T}(\zeta_1, \overline{\mathbf{w}}_1, \overline{q}_1)
	-
	\mathcal{T}(\zeta_2, \overline{\mathbf{w}}_2, \overline{q}_2)\Vert_{X_{I_*}}\leq \rho \Vert (\zeta_1, \overline{\mathbf{w}}_1, \overline{q}_1)-(\zeta_2, \overline{\mathbf{w}}_2, \overline{q}_2)\Vert_{X_{I_*}}.
\end{align*}
To present the proof clearly, we divide it in the following two steps.
\\ \\
{\bf Step 1:} {\em We show that $\mathcal{T}:B_R^{X_{I_*}} \rightarrow  B_R^{X_{I_*}}$, i.e. the ball $B_R^{X_{I_*}}$ is $\mathcal{T}$-invariant.}
To do this, we need to show that for any $(\zeta, \overline{\mathbf{w}}, \overline{q}) \in B_R^{X_{I_*}}$, we have 
\begin{align}
\label{ballToBall}
\Vert\mathcal{T}(\zeta, \overline{\mathbf{w}}, \overline{q})\Vert_{X_{I_*}}^2
=
\Vert(\eta, \overline{\bu}, \overline{\pi})\Vert_{X_{I_*}}^2
\leq R.
\end{align}
Indeed, according to \eqref{energyEstLinear}, we deduce the following estimate from \eqref{contEqAloneBarFixed}--\eqref{momEqAloneBarFixed} 
\begin{equation}
\begin{aligned}
\label{energyEstLinearForBall}
\Vert(\eta, \overline{\bu}, \overline{\pi})\Vert_{X_{I_*}}^2
&\lesssim
\mathcal{D}(g, \eta_0, \eta_*,\overline{\bu}_0)
 +
   \int_{I_*}\big(\Vert \partial_t h_{\zeta}(\overline{\mathbf{w}}) \Vert_{W^{-1,2}(\Omega)}^2\dt
 +
  \Vert h_{\zeta}(\overline{\mathbf{w}})\Vert_{W^{1,2}(\Omega)}^2
  \big)
 \dt
 \\&
\qquad  
+
\int_{I_*}\big(
   \Vert \mathbf{h}_{\zeta}(\overline{\mathbf{w}})\Vert_{L^{2}(\Omega)}^2 +
  \Vert \mathbf{H}_{\zeta}(\overline{\mathbf{w}}, \overline{q})\Vert_{W^{1,2}(\Omega)}^2
  \big)
 \dt
 \\&
 =:\mathcal{D}(g, \eta_0,\eta_*,\overline{\bu}_0)+\overline{K}_1+\overline{K}_2+\overline{K}_3+\overline{K}_4,
\end{aligned}
\end{equation}
where
\begin{equation*}
\begin{aligned}
		\mathcal{D}(g, \eta_0,\eta_*,\overline{\bu}_0)
		&=
		\Vert \eta_*\Vert_{W^{1,2}(\omega)}^2
		+
		\Vert \eta_0\Vert_{W^{3,2}(\omega)}^2
		+
		\Vert
		\overline{\bu}_0\Vert_{W^{1,2}(\Omega)}^2
		+
		\int_{I_*} \Vert g\Vert_{L^2(\omega)}^2  \dt.
\end{aligned}
\end{equation*}
Recalling the regularity assumption in \eqref{datasetImproved}, we choose $R>0$ large enough, such that
\begin{align}
\label{cDataR}
c \mathcal{D}(g, \eta_0, \eta_*,\overline{\bu}_0)
\leq \frac{1}{2}R,
\end{align}
where $c>0$ is the constant in \eqref{energyEstLinearForBall}. We show in what follows that the sum of the $\overline{K}_i$ is also bounded by $\frac{1}{2}R$ which will give the estimate \eqref{ballToBall}.
 
To estimate $\overline{K}_1$, we have
\begin{align*}
\partial_t h_{\zeta}(\overline{\mathbf{w}} )
&=
 (\mathbf{B}_{\eta_0}-\mathbf{B}_{\zeta}):\partial_t\nabx\overline{\mathbf{w}} 
 -
  (\partial_t\mathbf{B}_{\zeta}):\nabx\overline{\mathbf{w}}.
\end{align*}
We notice that the continuous embedding 
\begin{align}\label{eq:emb1}
L^\infty(I_*;W^{3,2}(\omega))\cap W^{1,2}(I_*;W^{2,2}(\omega))\hookrightarrow C^{0, 1/8}(I_*; W^{11/4, 2}(\omega))\hookrightarrow
L^\infty(I_*;W^{1,\infty}(\omega)),
\end{align}
scales with $T_*^{1/8}$. 
We thereby obtain from \eqref{210and212}--\eqref{218} that
\begin{equation}
\begin{aligned}
\label{k1ax}
\int_{I_*}\Vert (\mathbf{B}_{\eta_0}-\mathbf{B}_{\zeta}):\partial_t\nabx \overline{\mathbf{w}} \Vert_{W^{-1,2}_{\bx}}
^2\dt
&\lesssim
\int_{I_*}\Vert  \mathbf{B}_{\eta_0}-\mathbf{B}_{\zeta}\Vert_{L^\infty_\bx}^2\Vert\partial_t \overline{\mathbf{w}} \Vert_{L^2_\bx}^2
\dt
\\&\lesssim 
\sup_{I_*}\Vert  \eta_0 - \zeta\Vert_{W^{1,\infty}_\by}^2\int_{I_*}\Vert  \partial_t \overline{\mathbf{w}}  \Vert_{L^2_\bx}^2
\dt
\\&
\lesssim
T^{1/4}_*
\Vert (\zeta, \overline{\mathbf{w}}, \overline{q})\Vert_{X_{I_*}}^2.
\end{aligned}
\end{equation}
On the other hand, due to the continuous embedding
\begin{align*}
W^{2,2}(I_*;L^2(\omega))\cap W^{1,2}(I_*;W^{2,2}(\omega))\hookrightarrow W^{5/4,2}(I_\ast;W^{3/2,2}(\omega)) \hookrightarrow
W^{1,4}(I_*;W^{1,4}(\omega)),
\end{align*}
it follows from H\"older inequality that
\begin{equation}
\begin{aligned}
\label{k1cx}
\int_{I_*}\Vert\partial_t \mathbf{B}_{\zeta}:\nabx \overline{\mathbf{w}}\Vert_{W^{-1,2}_\bx}^2
\dt
&\lesssim
\int_{I_*}\Vert\partial_t \mathbf{B}_{\zeta}:\nabx \overline{\mathbf{w}} \Vert_{L^{4/3}_\bx}^2
\dt
\\&\lesssim
\int_{I_*}\Vert  \partial_t\mathbf{B}_{\zeta}\Vert^2_{L^4_\bx}\Vert\nabx \overline{\mathbf{w}} \Vert_{L^2_\bx}^2
\dt
\\&\lesssim T^{1/2}_*
\bigg(\int_{I_*} \Vert  \partial_t\zeta\Vert_{W^{1,4}_\by}^4\dt\bigg)^\frac{1}{2}
\sup_{I_*}\Vert  \nabx\overline{\mathbf{w}}\Vert_{L^2_\bx}^2
\\&
\lesssim
T^{1/2}_*
\Vert (\zeta, \overline{\mathbf{w}}, \overline{q})\Vert_{X_{I_*}}^2.
\end{aligned}
\end{equation} 
Combining \eqref{k1ax} with \eqref{k1cx}, we have
\begin{equation}
\begin{aligned}
\label{k1finalx}
\overline{K}_1
\lesssim T^{1/2}_*
\Vert (\zeta, \overline{\mathbf{w}}, \overline{q}) \Vert_{X_{I_*}}^2.
\end{aligned}
\end{equation}

To estimate $\overline{K}_2$, we note that
\begin{equation*}
	\begin{aligned}
\vert
 h_{\zeta}(\overline{\mathbf{w}})
\vert
+
\vert
\nabx h_{\zeta}(\overline{\mathbf{w}})
\vert
&\lesssim
\vert (\mathbf{B}_{\eta_0}-\mathbf{B}_{\zeta}) :\nabx \overline{\mathbf{w}}
\vert
+
\vert (\mathbf{B}_{\eta_0}-\mathbf{B}_{\zeta}):\nabx^2 \overline{\mathbf{w}}
\vert
\\&
\quad+
\vert \nabx(\mathbf{B}_{\eta_0}-\mathbf{B}_{\zeta}):\nabx \overline{\mathbf{w}}
\vert.
\end{aligned}
\end{equation*}
Using  the argument in \eqref{eq:emb1} again,
we derive from \eqref{210and212}--\eqref{218} that
\begin{equation}
\begin{aligned}
\label{K2ax}
&\int_{I_*}\Vert   (\mathbf{B}_{\eta_0}-\mathbf{B}_{\zeta}) :\nabx \overline{\mathbf{w}} \Vert_{L^2_\bx}^2\dt
+
\int_{I_*}\Vert   (\mathbf{B}_{\eta_0}-\mathbf{B}_{\zeta}): \nabx^2  \overline{\mathbf{w}} \Vert_{L^2_\bx}^2\dt
\\&\lesssim
\sup_{I_*}\Vert    \eta_0 - \zeta \Vert_{W^{1,\infty}_\by}^2
\int_{I_*} \Vert\nabx^2 \overline{\mathbf{w}} \Vert_{L^2_\bx}^2\dt
\\&\lesssim
T^{1/4}_*
\Vert (\zeta, \overline{\mathbf{w}}, \overline{q}) \Vert_{X_{I_*}}^2.
\end{aligned}
\end{equation}
According to the continuous embeddings:
\begin{align}
&L^\infty(I_*;W^{3,2}(\omega)) \hookrightarrow
L^\infty(I_*;W^{2,4}(\omega)), 
\label{eq:emb2}
\end{align}
and
\begin{align}
\begin{aligned}
\,W^{1,2}(I_*;L^2(\Omega))\cap L^2(I_*;W^{2,2}(\Omega))
&\hookrightarrow W^{1/8,2}(I_*;W^{7/4,2}(\Omega))
\hookrightarrow
L^2(I_*;W^{1,4}(\Omega)),
\end{aligned}
\label{eq:emb3}
\end{align}
where the latter scales with $T^{1/8}$,
it follows that
\begin{equation}
\begin{aligned}
\label{K2bx}
\int_{I_*}\Vert  \nabx (\mathbf{B}_{\eta_0}-\mathbf{B}_{\zeta}) \nabx \overline{\mathbf{w}} \Vert_{L^2_\bx}^2\dt
&\lesssim
\sup_{I_*}\Vert   \eta_0 - \zeta \Vert_{W^{2,4}_\by}^2
\int_{I_*} \Vert\nabx  \overline{\mathbf{w}}\Vert_{L^4_\bx}^2\dt
\\&\lesssim
T^{1/4}_*
\Vert (\zeta, \overline{\mathbf{w}}, \overline{q}) \Vert_{X_{I_*}}^2.
\end{aligned}
\end{equation}
We obtain from \eqref{K2ax} and \eqref{K2bx} that
\begin{align}
\label{k2finalx}
\overline{K}_2
\lesssim
 T^{1/4}_* 
\Vert (\zeta,\overline{\mathbf{w}}, \overline{q}) \Vert_{X_{I_*}}^2.
\end{align} 

To estimate $\overline{K}_3$, let us recall that
\begin{align*}
\mathbf{h}_{\zeta}(\overline{\mathbf{w}})
=
(J_{\eta_0}-J_{\zeta})\partial_t \overline{\mathbf{w}}
+
J_{\zeta} \nabx \overline{\mathbf{w}}  \partial_t \Psi_{\zeta}^{-1}\circ \Psi_{\zeta} 
+
\mathbf{B}_\zeta \nabx \overline{\mathbf{w}}\,\overline{\mathbf{w}}
+
J_{\zeta} \bff\circ \Psi_{\zeta}.
\end{align*}
Due to the continuous embeddings \eqref{eq:emb1}, 
it follows from the definition $J_\eta=\det(\nabx \Psi_\eta)$ and \eqref{210and212}--\eqref{218}  that
\begin{equation}
\begin{aligned}
\label{K3ax}
\int_{I_*}\Vert   (J_{\eta_0}-J_{\zeta})\partial_t \overline{\mathbf{w}}\Vert_{L^2_\bx}^2\dt
&\lesssim
\sup_{I_*}\Vert  \eta_0 - \zeta\Vert_{W^{1,\infty}_\by}
\int_{I_*}\Vert \partial_t\overline{\mathbf{w}} \Vert_{L^2_\bx}^2\dt
\\&\lesssim
T^{1/4}_*
\Vert (\zeta,\overline{\mathbf{w}}, \overline{q})  \Vert_{X_{I_*}}^2.
\end{aligned}
\end{equation}
By using the embeddings
\begin{align*}
&L^\infty(I_*;W^{3,2}(\omega))\hookrightarrow L^\infty(I_*;W^{1,\infty}(\omega)),
\\&
W^{1,\infty}(I_*;W^{1,2}(\omega))
\cap
W^{1,2}(I_*;W^{2,2}(\omega))\hookrightarrow W^{1,4}(I_*; W^{3/2, 2}(\omega)) \hookrightarrow
W^{1,4}(I_*;L^{\infty}(\omega)),
\end{align*}
we obtain that
\begin{equation}
\begin{aligned}
\label{K3bx}
&\int_{I_*}\Vert   J_{\zeta}  \nabx \overline{\mathbf{w}} \partial_t \Psi_{\zeta}^{-1}\circ \Psi_{\zeta} \Vert_{L^2_\bx}^2\dt
\\&\lesssim
\int_{I_*}\big(1+\Vert   \zeta
\Vert_{W^{1,\infty}_\by}^2
\big)
\Vert \nabx \overline{\mathbf{w}} \Vert_{L^2_\bx}^2
\Vert \partial_t\zeta \Vert_{L^\infty_\by}^2 \dt
\\&
\lesssim
T^{1/2}_*
\sup_{I_*}\big(1+\Vert   \zeta
\Vert_{W^{1,\infty}_\by}^2
\big)
\sup_{I_*}
\Vert \nabx \overline{\mathbf{w}} \Vert_{L^2_\bx}^2
\bigg(\int_{I_*}
\Vert \partial_t\zeta \Vert_{L^\infty_\by}^4
\dt\bigg)^\frac{1}{2}
\\&
\lesssim
T^{1/2}_*
\Vert (\zeta,\overline{\mathbf{w}}, \overline{q})  \Vert_{X_{I_*}}^2.
\end{aligned}
\end{equation}
Also, by using the embedding \eqref{eq:emb3}
we obtain
\begin{equation}
\begin{aligned}
\label{K3dx}
\int_{I_*}\Vert   \mathbf{B}_{\zeta} \nabx \overline{\mathbf{w}}\, \overline{\mathbf{w}}_1  \Vert_{L^2_\bx}^2\dt
&\lesssim
\int_{I_*}\Vert   \zeta  \Vert_{W^{1,\infty}_\by}^2
\Vert
\nabx \overline{\mathbf{w}} \Vert_{L^4_\bx}^2
\Vert \overline{\mathbf{w}}  \Vert_{L^4_\bx}^2\dt
\\&\lesssim 
\sup_{I_*}\Vert   \zeta  \Vert_{W^{3,2}_\by}^2
\sup_{I_*}
\Vert \overline{\mathbf{w}}  \Vert_{W^{1,2}_\bx}^2
\int_{I_*}\Vert  \nabx\overline{\mathbf{w}} \Vert_{L^{4}_\bx}^2
\dt
\\&
\lesssim
T^{1/4}_*
\Vert (\zeta,\overline{\mathbf{w}}, \overline{q})  \Vert_{X_{I_*}}^2.
\end{aligned}
\end{equation}
Next, by using \eqref{eq:emb1}, we obtain
\begin{equation}
\begin{aligned}
\label{K3fx}
\int_{I_*}\Vert 
J_{\zeta} \bff\circ \Psi_{\zeta} 
 \Vert_{L^2_\bx}^2\dt
&
\lesssim 
\sup_{I_*}
\Vert
\zeta  
\Vert_{W^{1,\infty}_\by}^2\int_{I_*} \Vert \bff \Vert_{L^2_\bx}^2
\dt
\\&
\lesssim
T^{1/4}_*
\Vert (\zeta,\overline{\mathbf{w}}, \overline{q})  \Vert_{X_{I_*}}^2.
\end{aligned}
\end{equation}
It follows from \eqref{K3ax}--\eqref{K3fx} that
\begin{equation}
\begin{aligned}
\label{k3finalx}
\overline{K}_3
\lesssim  T^{1/2}_* 
\Vert (\zeta, \overline{\mathbf{w}}, \overline{q})\Vert_{X_{I_*}}^2.
\end{aligned}
\end{equation}

Our next goal is to estimate $\overline{K}_4$. Since
\begin{equation}
\begin{aligned}
\nonumber
\mathbf{H}_{\zeta}(\overline{\mathbf{w}}, \overline{q})
=
(\mathbf{A}_{\eta_0} -\mathbf{A}_{\zeta})\nabx\overline{\mathbf{w}}
+
(\mathbf{B}_{\eta_0}-\mathbf{B}_{\zeta}) \overline{q},
\end{aligned}
\end{equation}
due to the continuous embeddings  \eqref{eq:emb1}, \eqref{eq:emb3} and \eqref{eq:emb2},
it follows from \eqref{210and212}--\eqref{218} that
\begin{equation}
\begin{aligned}
\label{K4ax}
&\int_{I_*}\Vert (\mathbf{A}_{\eta_0} -\mathbf{A}_{\zeta})\nabx \overline{\mathbf{w}} \Vert_{W^{1,2}_\bx}^2\dt\\
&\lesssim
\int_{I_*}\Vert \nabx(\mathbf{A}_{\eta_0} -\mathbf{A}_{\zeta})  \Vert_{L^4_\bx}^2\Vert\nabx\overline{\mathbf{w}} \Vert_{L^4_\bx}^2\dt
+
\int_{I_*}\Vert  \mathbf{A}_{\eta_0} -\mathbf{A}_{\zeta}\Vert_{L^\infty_\bx}^2
\Vert\nabx^2\overline{\mathbf{w}} \Vert_{L^2_\bx}^2\dt
\\&\lesssim 
\sup_{I_*}\Vert  \eta_0 - \zeta \Vert_{W^{2,4}_\by}^2
\int_{I_*} \Vert\nabx \overline{\mathbf{w}} \Vert_{L^4_\bx}^2\dt
+
\sup_{I_*}\Vert  \eta_0 - \zeta \Vert_{W^{1,\infty}_\by}^2
\int_{I_*}\Vert\nabx^2\overline{\mathbf{w}}\Vert_{L^2_\bx}^2\dt
\\&\lesssim
 T^{1/4}_* 
\Vert (\zeta, \overline{\mathbf{w}}, \overline{q})\Vert_{X_{I_*}}^2.
\end{aligned}
\end{equation}
Next, we use the embedding
\begin{align*}
L^2(I_*;W^{3,2}(\omega))\cap W^{1,2}(I_*;W^{2,2}(\omega))\hookrightarrow W^{2/3,2}(I_*;W^{7/3,2}(\omega))
\hookrightarrow L^\infty(I_*;W^{2,3}(\omega)),
\end{align*}
where the latter scales with $T_*^{1/6}$,
and $W^{2,3}(\omega)\hookrightarrow W^{1,\infty}(\omega)$
to obtain
\begin{equation}
\begin{aligned}
\label{K4bx}
&\int_{I_*}\Vert (\mathbf{B}_{\eta_0}-\mathbf{B}_{\zeta}) \overline{q} \Vert_{W^{1,2}_\bx}^2\dt\\
&\lesssim
\int_{I_*}\Vert \nabx(\mathbf{B}_{\eta_0}-\mathbf{B}_{\zeta}) \Vert_{L^3_\bx}^2
\Vert
\overline{q} \Vert_{L^6_\bx}^2\dt+
\int_{I_*}\Vert \mathbf{B}_{\eta_0}-\mathbf{B}_{\zeta} \Vert_{L^\infty_\bx}^2
\Vert
\nabx\overline{q} \Vert_{L^2_\bx}^2\dt
\\
&\lesssim 
\sup_{I_*}\Vert  \eta_0 -  \zeta \Vert_{W^{2,3}_\by}^2
\int_{I_*}
\Vert
\overline{q} \Vert_{W^{1,2}_\bx}^2\dt
+
\sup_{I_*}\Vert  \eta_0 - \zeta \Vert_{W^{1,\infty}_\by}^2
\int_{I_*}
\Vert
 \overline{q} \Vert_{W^{1,2}_\bx}^2\dt
 \\&\lesssim
T^{1/3}_*
\Vert (\zeta, \overline{\mathbf{w}}, \overline{q})\Vert_{X_{I_*}}^2.
\end{aligned}
\end{equation}
By using \eqref{K4ax} and \eqref{K4bx}, it follows that
\begin{align}
\label{k4finalx}
\overline{K}_4
\lesssim
T^{1/3}_*
\Vert (\zeta,\overline{\mathbf{w}}, \overline{q})  \Vert_{X_{I_*}}^2.
\end{align}
Collecting the estimates \eqref{k1finalx}, \eqref{k2finalx}, \eqref{k3finalx} and \eqref{k4finalx}, we have shown that
\begin{align}
\label{contractionEstx}
\Vert \mathcal{T}(\zeta, \overline{\mathbf{w}}, \overline{q})
\Vert_{X_{I_*}}^2
=
\Vert(\eta, \overline{\bu}, \overline{\pi})\Vert_{X_{I_*}}^2
&\leq 
c  T^{1/2}_*  
\Vert (\zeta,\overline{\mathbf{w}}, \overline{q})  \Vert_{X_{I_*}}^2.
\end{align}
Since $(\zeta,\overline{\mathbf{w}}, \overline{q})\in B_R^{X_{I_*}}$, by choosing $T_*$ in $I_*=(0,T_*)$ so that $ T^{1/2}_* <\tfrac{c^{-1}}{2}$, we find that the right-hand side of \eqref{contractionEstx} is bounded by $R/2$. This, together with \eqref{cDataR}, implies \eqref{ballToBall}.
\\ \\
{\bf Step 2:} {\em We prove that $\mathcal{T}$ is a contraction map.} To show the contraction property, we denote $\eta_{12}:=\eta_1-\eta_2$, $\overline{\bu}_{12}:=\overline{\bu}_1-\overline{\bu}_2$ and $\overline{\pi}_{12}:=\overline{\pi}_1-\overline{\pi}_2$. 
We derive the system that $(\eta_{12}, \overline \bu_{12}, \overline \pi_{12})$ satisfies, which reads
\begin{align}
\label{contEqAloneBarFixed12}
\mathbf{B}_{\eta_0}:\nabx \overline{\bu}_{12}= h_{\zeta_1}(\overline{\mathbf{w}}_1)- h_{\zeta_2}(\overline{\mathbf{w}}_2),
\end{align}
and
\begin{equation}
\begin{aligned}
\varrho_s\partial_t^2\eta_{12} -\gamma\partial_t\Dely \eta_{12} + \alpha\Dely^2\eta_{12}
&=
\bn^\intercal \big[
-\mathbf{A}_{\eta_0}  \nabx\overline{\bu}_{12} +\mathbf{B}_{\eta_0}\overline{\pi}_{12}\big]\circ\bm{\varphi} \bn 
\\&
\quad+
\bn^\intercal \big[\mathbf{H}_{\zeta_1}(\overline{\mathbf{w}}_1, \overline{q}_1)-\mathbf{H}_{\zeta_2}(\overline{\mathbf{w}}_2, \overline{q}_2)
\big]\circ\bm{\varphi} \bn ,
\label{shellEqAloneBarFixed12}
\end{aligned}
\end{equation}
and
\begin{equation}
\begin{aligned}
J_{\eta_0}\partial_t \overline{\bu}_{12}  -\divx(\mathbf{A}_{\eta_0}  \nabx\overline{\bu}_{12}) 
 &+\divx(\mathbf{B}_{\eta_0}\overline{\pi}_{12}) 
 = 
\mathbf{h}_{\zeta_1}(\overline{\mathbf{w}}_1)-\mathbf{h}_{\zeta_2}(\overline{\mathbf{w}}_2)
\\&-
\divx  \mathbf{H}_{\zeta_1}(\overline{\mathbf{w}}_1, \overline{q}_1)
+
\divx  \mathbf{H}_{\zeta_2}(\overline{\mathbf{w}}_2, \overline{q}_2).
\label{momEqAloneBarFixed12}
\end{aligned}
\end{equation}
Since the left-hand side of \eqref{contEqAloneBarFixed12}, \eqref{shellEqAloneBarFixed12} and \eqref{momEqAloneBarFixed12} are all linear as functions of $(\eta_{12},\overline{\bu}_{12},\overline{\pi}_{12})$, it suffices to estimate their right-hand sides and substitute these estimates into the corresponding right-hand terms in \eqref{energyEstLinear}. Precisely, we estimate the following integrals:
\begin{align*}
&K_1:=
\int_{I_*}\Vert \partial_t [h_{\zeta_1}(\overline{\mathbf{w}}_1)- h_{\zeta_2}(\overline{\mathbf{w}}_2)] \Vert_{W^{-1,2}(\Omega)}^2\dt,
\\&
K_2:=
  \int_{I_*}\Vert h_{\zeta_1}(\overline{\mathbf{w}}_1)- h_{\zeta_2}(\overline{\mathbf{w}}_2)\Vert_{W^{1,2}(\Omega)}^2\dt,
 \\&
K_3:=
  \int_{I_*}\Vert \mathbf{h}_{\zeta_1}(\overline{\mathbf{w}}_1)- \mathbf{h}_{\zeta_2}(\overline{\mathbf{w}}_2)\Vert_{L^2(\Omega)}^2\dt,
 \\&
K_4:=
\int_{I_*}\Vert \mathbf{H}_{\zeta_1}(\overline{\mathbf{w}}_1, \overline{q}_1)- \mathbf{H}_{\zeta_2}(\overline{\mathbf{w}}_2, \overline{q}_2)\Vert_{W^{1,2}(\Omega)}^2\dt.
\end{align*}
The estimates for these $K_i$'s can be obtained in the same manner as their corresponding $\overline{K}_i$'s above when showing that the mapping $\mathcal{T}$ maps the ball into itself. However, we proceed to give a summary of the various estimates.

To estimate $K_1$, we need some preliminary estimates. Recalling the definition of $h_{\eta}(\overline\bu)$ at the beginning of Subsection \ref{Sectionlinear}, we write
\begin{align*}
\partial_t[h_{\zeta_1}(\overline{\mathbf{w}}_1)
-
h_{\zeta_2}(\overline{\mathbf{w}}_2)]
&=
 (\mathbf{B}_{\eta_0}-\mathbf{B}_{\zeta_1}):\partial_t\nabx( \overline{\mathbf{w}}_1
-
\overline{\mathbf{w}}_2)
+
(\mathbf{B}_{\zeta_2}-\mathbf{B}_{\zeta_1}):\partial_t \nabx\overline{\mathbf{w}}_2
\\
&\quad +\partial_t(\mathbf{B}_{\eta_0}-\mathbf{B}_{\zeta_1}):\nabx( \overline{\mathbf{w}}_1
-
\overline{\mathbf{w}}_2)
+
\partial_t(\mathbf{B}_{\zeta_2}-\mathbf{B}_{\zeta_1}):\nabx \overline{\mathbf{w}}_2.
\end{align*}
Now, just as in \eqref{k1ax}, we obtain from
 \eqref{210and212}--\eqref{218} and the continuous embedding \eqref{eq:emb1} that
\begin{equation*}
\begin{aligned}
\label{k1a}
&\int_{I_*}\big(\Vert (\mathbf{B}_{\eta_0}-\mathbf{B}_{\zeta_1}):\partial_t\nabx( \overline{\mathbf{w}}_1
-
\overline{\mathbf{w}}_2) \Vert_{W^{-1,2}_{\bx}}
^2
+
\Vert (\mathbf{B}_{\zeta_1}-\mathbf{B}_{\zeta_2})\partial_t\nabx
\overline{\mathbf{w}}_2 \Vert_{W^{-1,2}_{\bx}}^2
\big)
\dt
\\&
\lesssim
T^{1/4}_*
\Vert (\zeta_1, \overline{\mathbf{w}}_1, \overline{q}_1)-(\zeta_2, \overline{\mathbf{w}}_2, \overline{q}_2)\Vert_{X_{I_*}}^2.
\end{aligned}
\end{equation*}
Also, as in \eqref{k1cx}, we obtain
\begin{equation*}
\begin{aligned}
\label{k1c}
&\int_{I_*}\big(\Vert\partial_t (\mathbf{B}_{\eta_0}-\mathbf{B}_{\zeta_1}):\nabx( \overline{\mathbf{w}}_1
-
\overline{\mathbf{w}}_2) \Vert_{W^{-1,2}_\bx}^2
+
\Vert\partial_t (\mathbf{B}_{\zeta_1}-\mathbf{B}_{\zeta_2}):\nabx
\overline{\mathbf{w}}_2 \Vert_{W^{-1,2}_\bx}^2
\big)
\dt
\\&
\lesssim
T^{1/2}_*
\Vert (\zeta_1, \overline{\mathbf{w}}_1, \overline{q}_1)-(\zeta_2, \overline{\mathbf{w}}_2, \overline{q}_2)\Vert_{X_{I_*}}^2,
\end{aligned}
\end{equation*} 
and thus,
\begin{equation}
\begin{aligned}
\label{k1final}
K_1
\lesssim T^{1/2}_*
\Vert (\zeta_1, \overline{\mathbf{w}}_1, \overline{q}_1)-(\zeta_2, \overline{\mathbf{w}}_2, \overline{q}_2)\Vert_{X_{I_*}}^2.
\end{aligned}
\end{equation}
To estimate $K_2$, we first note that
\begin{equation*}
	\begin{aligned}
&\vert
 h_{\zeta_1}(\overline{\mathbf{w}}_1)
-
h_{\zeta_2}(\overline{\mathbf{w}}_2)
\vert
+
\vert
\nabx[h_{\zeta_1}(\overline{\mathbf{w}}_1)
-
h_{\zeta_2}(\overline{\mathbf{w}}_2)]
\vert
\\
&\lesssim
\vert (\mathbf{B}_{\eta_0}-\mathbf{B}_{\zeta_1}) :\nabx( \overline{\mathbf{w}}_1
-
\overline{\mathbf{w}}_2)
\vert
+
\vert
(\mathbf{B}_{\zeta_1}-\mathbf{B}_{\zeta_2}) :\nabx\overline{\mathbf{w}}_2
\vert
\\&
\quad+
\vert (\mathbf{B}_{\eta_0}-\mathbf{B}_{\zeta_1}) :\nabx^2( \overline{\mathbf{w}}_1
-
\overline{\mathbf{w}}_2)
\vert
+
\vert
(\mathbf{B}_{\zeta_1}-\mathbf{B}_{\zeta_2}) :\nabx^2\overline{\mathbf{w}}_2
\vert
\\&
\quad+
\vert \nabx(\mathbf{B}_{\eta_0}-\mathbf{B}_{\zeta_1}):\nabx( \overline{\mathbf{w}}_1
-
\overline{\mathbf{w}}_2)
\vert
+
\vert
\nabx(\mathbf{B}_{\zeta_1}-\mathbf{B}_{\zeta_2}):\nabx \overline{\mathbf{w}}_2
\vert.
\end{aligned}
\end{equation*}
Similar to \eqref{K2ax}, we use \eqref{eq:emb1} and \eqref{210and212}--\eqref{218} and obtain
\begin{equation}
\begin{aligned}
\label{K2a}
&\int_{I_*}\big(\Vert   (\mathbf{B}_{\eta_0}-\mathbf{B}_{\zeta_1}) :\nabx( \overline{\mathbf{w}}_1
-
\overline{\mathbf{w}}_2)\Vert_{L^2_\bx}^2
+
\Vert  (\mathbf{B}_{\zeta_1}-\mathbf{B}_{\zeta_2}) :\nabx\overline{\mathbf{w}}_2\Vert_{L^2_\bx}^2
\big)\dt
\\
&\quad+
\int_{I_*}\big(\Vert   (\mathbf{B}_{\eta_0}-\mathbf{B}_{\zeta_1}) :\nabx^2( \overline{\mathbf{w}}_1
-
\overline{\mathbf{w}}_2)\Vert_{L^2_\bx}^2
+
\Vert  (\mathbf{B}_{\zeta_1}-\mathbf{B}_{\zeta_2}) :\nabx^2\overline{\mathbf{w}}_2\Vert_{L^2_\bx}^2
\big)\dt
\\&\lesssim
T^{1/4}_*
\Vert (\zeta_1, \overline{\mathbf{w}}_1, \overline{q}_1)-(\zeta_2, \overline{\mathbf{w}}_2, \overline{q}_2)\Vert_{X_{I_*}}^2.
\end{aligned}
\end{equation}
As in \eqref{K2bx}, we obtain
\begin{equation}
\begin{aligned}
\label{K2b}
 &\int_{I_*}\big(\Vert  \nabx (\mathbf{B}_{\eta_0}-\mathbf{B}_{\zeta_1})\nabx( \overline{\mathbf{w}}_1
-
\overline{\mathbf{w}}_2)\Vert_{L^2_\bx}^2
+
\Vert  \nabx(\mathbf{B}_{\zeta_1}-\mathbf{B}_{\zeta_2}) \nabx\overline{\mathbf{w}}_2\Vert_{L^2_\bx}^2
\big)\dt
\\&\lesssim
T^{1/4}_*
\Vert (\zeta_1, \overline{\mathbf{w}}_1, \overline{q}_1)-(\zeta_2, \overline{\mathbf{w}}_2, \overline{q}_2)\Vert_{X_{I_*}}^2.
\end{aligned}
\end{equation}
We obtain from \eqref{K2a} and \eqref{K2b} that
\begin{align}
\label{k2final}
K_2
\lesssim
 T^{1/4}_* 
\Vert (\zeta_1,\overline{\mathbf{w}}_1, \overline{q}_1) - (\zeta_2,\overline{\mathbf{w}}_2, \overline{q}_2) \Vert_{X_{I_*}}^2.
\end{align} 
To estimate $K_3$, we need some preliminary estimates. First of all, note that
\begin{align*}
\vert
\mathbf{h}_{\zeta_1}(\overline{\mathbf{w}}_1)
-
\mathbf{h}_{\zeta_2}(\overline{\mathbf{w}}_2)
\vert
&\lesssim
\vert (J_{\eta_0}-J_{\zeta_1})\partial_t( \overline{\mathbf{w}}_1
-
\overline{\mathbf{w}}_2)
\vert
+
\vert
(J_{\zeta_1}-J_{\zeta_2})\partial_t \overline{\mathbf{w}}_2
\vert
\\&
\quad+
\big\vert
J_{\zeta_1} \nabx (\overline{\mathbf{w}}_1 -\overline{\mathbf{w}}_2 ) \partial_t \Psi_{\zeta_1}^{-1}\circ \Psi_{\zeta_1} 
\big\vert
+
\big\vert
\big(J_{\zeta_1}
-
J_{\zeta_2} \big) \nabx \overline{\mathbf{w}}_2  \partial_t \Psi_{\zeta_1}^{-1}\circ \Psi_{\zeta_1} 
\big\vert
\\&
\quad+
\big\vert
J_{\zeta_2} \nabx \overline{\mathbf{w}}_2  \partial_t \big(\Psi_{\zeta_1}^{-1}\circ \Psi_{\zeta_1} 
-
\Psi_{\zeta_2}^{-1}\circ \Psi_{\zeta_2} \big)
\big\vert
\\&
\quad+
\big\vert
J_{\zeta_1}\big( \nabx \Psi_{\zeta_1}^{-1}\circ \Psi_{\zeta_1}\big)^\intercal \nabx (\overline{\mathbf{w}}_1 -\overline{\mathbf{w}}_2 ) \overline{\mathbf{w}}_1 
\big\vert
\\&
\quad+
\big\vert
\big[J_{\zeta_1}\big( \nabx \Psi_{\zeta_1}^{-1}\circ \Psi_{\zeta_1}\big)^\intercal
-
J_{\zeta_2}\big( \nabx \Psi_{\zeta_2}^{-1}\circ \Psi_{\zeta_2}\big)^\intercal \big] \nabx \overline{\mathbf{w}}_2  \overline{\mathbf{w}}_1 
\big\vert
\\&
\quad+
\big\vert
J_{\zeta_2}\big( \nabx \Psi_{\zeta_2}^{-1}\circ \Psi_{\zeta_2}\big)^\intercal \nabx \overline{\mathbf{w}}_2  (\overline{\mathbf{w}}_1-\overline{\mathbf{w}}_2)
\big\vert
\\&
\quad+
\vert
J_{\zeta_1} (\bff\circ \Psi_{\zeta_1} -
\bff\circ \Psi_{\zeta_2}  )
\vert
+
\vert
(J_{\zeta_1}-J_{\zeta_2}) 
\bff\circ \Psi_{\zeta_2}  
\vert.
\end{align*}
As in \eqref{K3ax}, using \eqref{eq:emb1} and
\eqref{210and212}--\eqref{218} we have
\begin{equation}
\begin{aligned}
\label{K3a}
&\int_{I_*}\big(\Vert   (J_{\eta_0}-J_{\zeta_1})\partial_t( \overline{\mathbf{w}}_1
-
\overline{\mathbf{w}}_2)\Vert_{L^2_\bx}^2
+
\Vert  (J_{\zeta_1}-J_{\zeta_2})\partial_t
\overline{\mathbf{w}}_2\Vert_{L^2_\bx}^2
\big)\dt
\\&\lesssim
T^{1/4}_*
\Vert (\zeta_1,\overline{\mathbf{w}}_1, \overline{q}_1) - (\zeta_2,\overline{\mathbf{w}}_2, \overline{q}_2) \Vert_{X_{I_*}}^2.
\end{aligned}
\end{equation}
Similar to \eqref{K3bx}, we obtain
\begin{equation}
\begin{aligned}
\label{K3b}
&\int_{I_*}\Vert   J_{\zeta_1}  \nabx (\overline{\mathbf{w}}_1 -\overline{\mathbf{w}}_2 ) \partial_t \Psi_{\zeta_1}^{-1}\circ \Psi_{\zeta_1} \Vert_{L^2_\bx}^2\dt
+
\int_{I_*}\big\Vert  (J_{\zeta_1} 
-
J_{\zeta_2}) \nabx \overline{\mathbf{w}}_2  \partial_t \Psi_{\zeta_1}^{-1}\circ \Psi_{\zeta_1} 
\big\Vert_{L^2_\bx}^2\dt
\\&\quad+
\int_{I_*}\big\Vert  J_{\zeta_2} \nabx \overline{\mathbf{w}}_2  \partial_t \big(\Psi_{\zeta_1}^{-1}\circ \Psi_{\zeta_1} 
-
\Psi_{\zeta_2}^{-1}\circ \Psi_{\zeta_2} \big)
\big\Vert_{L^2_\bx}^2\dt
\\&
\lesssim
T^{1/2}_*
\Vert (\zeta_1,\overline{\mathbf{w}}_1, \overline{q}_1) - (\zeta_2,\overline{\mathbf{w}}_2, \overline{q}_2) \Vert_{X_{I_*}}^2.
\end{aligned}
\end{equation}
Next, as in \eqref{K3dx},
we obtain
\begin{equation}
\begin{aligned}
\label{K3d}
&\int_{I_*}\Vert   J_{\zeta_1}\big( \nabx \Psi_{\zeta_1}^{-1}\circ \Psi_{\zeta_1}\big)^\intercal \nabx (\overline{\mathbf{w}}_1 -\overline{\mathbf{w}}_2 ) \overline{\mathbf{w}}_1  \Vert_{L^2_\bx}^2\dt
\\&\quad+
\int_{I_*}\big\Vert  
\big[J_{\zeta_1}\big( \nabx \Psi_{\zeta_1}^{-1}\circ \Psi_{\zeta_1}\big)^\intercal
-
J_{\zeta_2}\big( \nabx \Psi_{\zeta_2}^{-1}\circ \Psi_{\zeta_2}\big)^\intercal \big] \nabx \overline{\mathbf{w}}_2  \overline{\mathbf{w}}_1   \big\Vert_{L^2_\bx}^2\dt
\\&\quad
+\int_{I_*}\Vert
J_{\zeta_2}\big( \nabx \Psi_{\zeta_2}^{-1}\circ \Psi_{\zeta_2}\big)^\intercal \nabx \overline{\mathbf{w}}_2  (\overline{\mathbf{w}}_1-\overline{\mathbf{w}}_2)
\Vert_{L^2_\bx}^2\dt
\\&
\lesssim
T^{1/4}_*
\Vert (\zeta_1,\overline{\mathbf{w}}_1, \overline{q}_1) - (\zeta_2,\overline{\mathbf{w}}_2, \overline{q}_2) \Vert_{X_{I_*}}^2.
\end{aligned}
\end{equation}
By following the same argument as in \eqref{K3fx}, we obtain
\begin{equation}
\begin{aligned}
\label{K3f}
&\int_{I_*}\big(\Vert 
J_{\zeta_1} (\bff\circ \Psi_{\zeta_1} -
\bff\circ \Psi_{\zeta_2}  )
 \Vert_{L^2_\bx}^2
 +
 \Vert 
(J_{\zeta_1}-J_{\zeta_2}) 
\bff\circ \Psi_{\zeta_2} 
 \Vert_{L^2_\bx}^2
 \big)\dt
\\&
\lesssim
T^{1/4}_*
\Vert (\zeta_1,\overline{\mathbf{w}}_1, \overline{q}_1) - (\zeta_2,\overline{\mathbf{w}}_2, \overline{q}_2) \Vert_{X_{I_*}}^2.
\end{aligned}
\end{equation}
It follows from \eqref{K3a}--\eqref{K3f} that
\begin{equation}
\begin{aligned}
\label{k3final}
K_3
\lesssim  T^{1/2}_* 
\Vert (\zeta_1, \overline{\mathbf{w}}_1, \overline{q}_1)-(\zeta_2, \overline{\mathbf{w}}_2, \overline{q}_2)\Vert_{X_{I_*}}^2.
\end{aligned}
\end{equation}
Our next goal is to estimate $K_4$. First of all, note that
\begin{equation}
\begin{aligned}
\nonumber
\vert \mathbf{H}_{\zeta_1}(\overline{\mathbf{w}}_1, \overline{q}_1)- \mathbf{H}_{\zeta_2}(\overline{\mathbf{w}}_2, \overline{q}_2)\vert
&\lesssim
\vert (\mathbf{A}_{\eta_0} -\mathbf{A}_{\zeta_1})\nabx(\overline{\mathbf{w}}_1 - \overline{\mathbf{w}}_2)\vert
+
\vert (\mathbf{A}_{\zeta_1} -\mathbf{A}_{\zeta_2})\nabx \overline{\mathbf{w}}_2\vert
\\&
\quad+
\vert (\mathbf{B}_{\eta_0}-\mathbf{B}_{\zeta_1}) (\overline{q}_1 - \overline{q}_2)\vert
+
\vert (\mathbf{B}_{\zeta_1}-\mathbf{B}_{\zeta_2}) \overline{q}_2\vert,
\end{aligned}
\end{equation}
holds uniformly. By the same argument as in \eqref{K4ax}, by using \eqref{eq:emb1}, \eqref{eq:emb3} and \eqref{eq:emb2},
we obtain from \eqref{210and212}--\eqref{218} that
\begin{equation}
\begin{aligned}
\label{K4a}
&\int_{I_*}\big(\Vert (\mathbf{A}_{\eta_0} -\mathbf{A}_{\zeta_1})\nabx(\overline{\mathbf{w}}_1 - \overline{\mathbf{w}}_2)\Vert_{W^{1,2}_\bx}^2\dt
+
\Vert (\mathbf{A}_{\zeta_1} -\mathbf{A}_{\zeta_2})\nabx  \overline{\mathbf{w}}_2\Vert_{W^{1,2}_\bx}^2
\big)\dt
\\&\lesssim
 T^{1/4}_* 
\Vert (\zeta_1, \overline{\mathbf{w}}_1, \overline{q}_1)-(\zeta_2, \overline{\mathbf{w}}_2, \overline{q}_2)\Vert_{X_{I_*}}^2.
\end{aligned}
\end{equation}
Finally, we adopt the approach leading to \eqref{K4bx} and obtain
\begin{equation}
\begin{aligned}
\label{K4b}
&\int_{I_*}\big(\Vert (\mathbf{B}_{\eta_0}-\mathbf{B}_{\zeta_1}) (\overline{q}_1 - \overline{q}_2)\Vert_{W^{1,2}_\bx}^2
+
\int_{I_*}\Vert (\mathbf{B}_{\zeta_1}-\mathbf{B}_{\zeta_2})  \overline{q}_2\Vert_{W^{1,2}_\bx}^2
\big)\dt
 \\&\lesssim
T^{1/3}_*
\Vert (\zeta_1, \overline{\mathbf{w}}_1, \overline{q}_1)-(\zeta_2, \overline{\mathbf{w}}_2, \overline{q}_2)\Vert_{X_{I_*}}^2.
\end{aligned}
\end{equation}
By using \eqref{K4a} and \eqref{K4b}, it follows that
\begin{align}
\label{k4final}
K_4
\lesssim
T^{1/3}_*
\Vert (\zeta_1,\overline{\mathbf{w}}_1, \overline{q}_1) - (\zeta_2,\overline{\mathbf{w}}_2, \overline{q}_2) \Vert_{X_{I_*}}^2.
\end{align}
By collecting the estimates \eqref{k1final}, \eqref{k2final}, \eqref{k3final} and \eqref{k4final} together, we have
\begin{align*}
\Vert \mathcal{T}(\zeta_1, \overline{\mathbf{w}}_1, \overline{q}_1)
-
\mathcal{T}(\zeta_2, \overline{\mathbf{w}}_2, \overline{q}_2)\Vert_{X_{I_*}}^2
&\leq 
c  T^{1/2}_*  \Vert (\zeta_1, \overline{\mathbf{w}}_1, \overline{q}_1)-(\zeta_2, \overline{\mathbf{w}}_2, \overline{q}_2)\Vert_{X_{I_*}}^2.
\end{align*}
Choosing $T_*$ in $I_*=(0,T_*)$ so that $ T^{1/2}_* <c^{-1}$ yields the desired contraction property.
%
\end{proof}

\section{The acceleration estimate}
\label{sec:reg}
In this section, we prove the acceleration estimate for a solution satisfying the Serrin condition. In order to make the proof rigorous, we work with a strong solution, the existence of which is  guaranteed locally in time by Theorem \ref{thm:fluidStructureWithoutFK}. Eventually, we compare weak and strong solution by means of Theorem \ref{thm:weakstrong} and compare the result for a weak solution in Theorem \ref{thm:main}.
Let $(\bu,\eta)$ be a strong solution \eqref{1}--\eqref{interfaceCond} in the sense of Definition \ref{def:strongSolution} with data $(\bff, g, \eta_0, \eta_*, \bu_0)$, which is in particular a weak solution (see Definition \ref{def:weakSolution}) and thus satisfies the standard energy estimates
\begin{align}
\sup_{I^\ast}\|\bu\|_{L^2_\bx}^2+\int_{I^\ast}\|\nabla\bu\|_{L^2_\bx}^2\dt\lesssim\,C_0,\label{eq:aprioriu}\\
\sup_{I^\ast}\|\partial_t\eta\|_{L^2_\by}^2+\sup_{I^\ast}\|\Dely\eta\|_{L^2_\by}^2+\int_{I^\ast}\|\partial_t\naby\eta\|_{L^2_\by}^2\dt\lesssim\,C_0,\label{eq:apriorieta}
\end{align}
where
\begin{align*}
 C_0&:=\|\bu_0\|_{L^2_\bx}^2+\|\eta_\ast\|_{L^2_\by}^2+\|\Dely\eta_0\|_{L^2_\by}^2+\int_{I^\ast}\|\bff\|_{L^2_\bx}^2\dt+\int_{I^\ast}\|g\|_{L^2_\by}^2\dt.
\end{align*} 
The following acceleration estimate (or second order energy estimate) is one of the core results of the paper and directly leads to the main result in Theorem \ref{thm:main}. Under the Serrin condition, it holds uniformly in time allowing us to extend the local solution globally in time.
\begin{theorem}\label{prop2}
Suppose that the dataset
$(\bff, g, \eta_0, \eta_*, \bu_0)$
satisfies \eqref{dataset} and
\eqref{datasetImproved}.
Suppose that $(\eta,\bu)$ is a strong solution to \eqref{1}--\eqref{interfaceCond} in the sense of Definition \ref{def:strongSolution}. 
Furthermore, for some $r\in[2,\infty)$ and $s\in(3,\infty]$ we set
\begin{align}\label{eq:regu}
C_1&:=\|\bu\|_{L^r(I;L^s(\Omega_\eta))},\quad \tfrac{2}{r}+\tfrac{3}{s}\leq1,\\
C_2&:=\|\eta\|_{ L^\infty(I;C^{1}(\omega))}.\label{eq:regeta}
\end{align}
Finally, suppose that there is no degeneracy in the sense of \eqref{eq:1705}.
 Then we have the estimate
 \begin{equation}
\begin{aligned}
\label{est:reg}
&\sup_{I_\ast}\int_\omega
\big(\vert \partial_t\naby \eta\vert^2 
+
\vert \naby\Dely \eta\vert^2
\big)
\dy
+
\sup_{I_\ast}\int_{\Omega_\eta}\vert\nabx  \bu \vert^2\dx
\\&\quad+
\int_{I_\ast}\int_\omega
\big(\vert \partial_t\Dely \eta \vert^2 + \vert \partial_t^2 \eta\vert^2
 \big)\dy\dt
 +
\int_{I_\ast}\int_{\Omega_\eta}\big( \vert \nabx^2 \bu\vert^2 +\vert \partial_t \bu \vert^2  + \vert \nabx  \pi\vert^2
 \big)\dx\dt
 \\&\lesssim
 \int_\omega\big( 
 \vert \naby\eta_*\vert^2
 +
  \vert \naby\Dely\eta_0\vert^2
  \big)\dy
  +
  \int_{\Omega_{\eta_0}} 
   \vert\nabx \bu_0\vert^2 \dx
   \\&\quad+ \int_{I_\ast}\int_{\Omega_\eta}  \vert \bff\vert^2  \dx\dt+
 \int_{I_\ast}\int_\omega  \vert g\vert^2  \dy\dt,
\end{aligned}
\end{equation}
where the hidden constant depends only on $C_0, C_1$ and $C_2$.
\end{theorem}
\begin{proof}
We use
$
\bfphi=\partial_t\bu+\mathscr F_\eta(\partial_t\eta\bfn)\cdot\nabla\bu
$
and $\phi=\partial_t^2\eta$ as test functions for the fluid and shell equations respectively. Here $\mathscr F_\eta$ is the extension operator introduced in Section \ref{sec:ext}. 
 From the momentum equation in the strong form \eqref{1'}, we obtain for $t\in I^\ast$
\begin{align*}
&\int_{0}^t\int_{\Omega_\eta}\big(\partial_t\bu+\bu\cdot\nabla\bu\big)\cdot\big(\partial_t\bu+\mathscr F_\eta(\partial_t\eta\bfn)\cdot\nabla\bu\big)\dx\ds\\
&=\int_{0}^t\int_{\Omega_\eta}\Div\bftau\cdot\big(\partial_t\bu+\mathscr F_\eta(\partial_t\eta\bfn)\cdot\nabla\bu\big)\dx\ds+\int_{0}^t\int_{\Omega_\eta}\bff\cdot\big(\partial_t\bu+\mathscr F_\eta(\partial_t\eta\bfn)\cdot\nabla\bu)\big)\dx\ds,
\end{align*}
where $\bftau=\nabla\bu+\nabla\bu^\intercal-\pi\mathbb I_{3\times 3}$ is the Cauchy stress.
We now aim at  integrating by parts in the first term on the right-hand side
obtaining
\begin{align}\nonumber
&\int_{0}^t\int_{\Omega_\eta}\Div\bftau\cdot\big(\partial_t\bu+\mathscr F_\eta(\partial_t\eta\bfn)\cdot\nabla\bu\big)\dx\ds\\&=-\tfrac{1}{2}\int_{\Omega_\eta}|\nabla\bu|^2\dx+\tfrac{1}{2}\int_{\Omega_\eta}|\nabla\bu_0|^2\dx+\int_{I_\ast}\int_{\partial\Omega_\eta}(\partial_t\eta\bfn)\circ\bfvarphi_\eta^{-1}\cdot\bfn_\eta\circ\bfvarphi_\eta^{-1}|\nabla\bu|^2\dd\mathcal H^2\dt\nonumber\\
&\quad-\int_{0}^t\int_{\Omega_\eta}\nabla\bu:\nabla\big(\mathscr F_\eta(\partial_t\eta\bfn)\cdot\nabla\bu\big)\dx\ds+\int_{0}^t\int_{\Omega_\eta}\pi\,\Div\big(\mathscr F_\eta(\partial_t\eta\bfn)\nabla\bu\big)\dx\ds\nonumber\\
&\quad-\int_0^t\int_\omega \bfF\,\partial_t^2\eta\dy\ds\label{eq:well}
\end{align}
with 
\begin{equation}\label{Fdef}
\bfF=-\bfn^\intercal \bftau\circ\bfvarphi_\eta\bn_\eta|\det(\naby\bfvarphi_\eta)|.\end{equation}
Note that we also used Reynold's transport theorem (applied to $\int_{\Omega_{\eta(t)}}|\nabla\bu(t)|^2\dx$).
Although all the terms in equation \eqref{eq:well} are well-defined for a strong solution $(\eta,\bu)$ this is not true for its derivation. Hence we apply Lemma \ref{lem:smooth} to obtain a smooth approximation which fully justifies \eqref{eq:well} after passing to the limit.

Multiply the structure equation \eqref{2'} by $\partial_t^2\eta$ we obtain from the formal computation \eqref{3.10} that
\begin{equation*}
	\begin{aligned}
&\int_{I_\ast}\int_\omega|\partial_t^2\eta|^2\dy\dt
+
\sup_{I_\ast}\int_\omega|\partial_t\naby\eta|^2\dy
\\&\lesssim
\int_\omega|\naby\eta_*|^2\dy
+
\sup_{I_\ast}\int_\omega|\naby\Dely\eta|^2\dy+\int_{I_\ast}\int_\omega|\partial_t\Dely\eta|^2\dy\dt+\int_{I_\ast}\int_\omega(g+\bfF)\,\partial_t^2\eta\dy\dt.
\end{aligned}
\end{equation*}
It can be made rigorous by means of a spatial regularization argument. Since we consider periodic boundary conditions a spatial convolution can be applied without further difficulty.
Combining both, using Young's inequality and writing
$$\mathscr F_\eta(\partial_t\eta\bfn)\cdot\nabla\bu=\bu\cdot\nabla\bu+\mathscr F_\eta(\partial_t\eta\bfn)\cdot\nabla\bu-\bu\cdot\nabla\bu,$$
we derive that
\begin{align}\label{eq:0302}
\begin{aligned}
&\sup_{I_\ast}\int_{\Omega_\eta}|\nabla\bu|^2\dx+\int_{I_\ast}\int_{\Omega_\eta}|\partial_t\bu+\bu\cdot\nabla\bu|^2\dx\dt
+\int_{I_\ast}\int_\omega|\partial_t^2\eta|^2\dy\dt+\sup_{I_\ast}\int_\omega|\partial_t\naby\eta|^2\dy
\\
&\lesssim\int_{I_\ast}\int_{\Omega_\eta}|\bu\cdot\nabla\bu|^2\dx\dt+\int_{I_\ast}\int_{\partial\Omega_\eta}(\partial_t\eta\bfn)\circ\bfvarphi_\eta^{-1}\cdot\bfn_\eta\circ\bfvarphi_\eta^{-1}|\nabla\bu|^2\dd\mathcal H^2\dt\\
&\quad-\int_{I_\ast}\int_{\Omega_\eta}\big(\partial_t\bu+\bu\cdot\nabla\bu\big)\cdot\big(\mathscr F_\eta(\partial_t\eta\bfn)\cdot\nabla\bu\big)\dx\dt\\
&\quad-\int_{I_\ast}\int_{\Omega_\eta}\nabla\bu:\big(\mathscr F_\eta(\partial_t\eta\bfn)^\intercal\nabla^2\bu+\nabla\mathscr F_\eta(\partial_t\eta\bfn)\nabla\bu^\intercal\big)\dx\dt\\
&\quad+\int_{I_\ast}\int_{\Omega_\eta}\pi\,\Div\big(\mathscr F_\eta(\partial_t\eta\bfn)\nabla\bu\big)\dx\dt+\int_{I_*}\int_{\Omega_\eta}\bff\cdot \mathscr F_{\eta}(\partial_t\eta\bfn)\nabla\bu\dx\dt\\
&\quad+\int_{I_\ast}\int_{\Omega_\eta}|\bff|^2\dx\dt+\int_{\Omega_{\eta_0}}|\nabla\bu_0|^2\dx
+\int_\omega|\naby\eta_*|^2\dy+\sup_{I_\ast}\int_\omega|\naby\Dely\eta|^2\dy\\
&\quad+\int_{I_\ast}\int_\omega|\partial_t\Dely\eta|^2\dy\dt+\int_{I_\ast}\int_\omega|g|^2\dy\dt\\
&=:\mathrm{I}+\dots+\mathrm{XII}.
\end{aligned}
\end{align}
Notice that in \eqref{eq:0302} the last six terms are already uncritical. In particular, XI will be obtained on the left hand side via a second test, see \eqref{eq:secondtest} below. Hence we start the estimate of the terms $\mathrm{I}$--$\mathrm{VI}$. 

In order to control the first term, we make use of Theorem \ref{thm:stokessteady}. Its application to the moving domain $\Omega_\eta$ has been justified in Remark \ref{rem:stokes}, we thereby have
\begin{align}\label{eq:regstokes}
\|\bu\|_{W^{2,2}_\bx}+\|\pi\|_{W^{1,2}_\bx}\lesssim\|\partial_t\bu+\bu\cdot\nabla\bu\|_{L^2_\bx}+\|\bff\|_{L^2_\bx}+\|\partial_t\eta\|_{W^{3/2,2}_\by},
\end{align}
uniformly in time with a constant depending on $C_2$ from \eqref{eq:regeta}. Note that we also used the estimate
\begin{align}\label{again}
\|\partial_t\eta\bfn\circ \bfvarphi_\eta^{-1}\|_{W^{3/2,2}_\by}\lesssim\|\partial_t\eta\|_{W^{3/2,2}_\by},
\end{align}
which is a consequence of \eqref{eq:apriorieta} and the definition ${\bfvarphi}_\eta={\bfvarphi}+ \eta{\bfn}$. In fact,
$\bfvarphi_\eta^{-1}$ is uniformly bounded in time in the space
of Sobolev multipliers on $W^{3/2,2}(\omega)$ by \eqref{eq:MSa} and \eqref{eq:MSb}  (together with the assumption $\partial_1\bfvarphi_\eta\times\partial_2\bfvarphi_\eta \neq0$) since $\eta$ is uniformly bounded in time even in $W^{2,2}(\omega)$ by \eqref{eq:apriorieta}.
Hence the transformation rule \eqref{lem:9.4.1} applies.
For every $\kappa>0$ and for $s\in(3,\infty]$, we estimate by using Sobolev's inequality (recalling that $\partial\Omega_\eta$ is Lipschitz uniformly in time with a constant controlled by $C_2$, cf. \eqref{eq:regeta}) and \eqref{eq:aprioriu}
\begin{align*}
\begin{aligned}
\mathrm{I}&\leq \int_{I_\ast}\|\bu\|^2_{L^s_\bx}\|\nabla\bu\|_{L^{\frac{2s}{s-2}}_\bx}^2\dt
\leq\,c\int_{I_\ast}\|\bu\|_{L^s_\bx}^2\|\nabx\bu\|^{\frac{2s-6}{s}}_{L^{2}_\bx}\|\nabla\bu\|_{W^{1,2}_\bx}^{\frac{6}{s}}\dt\\
&\leq\,c(\kappa)\int_{I_\ast}\|\bu\|_{L^s_\bx}^{\frac{2s}{s-3}}\|\nabx\bu\|^{2}_{L^{2}_\bx}\dt+\kappa\int_{I_\ast}\|\nabla\bu\|_{W^{1,2}_\bx}^2\dt\\
&\leq \,c(\kappa)\int_{I_\ast}\|\bu\|_{L^s_\bx}^{\frac{2s}{s-3}}\|\nabx\bu\|^{2}_{L^{2}_\bx}\dt+\kappa\int_{I_\ast}\big(\|\partial_t\bu+\bu\cdot\nabla\bu\|_{L^2_\bx}^2+\|\bff\|^2_{L^2_\bx}+\|\partial_t\eta\|^2_{W^{3/2,2}_\by}\big)\dt,
\end{aligned}
\end{align*}
where the first part of the $\kappa$-term can be absorbed in the left-hand side of \eqref{eq:0302}. 
Note that $r:=\frac{2s}{s-3}\in[2,\infty)$ since $s\in(3,\infty]$. The resulting constant depends on $C_1$ from \eqref{eq:regu}.

For the boundary integral $\mathrm{II}$ on the right-hand side of \eqref{eq:0302}, we have
\begin{equation*}\label{estimateII}
\begin{aligned}
\mathrm{II}&\leq \int_{I_*}\lVert\nabla\bu\rVert_{L^4(\partial\Omega_{\eta})}\lVert\nabla\bu\rVert_{L^{\frac{8}{3}}(\partial\Omega_{\eta})}\lVert\partial_t\eta\circ\bfvarphi_\eta^{-1}\rVert_{L^{\frac{8}{3}}(\partial\Omega_\eta)} \dt\\
&\lesssim  \int_{I_*}\lVert\nabla\bu\rVert_{W^{1/2, 2}(\partial\Omega_\eta)}\lVert\nabla\bu\rVert_{W^{1/4, 2}(\partial\Omega_\eta)}\lVert\partial_t\eta\circ\bfvarphi_\eta^{-1}\rVert_{W^{1/4,2}(\partial\Omega_\eta)}\dt\\
&\lesssim \int_{I_*}\lVert\nabla\bu\rVert_{W^{1,2}_\bx}\lVert\nabla\bu\rVert_{W^{3/4,2}_\bx}\lVert\partial_t\eta\rVert_{W^{1/4,2}_\by}\dt\\
&\lesssim \int_{I_*}\lVert\nabla\bu\rVert_{W^{1,2}_\bx}^{\frac{7}{4}}\lVert\nabla\bu\rVert_{L^2_\bx}^{\frac{1}{4}}\lVert\partial_t\eta\rVert_{L^2_\by}^{\frac{3}{4}}\lVert\partial_t\eta\rVert_{W^{1,2}_\by}^{\frac{1}{4}}\dt\\
&\leq \kappa\int_{I_*}\lVert\nabla\bu\rVert_{W^{1,2}_\bx}^2\dt+C(\kappa)\int_{I_*}\lVert\partial_t\eta\rVert_{W^{1,2}_\by}^2\lVert\nabla\bu\rVert_{L^2_\bx}^2\dt,
\end{aligned}
\end{equation*}
where we used that $\lVert\partial_t\eta\rVert_{L^2_\by}$ is uniformly bounded in time (see \eqref{eq:apriorieta}) and the embeddings $W^{1/2,2}(\partial\Omega_\eta)\hookrightarrow L^4(\partial\Omega_\eta)$ and $W^{1/4,2}(\partial\Omega_\eta)\hookrightarrow L^{8/3}(\partial\Omega_\eta)$, as well as the following interpolation inequalities:
\begin{equation}\label{interalways}
\begin{aligned}
\lVert f\rVert_{W^{\frac{1}{4},2}(\omega)}&\lesssim \lVert f\rVert_{L^2(\omega)}^{\frac{3}{4}}\lVert f\rVert_{W^{1,2}(\omega)}^{\frac{1}{4}},\\
\lVert f\rVert_{W^{\frac{3}{4},2}(\Omega_\eta)}&\lesssim \lVert f\rVert_{L^2(\Omega_\eta)}^{\frac{1}{4}}\lVert f\rVert_{W^{1,2}(\Omega_\eta)}^{\frac{3}{4}}.
\end{aligned}
\end{equation}

For the third term $\mathrm{III}$, we first have
\begin{align*}
\mathrm{III}&\leq\,\kappa\int_{I_\ast}\|\partial_t\bu+\bu\cdot\nabla\bu\|_{L^2_\bx}^2\dt
+c(\kappa)\int_{I_\ast}\|\nabla\bu\mathscr F_\eta(\partial_t\eta\bfn)\|^2_{L^2_\bx}\dt.
\end{align*}
We have the estimate for the second integrals above that
\begin{equation*}
\begin{aligned}
&\int_{I_\ast}\|\nabla\bu\mathscr F_\eta(\partial_t\eta\bfn)\|^2_{L^2_\bx}\dt\leq \int_{I_*}\lVert \mathscr F_\eta(\partial_t\eta\bfn)\rVert_{L^4_\bx}^2\lVert \nabla\bu\rVert_{L^4_\bx}^2\dt\\
&\lesssim \int_{I_\ast}\lVert \mathscr F_\eta(\partial_t\eta\bfn)\rVert_{W^{3/4,2}_\bx}^2\lVert\nabla\bu\rVert_{L^2_\bx}^{\frac{1}{2}}\lVert\nabla\bu\rVert_{W^{1,2}_\bx}^{\frac{3}{2}}\dt\\
&\lesssim \int_{I_*}\lVert\partial_t\eta\rVert_{W^{1/4, 2}_\by}^2\lVert\nabla\bu\rVert_{L^2_\bx}^{\frac{1}{2}}\lVert\nabla\bu\rVert_{W^{1,2}_\bx}^{\frac{3}{2}}\dt\\
&\lesssim \int_{I_*}\lVert\partial_t\eta\rVert_{W^{1,2}_\by}^{\frac{1}{2}}\lVert\nabla\bu\rVert_{L^2_\bx}^{\frac{1}{2}}\lVert\nabla\bu\rVert_{W^{1,2}_\bx}^{\frac{3}{2}}\dt\\
&\leq \kappa \int_{I_*}\lVert\nabla\bu\rVert_{W^{1,2}_\bx}^2\dt+C(\kappa)\int_{I_*}\lVert\partial_t\eta\rVert_{W^{1,2}_\by}^2\lVert\nabla\bu\rVert_{L^2_\bx}^2\dt.
\end{aligned}
\end{equation*}
Here we used the 2D interpolation in \eqref{interalways} again and 
$$\lVert\mathscr F_\eta(\partial_t\eta\bfn)\rVert_{W^{3/4,2}_\bx}\lesssim \lVert\partial_t\eta\rVert_{W^{1/4,2}_\by}, $$ 
and the inequalities:
\begin{equation}\label{reatinter}
\begin{aligned}
\lVert f\rVert_{L^4(\Omega_\eta)}&\lesssim \lVert f\rVert_{W^{3/4, 2}(\Omega_\eta)},\\
\lVert f\rVert_{L^4(\Omega_\eta)}&\lesssim \lVert f\rVert_{L^2(\Omega_\eta)}^{\frac{1}{4}}\lVert f\rVert_{W^{1,2}(\Omega_\eta)}^{\frac{3}{4}}.
\end{aligned}
\end{equation}
Thus we have the estimate for $\mathrm{III}$:
\begin{equation*}\label{IIestimate}
\mathrm{III}\leq \kappa\int_{I_\ast}\left(\|\partial_t\bu+\bu\cdot\nabla\bu\|_{L^2_\bx}^2+\lVert\nabla\bu\rVert_{W^{1,2}_\bx}^2\right)\dt+C(\kappa) \int_{I_*}\lVert\partial_t\eta\rVert_{W^{1,2}_\by}^2\lVert\nabla\bu\rVert_{L^2_\bx}^2\dt.
\end{equation*}

Now we consider estimating the integral $\mathrm{IV}$. For this we have 
\begin{align}
&\int_{I_*}\int_{\Omega_\eta}\nabla\bu:\mathscr F_\eta(\partial_t\eta\bfn)^\intercal\nabla^2\bu\dx\dt\nonumber\\
&\leq \int_{I_*}\lVert\nabla^2\bu\rVert_{L^2_\bx}\lVert\nabla\bu\rVert_{L^4_\bx}\lVert\mathscr F_\eta(\partial_t\eta\bfn)\rVert_{L^4_\bx}\dt\nonumber\\
&\lesssim \int_{I_*}\lVert\nabla^2\bu\rVert_{L^2_\bx}\lVert\nabla\bu\rVert_{L^2_\bx}^{\frac{1}{4}}\lVert\nabla\bu\rVert_{W^{1,2}_\bx}^{\frac{3}{4}}\lVert\mathscr F_\eta(\partial_t\eta\bfn)\rVert_{W^{3/4,2}_\bx}\dt \nonumber\\
&\lesssim \int_{I_*}\lVert\bu\rVert_{W^{2,2}_\bx}^{\frac{7}{4}}\lVert\nabla\bu\rVert_{L^2_\bx}^{\frac{1}{4}}\lVert\partial_t\eta\rVert_{W^{1/4, 2}_\by}\dt\nonumber\\
&\lesssim \int_{I_*}\lVert\bu\rVert_{W^{2,2}_\bx}^{\frac{7}{4}}\lVert\nabla\bu\rVert_{L^2_\bx}^{\frac{1}{4}}\lVert\partial_t\eta\rVert_{W^{1, 2}_\by}^{\frac{1}{4}}\dt\nonumber\\
&\leq \kappa\int_{I_*}\lVert\bu\rVert_{W^{2,2}_\bx}^2\dt+C(\kappa)\int_{I_*}\lVert\partial_t\eta\rVert_{W^{1,2}_\bx}^2\lVert\nabla\bu\rVert_{L^2_\bx}^2\dt,\label{IV1}
\end{align}
where we used the energy estimate \eqref{eq:apriorieta} and the interpolation inequality \eqref{interalways} and \eqref{reatinter} again. Then for the second integral in $\mathrm{IV}$, we also have
\begin{equation}\label{IV2}
\begin{aligned}
&\int_{I_*}\int_{\Omega_\eta}\nabla\bu:\nabla\mathscr F_\eta(\partial_t\eta\bfn)\nabla\bu^\intercal\dx\dt\\
&\lesssim \int_{I_*}\lVert\nabla\bu\rVert_{L^4_\bx}\lVert\nabla\bu^\intercal\rVert_{L^4_\bx}\lVert\nabla\mathscr F_\eta(\partial_t\eta\bfn)\rVert_{L^2_\bx}\dt\\
&\lesssim \lVert\nabla\bu\rVert_{L^2_\bx}^{\frac{1}{4}}\lVert\nabla\bu\rVert_{W^{1,2}_\bx}^{\frac{3}{4}}\lVert\nabla\bu^\intercal\rVert_{L^2_\bx}^{\frac{1}{4}}\lVert\nabla\bu^\intercal\rVert_{W^{1,2}_\bx}^{\frac{3}{4}}\lVert\mathscr F_\eta(\partial_t\eta\bfn)\rVert_{W^{1,2}_\bx}\dt\\
&\lesssim \int_{I_*}\lVert\bu\rVert_{W^{2,2}_\bx}^{\frac{3}{2}}\lVert\bu\rVert_{W^{1,2}_\bx}^{\frac{1}{2}}\lVert\partial_t\eta\rVert_{W^{1/2,2}_\by}\dt\\
&\lesssim\int_{I_*} \lVert\bu\rVert_{W^{2,2}_\bx}^{\frac{3}{2}}\lVert\bu\rVert_{W^{1,2}_\bx}^{\frac{1}{2}}\lVert\partial_t\eta\rVert_{W^{1,2}_\by}^{\frac{1}{2}}\dt\\
&\leq \kappa\int_{I_*}\lVert \bu\rVert_{W^{2,2}_\bx}^2\dt+C(\kappa)\int_{I_*}\int_{I_*}\lVert\partial_t\eta\rVert_{W^{1,2}_\by}^2\lVert\bu\rVert_{W^{1,2}_\bx}^2\dt,
\end{aligned}
\end{equation}
where we used the interpolation:
$$\lVert \partial_t\eta\rVert_{W^{1/2, 2}_\by}\lesssim \lVert\partial_t\eta\rVert_{L^2_\by}^{\frac{1}{2}}\lVert\partial_t\eta\rVert_{W^{1,2}_\by}^{\frac{1}{2}}. $$
Putting the estimate \eqref{IV1} and \eqref{IV2} together, we arrive at
\begin{equation*}\label{IV}
\mathrm{IV}\leq \kappa\int_{I_*}\lVert \bu\rVert_{W^{2,2}_\bx}^2\dt+C(\kappa)\int_{I_*}\lVert\partial_t\eta\rVert_{W^{1,2}_\by}^2\lVert\bu\rVert_{W^{1,2}_\bx}^2\dt.
\end{equation*}

To estimate $\mathrm{V}$, we note that it can be written as
\begin{align*}
	\mathrm{V}=-\int_{I_\ast}\int_{\Omega_\eta}\nabla\pi\cdot\mathscr F_\eta(\partial_t\eta\bfn)\nabla\bu\dx\dt+\int_{I_\ast}\int_{\partial\Omega_\eta}\pi\,\mathscr F_\eta(\partial_t\eta\bfn)\nabla\bu\,\bfn_\eta\circ\bfvarphi_\eta^{-1}\,\dd\mathcal H^2\dt.
\end{align*}
And we have the estimates for the two integrals on the right-hand side above one by one. Firstly, we derive that
\begin{equation}\label{fsimilar}
\begin{aligned}
&\int_{I_\ast}\int_{\Omega_\eta}\nabla\pi\cdot\mathscr F_\eta(\partial_t\eta\bfn)\nabla\bu\dx\dt\\
&\leq \int_{I_*}\lVert\nabla\pi\rVert_{L^2_\bx}\lVert\mathscr F_\eta(\partial_t\eta\bfn)\rVert_{L^4_\bx}\lVert\nabla\bu\rVert_{L^4_\bx}\dt\\
&\lesssim \int_{I_*}\lVert\nabla\pi\rVert_{L^2_\bx}\lVert\mathscr F_\eta(\partial_t\eta\bfn)\rVert_{W^{3/4,2}_\bx}\lVert\nabla\bu\rVert_{L^2_\bx}^{\frac{1}{4}}\lVert\nabla\bu\rVert_{W^{1,2}_\bx}^{\frac{3}{4}}\dt\\
&\lesssim \int_{I_*}\lVert \nabla\pi\rVert_{L^2_\bx}\lVert\partial_t\eta\rVert_{W^{1/4,2}_\by}\lVert\nabla\bu\rVert_{L^2_\bx}^{\frac{1}{4}}\lVert\nabla\bu\rVert_{W^{1,2}_\bx}^{\frac{3}{4}}\dt\\
&\lesssim \int_{I_*}\lVert\nabla\pi\rVert_{L^2_\bx}\lVert\partial_t\eta\rVert_{W^{1,2}_\by}^{\frac{1}{4}}\lVert\nabla\bu\rVert_{L^2_\bx}^{\frac{1}{4}}\lVert\nabla\bu\rVert_{W^{1,2}_\bx}^{\frac{3}{4}}\dt\\
&\lesssim \kappa\int_{I_*}\left(\lVert\nabla\pi\rVert_{L^2_\bx}^2+\lVert\nabla\bu\rVert_{W^{1,2}_\bx}^2\right)\dt+C(\kappa)\int_{I_*}\lVert\partial_t\eta\rVert_{W^{1,2}_\by}^2\lVert\nabla\bu\rVert_{L^2_\bx}^2\dt.
\end{aligned}
\end{equation}
Then, by using similar technique, we also have 
\begin{equation*}
\begin{aligned}
&\int_{I_\ast}\int_{\partial\Omega_\eta}\pi\,\mathscr F_\eta(\partial_t\eta\bfn)\nabla\bu\,\bfn_\eta\circ\bfvarphi_\eta^{-1}\,\dd\mathcal H^2\dt\\
&\leq \int_{I_*}\lVert\pi\rVert_{L^4(\partial\Omega\eta)}\lVert\mathscr F_\eta(\partial_t\eta\bfn)\rVert_{L^{8/3}(\partial\Omega_\eta)}\lVert\nabla\bu\rVert_{L^{8/3}(\partial\Omega_\eta)}\dt\\
&\lesssim \int_{I_*}\lVert\pi\rVert_{W^{1/2,2}(\partial\Omega_\eta)}\lVert\mathscr  F_{\eta}(\partial_t\eta\bfn)\rVert_{W^{1/4,2}(\partial\Omega_\eta)}\lVert\nabla\bu\rVert_{W^{1/4, 2}(\partial\Omega_\eta)}\dt\\
&\lesssim \int_{I_*}\lVert\pi\rVert_{W^{1,2}_\bx}\lVert\partial_t\eta\rVert_{W^{1/4,2}_\by}\lVert\nabla\bu\rVert_{W^{3/4,2}_\bx}\dt\\
&\lesssim \int_{I_*}\lVert\pi\rVert_{W^{1,2}_\bx}\lVert\partial_t\eta\rVert_{W^{1,2}_\by}^{\frac{1}{4}}\lVert\nabla\bu\rVert_{L^2_\bx}^{\frac{1}{4}}\lVert\nabla\bu\rVert_{W^{1,2}_\bx}^{\frac{3}{4}}\dt\\
&\leq \kappa\int_{I_*}\left(\lVert\pi\rVert_{W^{1,2}_\bx}^2+\lVert\bu\rVert_{W^{2,2}_\bx}^2\right)\dt+C(\kappa)\int_{I_*}\lVert\partial_t\eta\rVert_{W^{1,2}_\by}^2\lVert\nabla\bu\rVert_{L^2_\bx}^2\dt.
\end{aligned}
\end{equation*}
Therefore, we obtain the estimate for $\mathrm{V}$ as follows:
\begin{equation*}\label{Vesti}
\begin{aligned}
\mathrm{V}\leq \kappa\int_{I_*}\left(\lVert\pi\rVert_{W^{1,2}_\bx}^2+\lVert\bu\rVert_{W^{2,2}_\bx}^2\right)\dt+C(\kappa)\int_{I_*}\lVert\partial_t\eta\rVert_{W^{1,2}_\by}^2\lVert\nabla\bu\rVert_{L^2_\bx}^2\dt.
\end{aligned}
\end{equation*}

Notice that the sixth integral on the right-hand side of \eqref{eq:0302}, i.e. $\mathrm{VI}$, can be treated similarly as \eqref{fsimilar}. 
Observe from the estimate of $\mathrm{V}$ above, we must estimate here also the $L^2$-norm of the pressure, for which we use
\eqref{eq:pressure} (noticing that $\int_{\omega}\bfn\cdot\bfn_\eta|\partial_y\bfvarphi_\eta|\dy$ is strictly positive by our assumption of non-degeneracy). We have
\begin{equation*}
	\begin{aligned}
\int_{I_\ast}\|\pi\|^2_{W^{1,2}_\bx}\dt&\lesssim \int_{I_\ast}\|\nabla\pi\|^2_{L^{2}_\bx}\dt+\int_{I_\ast}c_\pi^2\dt\\
&\lesssim \int_{I_\ast}\|\nabla\pi\|^2_{L^{2}_\bx}\dt+\int_{I_\ast}\int_\omega|\partial_t^2\eta|^2\dy\dt+\int_{I_\ast}\int_\omega|g|^2\dy\dt\\
&\quad+\int_{I_\ast}\|\pi_0\|_{L^{2}(\partial\Omega_\eta)}^2\dt+\int_{I_\ast}\|\nabla\bu\|_{L^{2}(\partial\Omega_\eta)}^2\dt.
\end{aligned}
\end{equation*}
 Moreover, for any $\epsilon\in(0,1/2)$ the last term above can be estimated as                                                   
\begin{align*}
\int_{I_\ast}\|\nabla\bu\|_{L^{2}(\partial\Omega_\eta)}^2\dt
&\lesssim \int_{I_\ast}\|\nabla\bu\|_{W^{1/2+\epsilon,2}(\Omega_\eta)}^2\dt\\
&\lesssim\int_{I_\ast}
\|\nabla\bu\|_{L^{2}(\Omega_\eta)}^{1-2\epsilon}
\| \bu\|_{W^{2,2}(\Omega_\eta)}^{1+2\epsilon}\dt\\
&\leq \kappa\int_{I_\ast}\big(\|\partial_t\bu+\bu\cdot\nabla\bu\|_{L^2_\bx}^2+\|\bff\|^2_{L^2_\bx}+\|\partial_t\eta\|^2_{W^{3/2,2}_\by}\big)\dt
+c(\kappa)\int_{I_\ast}\|\nabla\bu\|_{L^{2}_\bx}^{2}\dt,
\end{align*}
whereas by using Poincar\'e inequality,
\begin{align*}
\int_{I_\ast}\|\pi_0\|_{L^{2}(\partial\Omega_\eta)}^2\dt
\lesssim\int_{I_\ast}\|\nabla\pi_0\|_{L^{2}_\bx}^{2}\dt+\int_{I_\ast}\|\pi_0\|_{L^{2}_\bx}^{2}\dt\lesssim \int_{I_\ast}\|\nabla\pi_0\|_{L^{2}_\bx}^{2}\dt=\int_{I_\ast}\|\nabla\pi\|_{L^{2}_\bx}^{2}\dt,
\end{align*}
where we used that $(\pi_0)_{\Omega_{\eta}}=0$ by definition.
At this stage, the integrals on the
pressure in the above can now be controlled by means of \eqref{eq:regstokes}.

Combining all the above estimates, choosing $\kappa$ small enough and using \eqref{eq:apriorieta} once more we conclude that
\begin{align}\label{eq:reg2}
\begin{aligned}
&\sup_{I_\ast}\int_{\Omega_\eta}|\nabla\bu|^2\dx+\int_{I_\ast}\int_{\Omega_\eta}|\partial_t\bu+\bu\cdot\nabla\bu|^2\dx\dt
+\int_{I_\ast}\int_\omega|\partial_t^2\eta|^2\dx\dt
+
\sup_{I_\ast}\int_\omega|\partial_t\naby\eta|^2\dy 
\\
&\lesssim 
\int_{I_\ast}\|\bff\|_{L^2_\bx}^2\dt+\|\nabla\bu_0\|_{L^2_\bx}^2
+
\int_{I_\ast}\|g\|_{L^2_\by}^2\dt+\|\naby\eta_*\|^2_{L^2_\by}+\int_{I_\ast}\|\nabla\bu\|_{L^{2}_\bx}^{2}\dt
\\&
\quad+\int_{I_\ast}
\|\bu\|_{L^s_\bx}^{\frac{2s}{s-3}}\|\nabx\bu\|^{2}_{L^{2}_\bx}\dt 
+
\int_{I_\ast}\|\partial_t\eta\|_{W^{1,2}_\by}^2\|\nabla\bu\|_{L^2_\bx}^2\dt+\int_{I_*}\lVert\partial_t\eta\rVert_{W^{1,2}_\by}^2\dt
\\
&\quad+
\sup_{I_\ast}\|\naby\Dely\eta\|_{L^2_\by}^2
+\int_{I_\ast}\|\partial_t\Dely\eta\|_{L^2_\by}^2\dt.
\end{aligned}
\end{align}
In the above we used the interpolation for the structure norm $\lVert\partial_t\eta\rVert_{W^{3/2}_\by}^2$:
\begin{equation}\label{interpoeta}
 \lVert\partial_t\eta\rVert_{W^{3/2}_\by}^2\leq \lVert\partial_t\eta\rVert_{W^{1,2}_\by}\lVert\partial_t\eta\rVert_{W^{2,2}_\by}\leq \kappa\lVert\partial_t\eta\rVert_{W^{2,2}_\by}^2+C(\kappa)\lVert\partial_t\eta\rVert_{W^{1,2}_\by}^2.
\end{equation}

Testing the structure equation by $\partial_t\Dely\eta$ yields\footnote{This test can be rigorously performed by mollifying the structure equation and multiplying it with the mollified test-function.}
\begin{align}
\label{eq:secondtest}
\begin{aligned}
&\frac{1}{2}\sup_{I_\ast}\int_\omega|\partial_t\naby\eta|^2\dy+\int_{I_\ast}\int_\omega|\partial_t\Dely\eta|^2\dy\dt+\frac{1}{2}\sup_{I_\ast}\int_\omega|\naby\Dely\eta|^2\dy\\&=\frac{1}{2}\int_\omega|\naby\eta_*|^2\dy+\frac{1}{2}\int_\omega|\naby\Dely\eta_0|^2\dy-\int_{I_\ast}\int_\omega(g+\bfF)\,\partial_t\Dely\eta\dy\dt,
\end{aligned}
\end{align}
where $\bfF$ has been introduced in \eqref{Fdef}.

Using a similar argument as in \eqref{again} to control $\bfF$ by $\bftau$ and as in $\mathrm{I}-\mathrm{III}$ and $\mathrm{V}$ to estimate $\bftau$, we have
\begin{align*}
&\int_{I_\ast}\int_\omega\bfF\cdot\partial_t\Dely\eta\dy\dt\\
&\leq\int_{I_\ast}\|\bfF\|_{W^{1/2,2}(\omega)}\|\partial_t\Dely\eta\|_{W^{-1/2,2}(\omega)}\dt\\
&\lesssim\int_{I_\ast}\|\bftau\|_{W^{1/2,2}(\partial\Omega_\eta)}\|\partial_t\eta\|_{W^{3/2,2}(\omega)}\dt
\\
&\lesssim\int_{I_\ast}\big(\|\partial_t\bu+\bu\cdot\nabla\bu\|_{L^{2}_\bx}+\|\bff\|_{L^2_\bx}+\|\partial_t\eta\|_{W^{3/2,2}_\by}+\|\partial_t^2\eta\|_{L^2_\by}+\|g\|_{L^2_\by}\big)\|\partial_t\eta\|_{W^{3/2,2}_\by}\dt
\\
&\leq\kappa\int_{I_\ast}\big(\|\partial_t\bu+\bu\cdot\nabla\bu\|_{L^{2}_\bx}^2+\|\bff\|_{L^2_\bx}^2+\|\partial_t^2\eta\|_{L^2_\by}^2+\|g\|_{L^2_\by}^2\big)\dt+c(\kappa)\int_{I_\ast}\|\partial_t\eta\|_{W^{3/2,2}_\by}^2\dt
\\
&\leq\kappa\int_{I_\ast}\big(\|\partial_t\bu+\bu\cdot\nabla\bu\|_{L^{2}_\bx}^2+\|\partial_t^2\eta\|_{L^2_\by}^2+\|\partial_t\Dely\eta\|_{L^2_\by}^2+\|\bff\|_{L^2_\bx}^2\big)\dt
\\
&\quad+c(\kappa)\int_{I_\ast}\|\partial_t\eta\|_{W^{1,2}_\by}^2\dt+c(\kappa)\int_{I_\ast}\|g\|_{L^{2}_\by}^2\dt,
\end{align*}
where we used the interpolation \eqref{interpoeta} and \eqref{equivNorm} in the last step.
Hence we derive from \eqref{eq:apriorieta} and Gr\"onwall inequality that
\begin{align}\label{eq:reg3}
\begin{aligned}
&\sup_{I_\ast}\int_\omega|\partial_t\naby\eta|^2\dy+\int_{I_\ast}\int_\omega|\partial_t\Dely\eta|^2\dy\dt+\sup_{I_\ast}\int_\omega|\naby\Dely\eta|^2\dy\\&\leq\kappa\int_{I_\ast}\big(\|\partial_t\bu+\bu\cdot\nabla\bu\|_{L^{2}_\bx}^2 +\|\partial_t^2\eta\|_{L^2_\by}^2\big)\dt+ c(\kappa)\tilde C_0,
\end{aligned}
\end{align}
where 
\begin{align*}
\tilde C_0=C_0+\int_\omega|\naby\eta_*|^2\dy+\int_\omega|\naby\Dely\eta_0|^2\dy.
\end{align*}
Combining \eqref{eq:reg2} and \eqref{eq:reg3}, we arrive at 
\begin{align*}
&\sup_{I_\ast}\int_{\Omega_\eta}|\nabla\bu|^2\dx+\int_{I_\ast}\int_{\Omega_\eta}|\partial_t\bu+\bu\cdot\nabla\bu|^2\dx\dt
+\int_{I_\ast}\int_\omega|\partial_t^2\eta|^2\dy\dt\\
&\quad+\sup_{I_\ast}\int_\omega|\partial_t\naby\eta|^2\dy+\int_{I_\ast}\int_\omega|\partial_t\Dely\eta|^2\dy\dt+\sup_{I_\ast}\int_\omega|\naby\Dely\eta|^2\dy
\\
&\lesssim \int_{I_\ast}
\|\bu\|_{L^s_\bx}^{\frac{2s}{s-3}}\|\nabx\bu\|^{2}_{L^{2}_\bx}\dt+
\int_{I_\ast}\|\partial_t\eta\|_{W^{1,2}_\by}^2\|\nabla\bu\|_{L^2_\bx}^2\dt+C.
\end{align*}
Note that the condition \eqref{eq:regu} and \eqref{eq:apriorieta} imply that $\int_{I_\ast}\big(\|\bu\|_{L^s_\bx}^{\frac{2s}{s-3}}+\|\partial_t\eta\|_{W^{1,2}_\by}^2\big)\dt\leq\,c$ with a constant $c$ depending on $C_1$. Therefore, we obtain from Gr\"onwall's lemma that
\begin{align}\label{eq:reg4}
\begin{aligned}
\sup_{I_\ast}\int_{\Omega_\eta}&|\nabla\bu|^2\dx+\int_{I_\ast}\int_{\Omega_\eta}|\partial_t\bu+\bu\cdot\nabla\bu|^2\dx\dt
+\int_{I_\ast}\int_\omega|\partial_t^2\eta|^2\dy\dt\leq\,c,\\
\sup_{I_\ast}\int_\omega&\big(|\partial_t\naby\eta|^2
+
|\naby\Dely\eta|^2
\big)\dy
+
\int_{I_\ast}\int_\omega|\partial_t\Dely\eta|^2\dy\dt \leq c.
\end{aligned}
\end{align}
We can now use the momentum equation and \eqref{eq:regstokes} again to obtain (recall \eqref{again})
\begin{align}\label{eq:reg5}
\begin{aligned}
&\int_{I_\ast}\int_{\Omega_\eta}|\nabla^2\bu|^2\dx\dt+\int_{I_\ast}\int_{\Omega_\eta}|\nabla\pi|^2\dx\dt\\&\leq\,c\int_{I_\ast}\int_{\Omega_\eta}|\partial_t\bu+\bu\cdot\nabla\bu|^2\dx\dt+\int_{I_*}\|\bff\|_{L^2_\bx}^2\dt
+\int_{I_\ast}\|\partial_t\eta\|_{W^{3/2,2}_\by}^2\dt\leq\,c.
\end{aligned}
\end{align}
At this point, we notice that the only term required to obtain \eqref{est:reg} is a uniform-in-time bound for $\int_{I_\ast}\int_{\Omega_\eta}|\partial_t\bu|^2\dx\dt$. Since by \eqref{eq:reg5}, all the terms on the right-hand side of the momentum equation \eqref{2} are squared integrable in space-time, our desired estimate follows once we show that the convective term $\bu\cdot\nabx\bu$ is also squared integrable in space-time. Note that a bound for the sum $\partial_t\bv+\bu\cdot\nabx\bu$ as given in \eqref{eq:reg5} does not directly yield a bound for $\bu\cdot\nabx\bu$. Nevertheless, combining \eqref{eq:reg4}, \eqref{eq:reg5} and Sobolev embedding, we thereby obtain that
\begin{align*}
\int_{I_\ast}&\int_{\Omega_\eta}|\bu\cdot\nabla\bu|^2\dx\dt
\lesssim
\int_{I_\ast}\Vert \bu\Vert_{L^4_\bx}^2\Vert\nabx \bu\Vert_{L^4_\bx}^2\dt
\lesssim
\sup_{I_\ast}\Vert \bu\Vert_{W^{1,2}_\bx}^2\int_{I_\ast}\Vert\nabx \bu\Vert_{W^{1,2}_\bx}^2\leq c,
\end{align*}
which completes the proof.
\end{proof}

\begin{remark}
{\rm 
We remark here that the conditions we proposed in Theorem \ref{prop2} is the minimal assumption for the conditional strong solution for the fluid-structure system \eqref{1}--\eqref{interfaceCond}. The Serrin condition for the velocity of the fluid \eqref{eq:regu} is crucial in the estimate of the convective term and the Lipschitz condition for the structure \eqref{eq:regeta} plays an important role in the steady Stokes estimate related to the pressure estimate.
}
\end{remark}

\section{Weak-strong uniqueness}
\label{sec:weakStrong}
In this section, we are interested in the weak-strong uniqueness of the solutions for 
the fluid-structure interaction system \eqref{1}--\eqref{interfaceCond}. 
We aim to compare two solutions
$(\bu_1,\eta_1)$ and $(\bu_2,\eta_2)$, where $(\bu_1,\eta_1)$ is a weak solution satisfying \eqref{eq:regeta}  and $(\bu_2,\eta_2)$ is a strong solution, i.e. satisfies \eqref{est:reg}. Since the fluid domain depends on the deformation 
of the shell, we have to transfer the {\bf strong} solution by means of a change of 
variables to the {\bf weak} domain. We transform $\bu_2$ and $\pi_2$ (note that we have a pressure for the strong solution but not for the weak one) to the domain of the weak solution (that is $\Omega_{\eta_1}$) by setting
\begin{align}
\label{mapTwoToOne}
\underline{\bu}_2:=\bu_2\circ\bfPsi_{\eta_2-\eta_1},\quad \underline\pi_2:=\pi_2\circ\bfPsi_{\eta_2-\eta_1},
\quad \underline{\mathbf f}_2:=\mathbf f_2\circ\bfPsi_{\eta_2-\eta_1},
\end{align}
where the Hanzawa transform $\bm{\Psi}_{\eta_2-\eta_1} : \Omega_{\eta_1} \rightarrow\Omega_{\eta_2}$ is defined similarly as in \eqref{map}. With this information, we are now in the position to state our main result in this section.
\begin{theorem}\label{thm:weakstrong}
Let $(\bu_1,\eta_1)$ be a weak solution of \eqref{1}--\eqref{interfaceCond} with data $(\bff_1, g_1, \eta_{0,1}, \eta_{*,1}, \bu_{0,1})$ in the sense of Definition \ref{def:weakSolution} and let $(\bu_2,\eta_2)$ be a strong solution of \eqref{1}--\eqref{interfaceCond} with data $(\bff_2, g_2, \eta_{0,2}, \eta_{*,2}, \bu_{0,2})$ in the sense of Definition \ref{def:strongSolution}.
Suppose further that
\begin{equation*}\label{eq:regeta'}
\eta_1\in L^\infty(I;C^{0,1}(\omega))
\end{equation*}
and define 
$\underline{\bu}_2$ and $\underline{\mathbf f}_2$ in accordance with \eqref{mapTwoToOne}. Assume that
$\mathbf f_1,\underline{\mathbf f}_2\in L^2(I;L^2(\Omega_{\eta_1}))$ and $g_1,g_2\in L^2(I;L^2(\omega))$.
Then we have
\begin{equation}\label{final}
\begin{aligned}
&\sup_{t\in I}\int_{\Omega_{\eta_1(t)}}|\bu_1(t)-\underline{\bu}_2(t)|^2\dx+\sup_{t\in I}\int_\omega\big(|\partial_t(\eta_1-\eta_2)(t)|^2
+
|\Dely(\eta_1-\eta_2)(t)|^2\big)\dy
\\
&\qquad\qquad\qquad\qquad+\int_I
\int_{\Omega_{\eta_1(\sigma)}}|\nabla(\bu_1-\underline{\bu}_2)|^2\dx\dt 
+
\int_I\int_\omega|\partial_t\nabla_{\by}(\eta_1-\eta_2)|^2\dy\dt
\\
&\lesssim \int_{\Omega_{\eta_{0,1}}}|\bu_1(0)-\underline{\bu}_2(0)|^2\dx+\int_\omega|\partial_t(\eta_1-\eta_2)(0)|^2\dy+\int_\omega|\Dely(\eta_1-\eta_2)(0)|^2\dy
\\
&\qquad\qquad\qquad\qquad +\int_I\int_{\Omega_{\eta_1}}| \mathbf f_1-\underline {\mathbf f}_2|^2\dx\dt+\int_I\int_\omega| g_1-g_2|^2\dy\dt.
\end{aligned}
\end{equation}
\end{theorem}
\begin{proof}
It turned out more suitable to perform the uniqueness and stability analysis on the weaker geometry given by $\eta_1$. We therefore transfer the strong solution $(\eta_2,v_2)$ to the geometry given by $\eta_1$.
With the transformation \eqref{mapTwoToOne} in hand, we obtain the equations for $(\underline \bu_2, \eta_2)$ in $\Omega_{\eta_1}$ as follows:
\begin{align}
\label{contEqAloneBar'}
&\mathbf{B}_{\eta_2-\eta_1}:\nabx \underline{\bu}_2= 0,
\\
&\partial_t^2\eta_2 - \partial_t\Dely \eta_2 + \Dely^2\eta_2
=
g_2-\bn^\intercal \big[\mathbf{A}_{\eta_2-\eta_1}\nabx \underline{\bu}_2
-\mathbf{B}_{\eta_2-\eta_1} \underline{\pi}_2\big]\circ\bm{\varphi}_{\eta_1}\bn_{\eta_1} ,
\label{shellEqAloneBar'}
\\
&\partial_t \underline{\bu}_2  -\Delx\underline{\bu}_2
 +\nabla\underline{\pi}_2 
 = 
\mathbf{h}_{12}(\underline{\bu}_2)+
\divx  \big[\big(\mathbf{A}_{\eta_2-\eta_1}-\mathbb I_{3\times 3}\big)\nabx \underline{\bu}_2
+\big(\mathbb I_{3\times 3}-\mathbf{B}_{\eta_2-\eta_1} \big)\underline{\pi}_2\big],
\label{momEqAloneBar'}
\end{align}
where
\begin{align*}
\mathbf{h}_{12}(\underline{\bu}_2)
&=
(\mathbb I_{3\times 3}-J_{\eta_2-\eta_1})\partial_t \underline{\bu}_2
-
 J_{\eta_2-\eta_1}\nabx \underline{\bu}_2 \partial_t \bfPsi_{\eta_2-\eta_1}^{-1}\circ \bfPsi_{\eta_2-\eta_1} 
-
\mathbf B_{\eta_2-\eta_1}\nabx \underline{\bu}_2  \underline{\bu}_2
+
J_{\eta_2-\eta_1}  \underline{\mathbf f}_2,
\end{align*}
and the matrices $\mathbf{A}_{\eta_2-\eta_1}$ and $\mathbf{B}_{\eta_2-\eta_1}$ are similarly defined as in Subsection \ref{Sectionlinear} by replacing the subscript $\eta$ by $\eta_2-\eta_1$, respectively. 

Next we introduce a suitable bogovskij operator for our setting:
\[
\Bog_{\eta_1}(f):=\Bog(f\chi_{\Omega_{\eta_1}}),
\] 
where $\Bog$ is defined in Theorem~\ref{thm:ndBog}, depending on $\norm{\eta_1}_\infty=:L$ and $\norm{\nabla\eta_1}_\infty=:C_L$ only. Please note that this is the point where the additional Lipschitz assumption of the weak solution is crucially needed.

To obtain the difference estimate, we would like to test the equation for $(\bu_1-\underline\bu_2,\eta_1-\eta_2)$ by the pair $(\bu_1-\underline\bu_2+\mathrm{Bog}_{\eta_1}(\Div\underline\bu_2),\partial_t(\eta_1-\eta_2))$.
 However, $\bu_1$ is not smooth enough to qualify as a test function for the weak equation. We thus consider the following procedure:
In the first step, we use the energy inequality for $(\bu_1,\eta_1)$, that is
\begin{equation*}
\begin{aligned}
\label{energyEst'}
&\tfrac{1}{2}\int_\omega\vert\partial_t\eta_1 \vert^2\dy+\tfrac{1}{2}\int_\omega\vert\Dely\eta_1 \vert^2\dy
+
\int_0^t\int_\omega\vert\partial_t\naby\eta_1 \vert^2\dy\ds
\\&\quad+ \tfrac{1}{2} \int_{\Omega_{\eta_1(t)}}\vert\bu_1 \vert^2 \dx+
\int_0^t
 \int_{\Omega_{\eta_1(\sigma)}}\vert \nabla \bu_1 \vert^2 \dx\ds
\\& \leq \tfrac{1}{2}\int_\omega\vert\partial_t\eta_1(0) \vert^2\dy
 +\tfrac{1}{2}\int_\omega\vert\Dely\eta_1(0) \vert^2\dy+\tfrac{1}{2} \int_{\Omega_{\eta_1(0)}}\vert\bu_1(0) \vert^2 \dx\\
&\quad+
\int_0^t \int_\omega  g_1\partial_t\eta_1\dy\ds
 +\int_0^t \int_{\Omega_{\eta_1(\sigma)}} \bu_1\cdot \mathbf{f}_1 \dx\ds .
\end{aligned}
\end{equation*}
Next observe that
\begin{align*}
\int_0^t\int_\omega\Dely\eta_1\,\partial_t\Dely \eta_2\dy\ds
=-\int_0^t\int_\omega\partial_t\eta_1\cdot\Dely^2 \eta_2\dy\ds + 
\left[\int_\omega\Dely\eta_1\Dely \eta_2\dy\right]^{\sigma=t}_{\sigma=0},
\end{align*}
an identity that can be rigorously shown by using convolution  in space. This implies, by testing the equation for $(\bu_1,\eta_1)$ with  $(-\underline \bu_2+\mathrm{Bog}_{\eta_1}(\Div\underline\bu_2),-\partial_t\eta_2)$, that 
\begin{align*}
&\int_\omega \left(-\Dely\eta_1(t)\Dely\eta_2(t)-\partial_t \eta_1(t) \, \partial_t\eta_2(t)\right)\dy +\int_\omega \Dely\eta_1(0)\Dely\eta_2(0)\dy
\\&\quad-\int_{\Omega_{\eta_1(t)}} \bu_1(t)  
\cdot (\underline\bu_2(t)-\mathrm{Bog}_{\eta_1}(\Div\underline\bu_2(t)))\dx
\\
&\quad+
\int_0^t \int_\omega \big(\partial_t \eta_1\, \partial_t^2 \eta_2-\partial_t\naby\eta_1\cdot\partial_t\naby  \eta_2+
 g_1\, \partial_t\eta_2+\partial_t\eta_1\Dely^2 \eta_2 \big)\dy\ds
 \\
 &=
-\int_\omega \partial_t \eta_1(t) \, \partial_t\eta_2(t)\dy
-\int_{\Omega_{\eta_1(t)}} \bu_1(t)  
\cdot (\underline\bu_2(t)-\mathrm{Bog}_{\eta_1}(\Div\underline\bu_2(t)))\dx
\\
&\quad+
\int_0^t \int_\omega \big(\partial_t \eta_1\, \partial_t^2 \eta_2-\partial_t\naby\eta_1\cdot\partial_t\naby  \eta_2+
 g_1\, \partial_t\eta_2-\Dely\eta_1\,\partial_t\Dely \eta_2 \big)\dy\ds.
\\
&=-\int_\omega \partial_t \eta_1(0) \, \partial_t\eta_2(0) \dy
+
\int_{\Omega_{\eta_1(0)}} \bu_1(0)  \cdot (-\underline \bu_2(0)+\mathrm{Bog}_{\eta_1}(\Div\underline\bu_2(0)))\dx
\\
&\quad+\int_0^t  \int_{\Omega_{\eta_1(\sigma)}} \big(  \bu_1\cdot \partial_t (-\underline\bu_2+\mathrm{Bog}_{\eta_1}(\Div\underline\bu_2)) 
 \big) + \bu_1 \otimes \bu_1: \nabla  \mathrm{Bog}_{\eta_1}(\Div\underline\bu_2))
 \big) \dx\ds\\
&\quad+\int_0^t  \int_{\Omega_{\eta_1(\sigma)}}(\bu_1\cdot\nabla)\bu_1\cdot\underline\bu_2\dx\ds-\int_0^t  \int_{\partial\Omega_{\eta_1(\sigma)}} |\bu_1|^2\bfn\cdot\partial_t\eta_2 \bn_{\eta_1 }\circ\bfvarphi^{-1}_ {\eta_1} \dd\mathcal H^2\ds
\\
&\quad+\int_0^t  \int_{\Omega_{\eta_1(\sigma)}} \big( 
 -  
\nabla \bu_1:\nabla  (-\underline\bu_2+\mathrm{Bog}_{\eta_1}(\Div\underline\bu_2))+\bff_1\cdot (-\underline\bu_2+\mathrm{Bog}_{\eta_1}(\Div\underline\bu_2)) \big) \dx\ds.
\end{align*}
Finally, we multiply the (strong) equation for $(-\underline\bu_2,-\eta_2)$ by $(\bu_1-\underline\bu_2+\mathrm{Bog}_{\eta_1}(\Div\underline\bu_2),\partial_t(\eta_1-\eta_2))$. This implies after integration by parts
\begin{align*}
&\int_\omega \left(\frac{\abs{\partial_t \eta_2(t)}^2}{2}+\frac{\abs{\Dely\eta_2(t)}^2}{2}\right)\dy -\int_0^t\int_\omega \partial_t\naby\eta_2\cdot\partial_t\naby(\eta_1-\eta_2)\dy\dt
+\tfrac{1}{2} \int_{\Omega_{\eta_1(t)}}|\underline\bu_2 |^2 \dx
\\
&\quad-\tfrac{1}{2}  \int_0^t\int_{\partial\Omega_{\eta_1(\sigma)}}\bfn\circ\bfvarphi^{-1}_ {\eta_1}\cdot\partial_t\eta_1 \bn_{\eta_1 }\circ\bfvarphi^{-1}_ {\eta_1}|\underline\bu_2 |^2 \dd \mathcal H^2\dd \sigma
-\int_0^t\int_{\Omega_{\eta_1(\sigma)}}\partial_t\underline\bu_2 \cdot\bu_1 \dx\ds
\\
&\quad+ \int_0^t\int_{\Omega_{\eta_1(\sigma)}}\nabla \underline\bu_2:\nabla(\underline\bu_2-\bu_1)\dx\ds
\\
&=\tfrac{1}{2} \int_{\Omega_{\eta_1(0)}}|\underline\bu_2(0) |^2 \dx+\int_\omega \left(\frac{\abs{\partial_t \eta_2(0)}^2}{2}+\frac{\abs{\Dely\eta_2(0)}^2}{2}\right)\dy 
\\
&\quad-\int_0^t\int_{\Omega_{\eta_1(\sigma)}}\mathbf h_{12}(\underline\bu_2) \cdot\big(\bu_1- \underline\bu_2+ \mathrm{Bog}_{\eta_1}(\Div\underline\bu_2)\big)\dx\ds
\\
&\quad-
\int_0^t\int_{\Omega_{\eta_1(\sigma)}}\underline\bu_2 \cdot \partial_t\mathrm{Bog}_{\eta_1}(\Div\underline\bu_2)\dx\ds+
\int_{\Omega_{\eta_1}}\underline\bu_2 \cdot \mathrm{Bog}_{\eta_1}(\Div\underline\bu_2)\dx
\\
&\quad-
\int_{\Omega_{\eta_1(0)}}\underline\bu_2(0) \cdot \mathrm{Bog}_{\eta_1(0)}(\Div\underline\bu_2(0))\dx
+ \int_0^t\int_{\Omega_{\eta_1(\sigma)}}\nabla \underline\bu_2:\nabla\mathrm{Bog}_{\eta_1}(\Div\underline\bu_2)\dx\ds
\\
&\quad+\int_0^t\int_{\Omega_{\eta_1(\sigma)}}\big(\mathbf A_{\eta_2-\eta_1}-\mathbb I_{3\times 3}\big)\nabla\underline\bu_2:\nabla(\bu_1-\underline \bu_2+\mathrm{Bog}_{\eta_1}(\Div\underline\bu_2))\dx\ds\\
&\quad+\int_0^t\int_{\Omega_{\eta_1(\sigma)}}\big(\mathbb I_{3\times 3}-\mathbf B_{\eta_2-\eta_1}\big)\underline \pi_2:\nabla(\bu_1-\underline \bu_2+\mathrm{Bog}_{\eta_1}(\Div\underline\bu_2))\dx\ds\\
&\quad+\int_0^t\int_\omega\partial_t^2\eta_2~\partial_t\eta_1\dy\ds-\int_0^t\int_\omega\naby\Dely\eta_2~\partial_t\naby\eta_1\dy\ds-\int_0^t\int_\omega g_2\partial_t(\eta_1-\eta_2)\dy\ds\\
&\quad +\int_0^t\int_\omega\partial_t\eta_2~\partial_t\eta_1\dy\ds-\int_\omega\partial_t\eta_2(t)~\partial_t\eta_1(t)\dy.
\end{align*}

Combining the above 
we find that
\begin{align}
&\tfrac{1}{2}\int_{\Omega_{\eta_1(t)}}|\bu_1(t)-\underline{\bu}_2(t)|^2\dx+\int_0^t
\int_{\Omega_{\eta_1(\sigma)}}|\nabla(\bu_1-\underline{\bu}_2)|^2\dx\ds\nonumber\\
&\quad+\tfrac{1}{2}\int_\omega|\partial_t(\eta_1-\eta_2)(t)|^2\dy+\int_0^t\int_\omega|\partial_t\nabla_{\by}(\eta_1-\eta_2)|^2\dy\ds+\tfrac{1}{2}\int_\omega|\Dely(\eta_1-\eta_2)(t)|^2\dy\nonumber\\
&\leq 
\tfrac{1}{2}\int_{\Omega_{\eta_1(0)}}|\bu_1(0)-\underline{\bu}_2(0)|^2\dx+\tfrac{1}{2}\int_\omega|\partial_t(\eta_1-\eta_2)(0)|^2\dy+\tfrac{1}{2}\int_\omega|\Dely(\eta_1-\eta_2)(0)|^2\dy\nonumber\\
&\quad-\int_0^t\int_{\Omega_{\eta_1(\sigma)}}\big((\mathbb I_{3\times 3}-J_{\eta_2-\eta_1})\partial_t \underline{\bu}_2\big)\cdot\big(\bu_1- \underline\bu_2+ \mathrm{Bog}_{\eta_1}(\Div\underline\bu_2)\big)\dx\ds\nonumber\\
&\quad+\int_0^t\int_{\Omega_{\eta_1(\sigma)}}\big(
 J_{\eta_2-\eta_1}\nabx \underline{\bu}_2 \partial_t \bfPsi_{\eta_2-\eta_1}^{-1}\circ \bfPsi_{\eta_2-\eta_1}\big)\cdot\big(\bu_1- \underline\bu_2+ \mathrm{Bog}_{\eta_1}(\Div\underline\bu_2)\big)\dx\ds\nonumber\\
&\quad+\int_0^t\int_{\Omega_{\eta_1(\sigma)}}\nabx \underline{\bu}_2  \underline{\bu}_2\cdot\big(\bu_1- \underline\bu_2+ \mathrm{Bog}_{\eta_1}(\Div\underline\bu_2)\big)\dx\ds\nonumber
\\
&\quad+\tfrac{1}{2} \int_0^t\int_{\partial\Omega_{\eta_1(\sigma)}}\bfn\circ\bfvarphi^{-1}_ {\eta_1}\cdot\partial_t\eta_1 \bn_{\eta_1 }\circ\bfvarphi^{-1}_ {\eta_1}|\underline\bu_2 |^2 \dd \mathcal H^2\ds\nonumber\\
&\quad- \int_0^t\int_{\partial\Omega_{\eta_1(\sigma)}}\bfn\circ\bfvarphi^{-1}_ {\eta_1}\cdot\partial_t\eta_2 \bn_{\eta_1 }\circ\bfvarphi^{-1}_ {\eta_1}| \bu_1 |^2 \dd \mathcal H^2\ds\nonumber
\\
&\quad
+\int_0^t\int_{\Omega_{\eta_1(\sigma)}}\big( 
\mathbf B_{\eta_2-\eta_1}- \mathbb I
_{3\times 3}\big)\nabx \underline{\bu}_2  \underline{\bu}_2\cdot\big(\bu_1- \underline\bu_2+ \mathrm{Bog}_{\eta_1}(\Div\underline\bu_2)\big)\dx\ds\nonumber\\
&\quad+
\int_0^t\int_{\Omega_{\eta_1(\sigma)}}(\bu_1-\underline\bu_2) \cdot \partial_t\mathrm{Bog}_{\eta_1}(\Div\underline\bu_2)\dx\ds-
\int_{\Omega_{\eta_1(t)}}(\bu_1-\underline\bu_2) \cdot \mathrm{Bog}_{\eta_1}(\Div\underline\bu_2)\dx\nonumber\\
&\quad+
\int_{\Omega_{\eta_1(0)}}(\bu_1-\underline\bu_2)(0) \cdot \mathrm{Bog}_{\eta_1(0)}(\Div\underline\bu_2(0))\dx
- \int_0^t\int_{\Omega_{\eta_1(\sigma)}}\nabla(\bu_1- \underline\bu_2):\nabla\mathrm{Bog}_{\eta_1}(\Div\underline\bu_2)\dx\ds\nonumber
\\
&\quad+\int_0^t\int_{\Omega_{\eta_1(\sigma)}}\big(\mathbf A_{\eta_2-\eta_1}-\mathbb I_{3\times 3}\big)\nabla\underline\bu_2:\nabla(\bu_1-\underline \bu_2+\mathrm{Bog}_{\eta_1}(\Div\underline\bu_2))\dx\ds\nonumber\\
&\quad+\int_0^t\int_{\Omega_{\eta_1(\sigma)}}\big(\mathbb I_{3\times 3}-\mathbf B_{\eta_2-\eta_1}\big)\underline \pi_2:\nabla(\bu_1-\underline \bu_2+\mathrm{Bog}_{\eta_1}(\Div\underline\bu_2))\dx\ds\nonumber\\
&\quad+\int_0^t \int_\omega ( g_1-g_2)\partial_t(\eta_1-\eta_2)\dy\ds
 +\int_0^t \int_{\Omega_{\eta_1(\sigma)}}  (\bu_1-\underline \bu_2)\cdot (\mathbf{f}_1-\underline{\mathbf f}_2)\dx\ds\nonumber\\
 &\quad+\int_0^t \int_{\Omega_{\eta_1(\sigma)}} (\mathbb I_{3\times 3}-J_{ \eta_2-\eta_1})\underline{\mathbf f}_2\cdot (\bu_1-\underline \bu_2)\dx\ds+\int_0^t\int_{\Omega_{\eta_1}}(\mathbb I_{3\times 3}-J_{\eta_2-\eta_1})\underline{\mathbf f}_2\mathrm{Bog}_{\eta_1}(\Div \underline \bu_2)\dx\ds\nonumber\\
 &\quad+\int_0^t\int_{\Omega_{\eta_1(\sigma)}}\bu_1\otimes \bu_1:\nabla\mathrm{Bog}_{\eta_1}(\Div\underline \bu_2)\dx\ds+\int_0^t\int_{\Omega_{\eta_1(\sigma)}}\bu_1\cdot\nabla\bu_1\cdot\underline \bu_2\dx\ds\nonumber\\
 &\quad+\int_0^t\int_{\Omega_{\eta_1(\sigma)}}(\mathbf f_1-\underline{\mathbf f}_2)\mathrm{Bog}_{\eta_1}(\Div\underline \bu_2)\dx\ds\nonumber\\
 &=:\sum_{i=1}^{22} R_{i}.\label{mainineq}
\end{align}
Note that $R_1$, $R_2$ and $R_3$ are in good form and so we start the estimate for the remaining terms on the right side of \eqref{mainineq}.
For the terms including the forcing terms, we estimate as follows:
\begin{align*}
&R_{12}+R_{16}+R_{17}+R_{18}+R_{19}+R_{22}\\
&\leq
 \delta\int_0^t\left(\lVert\bu_1-\underline \bu_2\rVert_{L^2(\Omega_{\eta_1})}^2+\lVert\nabla(\bu_1-\underline \bu_2)\rVert_{L^2(\Omega_{\eta_1})}^2 +\lVert\partial_t\eta_1-\partial_t\eta_2\rVert_{L^2(\omega)}^2\right)\ds\\
 &\quad+C(\delta)\int_0^t\lVert\underline {\mathbf f}_2\rVert_{L^2(\Omega_{\eta_1})}^2\lVert\eta_1-\eta_2\rVert_{W^{2,2}(\omega)}^2\ds+C(\delta)\int_0^t\left(\lVert g_1-g_2\rVert_{L^2(\omega)}^2+\lVert \mathbf f_1-\underline {\mathbf f}_2\rVert_{L^2(\Omega_{\eta_1})}^2\right)\ds
 \\
 &\quad+C(\delta)\lVert \bu_1(0)-\underline \bu_2(0)\rVert_{L^2(\Omega_{\eta_1(0)})}^2.
\end{align*}

Using the fact that $\mathbb I_{3\times 3}-J_{\eta_2-\eta_1}\sim \nabla_{\by}(\eta_2-\eta_1)$ we estimate $R_4$:
\begin{align*}
R_4&\lesssim 
\int_0^t\lVert\mathbb I_{3\times 3}-J_{\eta_2-\eta_1}\rVert_{L^4(\Omega_{\eta_1(\sigma)})}\lVert\partial_t\underline \bu_2\rVert_{L^2(\Omega_{\eta_1(\sigma)})}\lVert\bu_1- \underline\bu_2+ \mathrm{Bog}_{\eta_1}(\Div\underline\bu_2)\rVert_{L^4(\Omega_{\eta_1(\sigma)})}\ds\\
&\lesssim \int_0^t\lVert\mathbb I_{3\times 3}-J_{\eta_2-\eta_1}\rVert_{W^{1,2}(\Omega_{\eta_1(\sigma)})}\lVert\partial_t\underline \bu_2\rVert_{L^2(\Omega_{\eta_1(\sigma)})}\lVert\bu_1- \underline\bu_2+ \mathrm{Bog}_{\eta_1}(\Div\underline\bu_2)\rVert_{W^{1,2}(\Omega_{\eta_1(\sigma)})} \ds \\
&\leq \delta\int_0^t\lVert\nabla(\bu_1-\underline \bu_2)\rVert_{L^2(\Omega_{\eta_1(\sigma)})}^2\ds +C(\delta) \int_0^t\lVert\partial_t\underline \bu_2\rVert_{L^2(\Omega_{\eta_1(\sigma)})}^2\lVert\eta_1-\eta_2\rVert_{W^{2,2}(\omega)}^2\ds,
\end{align*}
where we used the estimate:
\begin{align*}
\lVert \mathrm{Bog}_{\eta_1}(\Div\underline \bu_2)\rVert_{W^{1,2}(\Omega_{\eta_1})}=\lVert \mathrm{Bog}_{\eta_1}(\Div(\underline \bu_2-\bu_1))\rVert_{W^{1,2}(\Omega_{\eta_1})}\lesssim \lVert\bu_1-\underline \bu_2\rVert_{W^{1,2}(\Omega_{\eta_1})}.
\end{align*}
According to the properties of the map $\bfPsi_{\eta_2-\eta_1}$ discussed in Section \ref{ssec:geom}, we continue estimating $R_5$:
\begin{align*}
R_5
& \lesssim \int_0^t\lVert\nabla\underline \bu_2\rVert_{L^4(\Omega_{\eta_1(\sigma)})}\lVert\partial_t\bfPsi_{\eta_2-\eta_1}^{-1}\circ\bfPsi_{\eta_2-\eta_1}\rVert_{L^4(\Omega_{\eta_1(\sigma)})}\lVert\bu_1-\underline \bu_2\rVert_{L^2(\Omega_{\eta_1(\sigma)})}\ds\\
&\lesssim \int_0^t\lVert \underline \bu_2\rVert_{W^{2,2}(\Omega_{\eta_1(\sigma)})}\lVert\partial_t(\eta_1-\eta_2)\rVert_{L^4(\omega)}\lVert\bu_1-\underline \bu_2\rVert_{L^2(\Omega_{\eta_1(\sigma)})}\ds\\
&\leq\delta\int_0^t\lVert\partial_t\nabla_{\by}(\eta_1-\eta_2)\rVert_{L^2(\omega)}^2\ds+C(\delta)\int_0^t\lVert\underline \bu_2\rVert_{W^{2,2}(\Omega_{\eta_1(\sigma)})}^2\lVert\bu_1-\underline \bu_2\rVert_{L^2(\Omega_{\eta_1(\sigma)})}^2\ds.
\end{align*}
For the $R_9$ term, we use the fact that $\Vert\nabx\underline{\bv}_2\Vert_{L^2_\bx}$ is essentially bounded in time and $\mathbb I_{3\times 3}-\mathbf B_{\eta_1-\eta_2}\sim \nabla_{\by}(\eta_1-\eta_2)$ to obtain
\begin{align*}
R_9
&\lesssim\int_0^t\Vert I_{3\times 3}-\mathbf B_{\eta_1-\eta_2}\Vert_{L^4(\Omega_{\eta_1(\sigma)})}\Vert\underline{\bv}_2\Vert_{L^\infty(\Omega_{\eta_1(\sigma)})}
\Vert \bv_1-\underline{\bv}_2\Vert_{L^4(\Omega_{\eta_1(\sigma)})}\ds\\
&\leq\delta\int_0^t\Vert \nabx(\bv_1-\underline{\bv}_2)\Vert_{L^2(\Omega_{\eta_1(\sigma)})}^2\ds+C(\delta)\int_0^t
\Vert\underline{\bv}_2\Vert_{W^{2,2}(\Omega_{\eta_1(\sigma)})}^2
\Vert {\eta_1-\eta_2}\Vert_{W^{2,2}(\omega)}^2
\ds.
\end{align*}
We now start the estimate for $R_{10}$. Using the properties of the Bogovskij operator (please refer to Remark \ref{time-bog}), we have 
\begin{align*}
R_{10}&\leq 
\int_0^t\lVert\bu_1-\underline \bu_2\rVert_{L^6(\Omega_{\eta_1(\sigma)})}\lVert(\mathbb I_{3\times 3}-\mathbf B_{\eta_2-\eta_1})\partial_t \nabla\underline \bu_2\rVert_{W^{-1, \frac{6}{5}}(\Omega_{\eta_1(\sigma)})}\ds\\
&\quad+\int_0^t\lVert\bu_1-\underline \bu_2\rVert_{L^6(\Omega_{\eta_1(\sigma)})}\lVert\partial_t(\mathbf B_{\eta_2-\eta_1})\nabla\underline \bu_2\rVert_{W^{-1, \frac{6}{5}}(\Omega_{\eta_1(\sigma)})}\ds\\
&\lesssim \int_0^t\lVert\bu_1-\underline \bu_2\rVert_{W^{1,2}(\Omega_{\eta_1(\sigma)})}\lVert\mathbb I_{3\times 3}-\mathbf B_{\eta_2-\eta_1}\rVert_{W^{1,2}(\Omega(\eta_1(\sigma))}\lVert\partial_t\underline \bu_2\rVert_{L^2(\Omega_{\eta_1(\sigma)})}\\
&\quad+\int_0^t\lVert\bu_1-\underline \bu_2\rVert_{W^{1,2}(\Omega_{\eta_1(\sigma)})}\lVert\partial_t(\eta_1-\eta_2)\rVert_{L^2(\omega)}\lVert\nabla\underline \bu_2\rVert_{W^{1,2}(\Omega_{\eta_1(\sigma)})}\ds\\
&\leq \delta \int_0^t\lVert\nabla(\bu_1-\underline\bu_2)\rVert_{L^2(\Omega_{\eta_1(\sigma)})}^2+C(\delta)\int_0^t\lVert\partial_t\underline \bu_2\rVert_{L^2(\Omega_{\eta_1(\sigma)})}^2\lVert\eta_1-\eta_2\rVert_{W^{2,2}(\omega)}^2\ds\\
&\quad+C(\delta)\int_0^t\lVert\underline \bu_2\rVert_{W^{2,2}(\Omega_{\eta_1(\sigma)})}^2\lVert\partial_t(\eta_1-\eta_2)\rVert_{L^2(\omega)}^2\ds.
\end{align*}
To estimate $R_{11}$, we rewrite $\Div\underline \bu_2$ as follows:
\begin{equation}\label{divv2}
\Div\underline \bu_2=\mathbb I_{3\times 3}:\nabla\underline \bu_2=(\mathbb I_{3\times 3}-\mathbf B_{\eta_2-\eta_1}):\nabla\underline \bu_2,
\end{equation}
where we take into account the divergence free condition for $\underline \bu_2$ on $\Omega_{\eta_1}$ derived in \eqref{contEqAloneBar'}.
With \eqref{divv2}, we have
\begin{align*}
R_{11}
&\leq \lVert\bu_1-\underline\bu_2\rVert_{L^2(\Omega_{\eta_1})}\lVert\mathrm{Bog}\left((\mathbb I_{3\times 3}-\mathbf B):\nabla\underline\bu_2\right)\rVert_{L^2(\Omega_{\eta_1})}\\
&\leq \lVert\bu_1-\underline\bu_2\rVert_{L^2(\Omega_{\eta_1})}\lVert(\mathbb I_{3\times 3}-\mathbf B):\nabla\underline\bu_2\rVert_{W^{-1,2}(\Omega_{\eta_1})}\\
&\leq \delta\lVert\bu_1-\underline\bu_2\rVert_{ L^2(\Omega_{\eta_1})}^2+C(\delta)\lVert(\mathbb I_{3\times 3}-\mathbf B):\nabla\underline\bu_2\rVert_{ L^{\frac{6}{5}}(\Omega_{\eta_1})}^2\\
&\leq \delta\lVert\bu_1-\underline\bu_2\rVert_{ L^2(\Omega_{\eta_1})}^2+C(\delta)\lVert\nabla\underline\bu_2\rVert_{ L^2(\Omega_{\eta_1})}^2\lVert\nabla_{\by}(\eta_1-\eta_2)\rVert_{L^3(\omega)}^2\\
&\leq \delta\lVert\bu_1-\underline\bu_2\rVert_{ L^2(\Omega_{\eta_1})}^2+C(\delta)\lVert\nabla_{\by}(\eta_1-\eta_2)\rVert_{L^2(\omega)}^{\frac{4}{3}}\lVert\nabla_{\by}(\eta_1-\eta_2)\rVert_{W^{1,2}(\omega)}^{\frac{2}{3}}\\
&\leq \delta\lVert\bu_1-\underline\bu_2\rVert_{ L^2(\Omega_{\eta_1})}^2+\varepsilon\lVert\nabla_{\by}(\eta_1-\eta_2)\rVert_{W^{1,2}(\omega)}^2+C(\delta, \varepsilon)\lVert\nabla_{\by}(\eta_1-\eta_2)\rVert_{L^2(\omega)}^2\\
&\leq \delta\lVert\bu_1-\underline\bu_2\rVert_{ L^2(\Omega_{\eta_1})}^2+\varepsilon\lVert\nabla_{\by}(\eta_1-\eta_2)\rVert_{W^{1,2}(\omega)}^2\\
&\qquad\qquad \qquad +C(\delta, \varepsilon)\lVert\nabla_{\by}(\eta_1-\eta_2)\rVert_{L^2(0, t;L^2(\omega))}\lVert\naby(\eta_1-\eta_2)\rVert_{W^{1,2}(0,t; L^2(\omega)} 
\\
&\leq \delta\lVert\bu_1-\underline\bu_2\rVert_{ L^2(\Omega_{\eta_1})}^2+\varepsilon\lVert\nabla_{\by}(\eta_1-\eta_2)\rVert_{W^{1,2}(\omega)}^2+\nu\lVert\naby(\eta_1-\eta_2)\rVert_{W^{1,2}(0, t; L^2(\omega))}^2
\\
&\qquad\qquad\qquad +C(\delta, \varepsilon, \nu)\lVert\eta_1-\eta_2\rVert_{L^2(0, t; W^{2,2}(\omega))}^2,
\end{align*}
where we used interpolation in time as follows:
\begin{align*}
\lVert\nabla_{\by}(\eta_1-\eta_2)\rVert_{L^\infty(0, t; L^2(\omega))}\lesssim \lVert\nabla_{\by}(\eta_1-\eta_2)\rVert_{L^2(0, t; L^2(\omega))}^{\frac{1}{2}}\lVert\nabla_{\by}(\eta_1-\eta_2)\rVert_{W^{1,2}(0, t; L^2(\omega))}^{\frac{1}{2}}. 
\end{align*}
Using \eqref{divv2}, we have the estimate for $R_{13}$:
\begin{align*}
R_{13}&\leq \int_0^t\lVert\nabla(\bu_1-\underline \bu_2)\rVert_{L^2(\Omega_{\eta_1(\sigma)})}\lVert\nabla\mathrm{Bog}_{\eta_1}(\Div\underline \bu_2)\rVert_{L^2(\Omega_{\eta_1(\sigma)})}\ds\\
&\lesssim\int_0^t\lVert\nabla(\bu_1-\underline \bu_2)\rVert_{L^2(\Omega_{\eta_1(\sigma)})}\lVert(\mathbb I_{3\times 3}-\mathbf B_{\eta_2-\eta_1}):\nabla\underline \bu_2\rVert_{L^2(\Omega_{\eta_1(\sigma)})}\\
&\lesssim \int_0^t\lVert\nabla(\bu_1-\underline \bu_2)\rVert_{L^2(\Omega_{\eta_1(\sigma)})}\lVert\nabla_{\by}(\eta_1-\eta_2)\rVert_{L^4(\omega)}\lVert\nabla\underline \bu_2\rVert_{L^4(\Omega_{\eta_1(\sigma)})}\\
&\leq \delta\int_0^t\lVert\nabla(\bu_1-\underline \bu_2)\rVert_{L^2(\Omega_{\eta_1(\sigma)})}^2+C(\delta)\int_0^t\lVert\underline \bu_2\rVert_{W^{2,2}(\Omega_{\eta_1(\sigma)})}^2\lVert\eta_1-\eta_2\rVert_{W^{2,2}(\omega)}^2,
\end{align*}
where we also used 
\begin{align*}
\lVert\nabla\mathrm{Bog}_{\eta_1}(\Div\underline\bu_2)\rVert_{L^2(\Omega_{\eta_1(\sigma)})}\lesssim \lVert\Div\underline \bu_2\rVert_{L^2(\Omega_{\eta_1(\sigma)})}. 
\end{align*}
Recalling the regularity $\underline \pi_2\in L^2(I,W^{1,2}(\Omega_{\eta_1}))$and $\underline \bu_2\in L^2(I,W^{2,2}(\Omega_{\eta_1}))$, we estimate $R_{14}$ and $R_{15}$ in a similar way. Notice that $\mathbb I_{3\times 3}-\mathbf B_{\eta_1-\eta_2}\sim \nabla_{\by}(\eta_1-\eta_2)$ and $\mathbf A_{\eta_2-\eta_2}-\mathbb I_{3\times 3}\sim \nabla_{\by}(\eta_1-\eta_2)$, we thus have 
\begin{equation*}\label{14est}
R_{14}\leq \delta\int_0^t\lVert\nabla(\bu_1-\underline \bu_2)\rVert_{L^2(\Omega_{\eta_1(\sigma)})}^2\ds+C(\delta)\int_0^t\lVert\underline \bu_2\rVert_{W^{2,2}(\Omega_{\eta_1(\sigma)})}^2\lVert\eta_1-\eta_2\rVert_{W^{2,2}(\omega)}^2\ds,
\end{equation*}
and 
\begin{equation*}\label{15es}
R_{15}\leq \delta\int_0^t\lVert\nabla(\bu_1-\underline \bu_2)\rVert_{L^2(\Omega_{\eta_1(\sigma)})}^2\ds +C(\delta) \int_0^t\lVert\underline \pi_2\rVert_{W^{1,2}(\Omega_{\eta_1(\sigma)})}^2\lVert\eta_1-\eta_2\rVert_{W^{2,2}(\omega)}^2\ds.
\end{equation*}
Now we deal with the estimate of $R_6$, $R_7$, $R_8$ and $R_{20}$, $R_{21}$ together. We first take an integration by part of $R_{20}$ and obtain
\begin{equation*}
\begin{aligned}
R_{20}=-\int_0^t\int_{\Omega_{\eta_1(\sigma)}}\bu_1\cdot \nabla\bu_1\cdot \mathrm{Bog}_{\eta_1}(\Div\underline \bu_2)\dx\ds,
\end{aligned}
\end{equation*}
where we do not have the boundary term since the Bogovskij operator vanishes on the boundary. We then rewrite $R_6$ in the following way:
\begin{align*}
R_6
&=\int_0^t\int_{\Omega_{\eta_1(\sigma)}}(\underline \bu_2-\bu_1)\nabla\underline \bu_2(\bu_1-\underline\bu_2+\mathrm{Bog}_{\eta_1}(\Div\underline \bu_2))\dx\ds\\
&\quad+\int_0^t\int_{\Omega_{\eta_1(\sigma)}}(\bu_1-\underline\bu_2)\nabla(\underline\bu_2-\bu_1)(\bu_1-\underline\bu_2+\mathrm{Bog}_{\eta_1}(\Div\underline \bu_2))\dx\ds\\
&\quad+\int_0^t\int_{\Omega_{\eta_1(\sigma)}}\underline\bu_2\nabla(\underline\bu_2-\bu_1)(\bu_1-\underline\bu_2+\mathrm{Bog}_{\eta_1}(\Div\underline \bu_2))\dx\ds\\
&\quad+\int_0^t\int_{\Omega_{\eta_1(\sigma)}}\bu_1\nabla\bu_1(\bu_1-\underline\bu_2+\mathrm{Bog}_{\eta_1}(\Div\underline\bu_2))\dx\ds\\
&=R_{6,0}+R_{6,1}+R_{6,2}+R_{6,3}.
\end{align*}
For $R_{6,0}$, we have
\begin{align*}
R_{6,0}&\leq \int_0^t\lVert\underline\bu_2-\bu_1\rVert_{L^4(\Omega_{\eta_1(\sigma)})}\lVert\nabla\underline\bu_2\rVert_{L^2(\Omega_{\eta_1(\sigma)})}\lVert\bu_1-\underline\bu_2+\mathrm{Bog}_{\eta_1}(\Div\underline\bu_2)\rVert_{L^4(\Omega_{\eta_1(\sigma)})}\ds\\
&\lesssim \int_0^t\lVert\underline\bu_2-\bu_1\rVert_{L^4(\Omega_{\eta_1(\sigma)})}^2\lVert\nabla\underline\bu_2\rVert_{L^2(\Omega_{\eta_1(\sigma)})}\ds\\
&\lesssim\int_0^t\lVert\underline\bu_2-\bu_1\rVert_{L^2(\Omega_{\eta_1(\sigma)})}^{\frac{1}{2}}\lVert\nabla(\bu_1-\underline\bu_2)\rVert_{L^2(\Omega_{\eta_1(\sigma)})}^{\frac{3}{2}}\ds\\
&\leq \delta\int_0^t\lVert\nabla(\bu_1-\underline\bu_2)\rVert_{L^2(\Omega_{\eta_1(\sigma)})}^2\ds+C(\delta)\int_0^t\lVert\bu_1-\underline\bu_2\rVert_{L^2(\Omega_{\eta_1(\sigma)})}^2\ds,
\end{align*}
where we used the fact that $\underline\bu_2\in L^\infty(I, W^{1,2}(\Omega_{\eta_1}))$ and the interpolation inequality in $3D$:
$$\lVert\underline\bu_2-\bu_1\rVert_{L^4(\Omega_{\eta_1(\sigma)})}\lesssim \lVert\underline\bu_2-\bu_1\rVert_{L^2(\Omega_{\eta_1(\sigma)})}^{\frac{1}{4}}\lVert\nabla(\bu_1-\underline\bu_2)\rVert_{L^2(\Omega_{\eta_1(\sigma)})}^{\frac{3}{4}}. $$
Taking an integration by parts with respect to space, we estimate $R_{6,1}$ as
\begin{align*}
R_{6,1}&=-\int_0^t\int_{\Omega_{\eta_1(\sigma)})}\frac{1}{2}\nabla|\bu_1-\underline\bu_2|^2\cdot(\bu_1-\underline\bu_2+\mathrm{Bog}_{\eta_1}(\Div\underline\bu_2))\dx\ds\\
&=-\frac{1}{2}\int_0^t\int_{\partial\Omega_{\eta_1}}|\bu_1-\underline\bu_2|^2\bn\cdot(\partial_t\eta_1-\partial_t\eta_2)\bn_{\eta_1 }\circ\bfvarphi_{\eta_1}^{-1}\dd \mathcal H^2\ds,
\end{align*}
where we have used the fact that after integration by parts,
\begin{align*}
\divx(\bu_1-\underline\bu_2+\mathrm{Bog}_{\eta_1}(\Div\underline\bu_2))=\divx(-\underline\bu_2+\mathrm{Bog}_{\eta_1}(\Div\underline\bu_2))=0\qquad \text{in}~~~~~~~ \Omega_{\eta_1}.
\end{align*}
The estimate of $R_{6,2}$ is straightforward and we get
\begin{align*}
R_{6,2}&\leq \int_0^t\lVert\nabla(\bu_1-\underline\bu_2)\rVert_{L^2(\Omega_{\eta_1(\sigma)})}\lVert\underline\bu_2\rVert_{L^\infty(\Omega_{\eta_1(\sigma)})}\lVert\bu_1-\underline\bu_2+\mathrm{Bog}_{\eta_1}(\Div\underline\bu_2)\rVert_{L^2(\Omega_{\eta_1(\sigma)})}\ds\\
&\leq \delta\int_0^t\lVert\nabla(\bu_1-\underline\bu_2)\rVert_{L^2(\Omega_{\eta_1(\sigma)})}^2\ds+C(\delta)\int_0^t\lVert\underline \bu_2\rVert_{W^{2,2}(\Omega_{\eta_1(\sigma)})}^2\lVert\bu_1-\underline\bu_2\rVert_{L^2(\Omega_{\eta_1(\sigma)})}^2\ds.
\end{align*}
Adding $R_6$, $R_7$, $R_8$ and $R_{20}$ together, $R_{21}$, we arrive at
\begin{align}
\nonumber
&R_6+R_7+R_8+R_{20}+R_{21}\\
\nonumber&\leq \delta\int_0^t\lVert\nabla(\bu_1-\underline\bu_2)\rVert_{L^2(\Omega_{\eta_1(\sigma)})}^2\ds+C(\delta)\int_0^t\left(1+\lVert\underline \bu_2\rVert_{W^{2,2}(\Omega_{\eta_1(\sigma)})}^2\right)\lVert\bu_1-\underline\bu_2\rVert_{L^2(\Omega_{\eta_1(\sigma)})}^2\ds\\
\nonumber&\quad+\int_0^t\int_{\partial\Omega_{\eta_1}}\bfn\circ\bfvarphi_{\eta_1}^{-1}\left(\frac{1}{2}|\bu_1|^2\partial_t\eta_1-\frac{1}{2}|\bu_1-\underline\bu_2|^2(\partial_t\eta_1-\partial_t\eta_2)\right)\bfn_{\eta_1}\circ\bfvarphi_{\eta_1}^{-1}\dd \mathcal{H}^2\ds\\
\label{sumbound}&\quad+\int_0^t\int_{\partial\Omega_{\eta_1}}\bfn\circ\bfvarphi_{\eta_1}^{-1}\left(\frac{1}{2}|\underline\bu_2|^2\partial_t\eta_1-|\bu_1|^2\partial_t\eta_2\right)\bfn_{\eta_1}\circ\bfvarphi_{\eta_1}^{-1}\dd \mathcal{H}^2\ds,
\end{align}
where we use an integration by parts for the following term
$$\int_0^t\int_{\Omega_{\eta_1(\sigma)}}\bu_1\cdot\nabla\bu_1\bu_1\dx\ds=\int_0^t\int_{\partial\Omega_{\eta_1}}\frac{1}{2}\bfn|\bu_1|^2\partial_t\eta_1\bfn_{\eta_1}\circ\bfvarphi_{\eta_1}^{-1}\dd \mathcal{H}^2\ds.  $$
To deal with the boundary terms on the right side of \eqref{sumbound}, we notice that
\begin{equation}\label{boundaryrew}
\begin{aligned}
&\frac{1}{2}|\bu_1|^2\partial_t\eta_1-\frac{1}{2}|\bu_1-\underline\bu_2|^2(\partial_t\eta_1-\partial_t\eta_2)+\frac{1}{2}|\underline\bu_2|^2\partial_t\eta_1-|\bu_1|^2\partial_t\eta_2\\
&=-\frac{1}{2}|\bu_1-\underline\bu_2|^2\partial_t\eta_2-\underline\bu_2(\bu_1-\underline\bu_2)\partial_t\eta_2+(\bu_1-\underline\bu_2)\underline\bu_2(\partial_t\eta_1-\partial_t\eta_2)+\underline\bu_2^2(\partial_t\eta_1-\partial_t\eta_2)\\
&=: \mathrm{I+II+III+IV}.
\end{aligned}
\end{equation}
Recalling the definition of the map $\bfvarphi_{\eta_1}:\omega\to \partial\Omega_{\eta_1}$ and using the boundary conditions of $\bu_1$ and $\underline\bu_2$, we note that the boundary integration of $\mathrm{II}$ and $\mathrm{IV}$ in \eqref{boundaryrew} satisfy
\begin{equation*}
\begin{aligned}
&\int_0^t\int_{\partial\Omega_{\eta_1}}\bfn\circ\bfvarphi_{\eta_1}^{-1}\cdot (\mathrm{II}+\mathrm{IV})\bfn_{\eta_1}\circ\bfvarphi_{\eta_1}^{-1}\dd \mathcal{H}^2\ds\\
&=-\int_0^t\int_\omega\partial_t\eta_2\bfn(\partial_t\eta_1\bfn-\partial_2\eta_2\bfn)\partial_t\eta_2\dy\ds+\int_0^t\int_{\omega}|\partial_t\eta_2\bfn|^2(\partial_t\eta_1-\partial_t\eta_2)\dy\ds\\
&=0.
\end{aligned}
\end{equation*}
Then we estimate the remaining two boundary integrals in \eqref{sumbound} which are related to $\mathrm{I}$ and $\mathrm{III}$. We first have 
\begin{equation*}
\begin{aligned}
&\int_0^t\int_{\partial\Omega_{\eta_1}}\bfn\circ\bfvarphi_{\eta_1}^{-1}\cdot~ \mathrm{I}~\bfn_{\eta_1}\circ\bfvarphi_{\eta_1}^{-1}\dd \mathcal{H}^2\ds\\
&\leq 
\int_0^t\lVert\bu_1-\underline\bu_2\rVert_{L^2(\partial\Omega_{\eta_1})}^2\lVert\partial_t\eta_2\bfn_{\eta_1}\circ\bfvarphi_{\eta_1}^{-1}\rVert_{L^\infty(\partial\Omega_{\eta_1})}\ds\\
&\lesssim \int_0^t\lVert\bu_1-\underline\bu_2\rVert_{W^{\frac{1}{6}, 2}(\partial\Omega_{\eta_1})}^2\lVert\partial_t\eta_2\rVert_{W^{\frac{5}{3}, 2}(\omega)}\ds\\
&\lesssim \int_0^t\lVert\bu_1-\underline\bu_2\rVert_{W^{\frac{2}{3}, 2}(\Omega_{\eta_1})}^2\lVert\partial_t\eta_2\rVert_{W^{\frac{5}{3}, 2}(\omega)}\ds\\
&\lesssim \int_0^t\lVert\bu_1-\underline\bu_2\rVert_{L^2(\Omega_{\eta_1})}^{\frac{2}{3}}\lVert\nabla(\bu_1-\underline\bu_2)\rVert_{L^2(\Omega_{\eta_1})}^{\frac{4}{3}}\lVert\partial_t\eta_2\rVert_{W^{1,2}(\omega)}^{\frac{1}{3}}\lVert\partial_t\eta_2\rVert_{W^{2,2}(\omega)}^{\frac{2}{3}}\ds\\
&\leq \delta\int_0^t\lVert\nabla(\bu_1-\underline\bu_2)\rVert_{L^2(\Omega_{\eta_1})}^2\ds+C(\delta)\int_0^t\lVert\partial_t\eta_2\rVert_{W^{2,2}(\omega)}^2\lVert\bu_1-\underline\bu_2\rVert_{L^2(\Omega_{\eta_1})}^2\ds,
\end{aligned}
\end{equation*}
where we used the embedding $W^{\frac{5}{3},2}(\partial\Omega_{\eta_1})\hookrightarrow L^\infty(\partial\Omega_{\eta_1})$ and the interpolation inequalities:
\begin{equation*}
\begin{aligned}
\lVert\bu_1-\underline\bu_2\rVert_{W^{\frac{2}{3}, 2}(\Omega_{\eta_1})}&\lesssim \lVert \bu_1-\underline\bu_2\rVert_{L^2(\Omega_{\eta_1})}^{\frac{1}{3}}\lVert\bu_1-\underline\bu_2\rVert_{W^{1,2}(\Omega_{\eta_1})}^{\frac{2}{3}},\\
\lVert\partial_t\eta_2\rVert_{W^{\frac{5}{3}, 2}(\omega)}&\lesssim \lVert\partial_t\eta_2\rVert_{W^{1,2}(\omega)}^{\frac{1}{3}}\lVert\partial_t\eta_2\rVert_{W^{2,2}(\omega)}^{\frac{2}{3}}.
\end{aligned}
\end{equation*}
We also have the estimate for $\mathrm{III}$:
\begin{align*}
&\int_0^t\int_{\partial\Omega_{\eta_1}}\bfn\circ\bfvarphi_{\eta_1}^{-1}\cdot~ \mathrm{III}~\bfn_{\eta_1}\circ\bfvarphi_{\eta_1}^{-1}\dd \mathcal{H}^2\ds\\
&\leq \int_0^t\lVert\bu_1-\underline\bu_2\rVert_{L^2(\partial\Omega_{\eta_1})}\lVert\underline\bu_2\rVert_{L^\infty(\partial\Omega_{\eta_1})}\lVert(\partial_t\eta_1-\partial_t\eta_2)\bfn_{\eta_1}\circ\bfvarphi_{\eta_1}^{-1}\rVert_{L^2(\partial\Omega_{\eta_1})}\ds\\
&\lesssim\int_0^t\lVert\bu_1-\underline\bu_2\rVert_{W^{\frac{1}{2}, 2}(\partial\Omega_{\eta_1})}\lVert\underline\bu_2\rVert_{W^{\frac{3}{2}, 2}(\partial\Omega_{\eta_1})}\lVert\partial_t\eta_1-\partial_t\eta_2\rVert_{L^2(\omega)}\ds\\
&\lesssim \int_0^t\lVert\bu_1-\underline\bu_2\rVert_{W^{1,2}(\Omega_{\eta_1})}\lVert\underline\bu_2\rVert_{W^{2,2}(\Omega_{\eta_1})}\lVert\partial_t\eta_1-\partial_t\eta_2\rVert_{L^2(\omega)}\ds\\
&\leq \delta\int_0^t\lVert\bu_1-\underline\bu_2\rVert_{W^{1,2}(\Omega_{\eta_1})}^2\ds+C(\delta)\int_0^t\lVert\underline\bu_2\rVert_{W^{2,2}(\Omega_{\eta_1})}^2\lVert\partial_t\eta_1-\partial_t\eta_2\rVert_{L^2(\omega)}^2\ds.
\end{align*}
Putting all the estimates together and taking the supremum with respect to time on  both sides of \eqref{mainineq}, we obtain from Gr\"onwall's lemma that for every $T>0$, \eqref{final} holds.
\end{proof}

\begin{remark}{\rm
The estimate from Theorem \ref{thm:weakstrong} also applies when the forcing in the momentum equation is in divergence form, that is
 \begin{equation*}
\left\{\begin{aligned}
& \partial_t^2\eta -\partial_t\Dely \eta + \Dely^2\eta=g-\bn^\intercal\big(\bm{\tau}+\mathbf F\big)\circ\bm{\varphi}_\eta\bn_\eta
 \vert \mathrm{det}(\naby \bm{\varphi}_\eta)\vert
&\text{ for all }  (t,\by)\in I\times\omega,\\
 &\partial_t \bu  + (\mathbf{v}\cdot \nabx)\mathbf{v} 
 = 
 \Delx \bu -\nabx\pi+ \Div\bfF &\text{ for all }(t,\bx)\in I\times\Omega_\eta,\\
 &\Div \bu=0&\text{ for all }(t,\bx)\in I \times\Omega_\eta,
\end{aligned}\right.
\end{equation*}
for some $\mathbf F:I\times\Omega_\eta\rightarrow\R^{3\times 3}$.
In this case we obtain the estimate
\begin{align}
\nonumber
&\sup_{t\in I}\int_{\Omega_{\eta_1(t)}}|\bu_1(t)-\underline{\bu}_2(t)|^2\dx+\sup_{t\in I}\int_\omega\big(|\partial_t(\eta_1-\eta_2)(t)|^2
+
|\Dely(\eta_1-\eta_2)(t)|^2\big)\dy
\\
\nonumber&\qquad\qquad\qquad\qquad+\int_I
\int_{\Omega_{\eta_1(\sigma)}}|\nabla(\bu_1-\underline{\bu}_2)|^2\dx\dt 
+
\int_I\int_\omega|\partial_t\nabla_{\by}(\eta_1-\eta_2)|^2\dy\dt\\
\nonumber&\lesssim \int_{\Omega_{\eta_1(0)}}|\bu_1(0)-\underline{\bu}_2(0)|^2\dx+\int_\omega|\partial_t(\eta_1-\eta_2)(0)|^2\dy+\int_\omega|\Dely(\eta_1-\eta_2)(0)|^2\dy\\
&\qquad\qquad\qquad\qquad +\int_I\int_{\Omega_{\eta_1}}| \mathbf F_1-\underline {\mathbf F}_2|^2\dx\dt+\int_I\int_\omega| g_1-g_2|^2\dy\dt,\nonumber
\end{align}
where $\underline{\mathbf F}_2:=\bfF_2\circ \bfPsi_{\eta_2-\eta_1}$.}
\end{remark}

\section{The main result}\label{summary}
In the following, we formulate the desired conditional regularity and uniqueness result for \eqref{1}--\eqref{interfaceCond}, which implies Theorem~\ref{thm:mainsimple} and Corollary~\ref{cor:mainsimple}.
Its proof follows by combining
Theorems \ref{thm:fluidStructureWithoutFK},  \ref{prop2} and \ref{thm:weakstrong}.
\begin{theorem}\label{thm:main}
Let $T>0$ be given. Suppose that the dataset
$(\bff, g, \eta_0, \eta_*, \bu_0)$
satisfies \eqref{dataset} and
\eqref{datasetImproved}.
Let $(\bu,\eta)$ be a weak solution of \eqref{1}--\eqref{interfaceCond} in the sense of Definition \ref{def:weakSolution}. Suppose that we
have
\begin{align}\label{eq:regu'}
\bu&\in L^r(I;L^s(\Omega_\eta)),\quad \tfrac{2}{r}+\tfrac{3}{s}\leq1,\\
\eta&\in L^\infty(I;C^{1}(\omega)).\label{eq:regeta''}
\end{align}
Then $(\bu,\eta)$ is a strong solution in the sense of Definition \ref{def:strongSolution} on $I = (0, t)$, where $t < T$ only in case $\Omega_{\eta(s)}$ approaches a self-intersection when $s\rightarrow t$ or it degenerates\footnote{Self-intersection and degeneracy are excluded if $\sup_t\|\eta\|_{W^{1,\infty}_{y}}<L$, cf. \eqref{eq:boundary1} and \eqref{eq:1705}.} (namely, if $\displaystyle\lim_{s\rightarrow t}(\partial_1\bfvarphi_\eta\times \partial_2\bfvarphi_\eta)(s,\by)=0$ or $\displaystyle\lim_{s\rightarrow t}\bfn(\by)\cdot\bfn_{\eta(s)}(\by)=0$ for some $\by\in\omega$). 
Moreover, $(\bu,\eta)$  is unique in the class of weak solutions with deformation in $L^\infty(I,C^{0,1}(\omega))$.
\end{theorem}
\begin{proof}
Consider first the problem on the interval $(0,T^\ast)$ in which the strong solution exists by Theorem \ref{thm:fluidStructureWithoutFK}. On account of \eqref{eq:regeta''} Theorem \ref{thm:weakstrong} applies and thus both solutions coincide. Hence the strong solution satisfies \eqref{eq:regu'} (with a constant independent of $T^\ast$). Thus we obtain the estimate
from Theorem \ref{prop2}. Now we can apply Theorem \ref{thm:fluidStructureWithoutFK} to obtain a strong solution on the interval $(T^\ast,2T^\ast)$ with initial data $\bfu(T^\ast),\eta(T^\ast),\partial_t\eta(T^\ast)$.
This procedure can now be repeated until the moving boundary approaches a self-intersection or degenerates (that is $(\partial_1\bfvarphi_\eta\times \partial_2\bfvarphi_\eta)(T,\by)=0$ for some $\by\in\omega$).
\end{proof}

\section*{Acknowledgments}
S. Schwarzacher and P. Su are partially supported by the ERC-CZ Grant CONTACT LL2105 funded by the Ministry of Education, Youth and Sport of the Czech Republic and the University Centre UNCE/SCI/023 of Charles University. S. Schwarzacher also acknowledges the support of the VR
Grant 2022-03862 of the Swedish Science Foundation.

\section*{Compliance with Ethical Standards}
\smallskip
\par\noindent
{\bf Conflict of Interest}. The authors declare that they have no conflict of interest.

\smallskip
\par\noindent
{\bf Data Availability}. Data sharing is not applicable to this article as no datasets were generated or analysed during the current study.

\end{document}